\documentclass[11px_0t,a4paper,reqno,final]{amsart}

\usepackage[english]{babel}
\usepackage[utf8]{inputenc} % Accents, eg italian ones
\usepackage[T1]{fontenc} % Special characters from non-English alphabets

\usepackage{amsmath, amssymb, amsthm, bm} % Math packages
\usepackage{mathabx} % Includes eg widecheck
\usepackage{bbm} % Allows boldface for numbers (e.g. 1)

\usepackage{graphicx}
\usepackage{color}
\usepackage{mathtools}

\usepackage{xifthen}

\usepackage{enumerate} % Allows for [(i)] in enumerate
\usepackage[colorlinks = true, citecolor = black, linkcolor=red, final]{hyperref}

\usepackage[inline]{showlabels}
\usepackage[capitalise]{cleveref}

\title[\MakeUppercase{\Squashed} Laakso spaces]{\MakeUppercase{\Squashed} Laakso spaces, pure PI unrectifiability and differentiability of Lipschitz functions}

\author{David Bate}
\author{Pietro Wald}
\date{\today}

\newcommand*{\shortcutsLDLR}[1][]{shortcuts \ifthenelse{\isempty{#1}}{}{#1-}uniformly}
\newcommand*{\shortcuts}{shortcuts}
\newcommand*{\Shortcuts}{Shortcuts}
\newcommand*{\jumpset}{shortcut}
\newcommand*{\jumpsets}{shortcuts}
\newcommand*{\squashed}{shortcut}
\newcommand*{\Squashed}{Shortcut}

\newcommand{\itemref}[1]{(\ref{#1})}

\newcommand*{\Bad}{\mathrm{Bad}}
\newcommand*{\Badf}{\operatorname{Bad}}

\newcommand*{\var}{\operatorname{var}}
\newcommand*{\gap}{\operatorname{gap}}

\newcommand{\G}{\mathbb{G}}
\newcommand{\Heis}{\mathbb{H}}
\newcommand{\R}{\mathbb{R}}

\newcommand{\N}{\mathbb{N}}
\newcommand{\Q}{\mathbb{Q}}

\newcommand{\Sp}{\mathbb{S}}
\newcommand{\Prob}{\mathbb{P}}

\newcommand{\sch}{\mathcal S}
\newcommand{\sqL}{\mathcal L}

\newcommand{\calI}{\mathcal{I}}
\newcommand{\calJ}{\mathcal{J}}

\newcommand{\calL}{\mathcal{L}}

\newcommand{\calE}{\mathcal{E}}
\newcommand{\calN}{\mathcal{N}}

\newcommand{\calF}{\mathcal{F}}

\newcommand{\calQ}{\mathcal{Q}}

\newcommand{\calU}{\mathcal{U}}

\newcommand{\spt}{\operatorname{spt}}

\newcommand{\calP}{\mathcal{P}}

\newcommand{\calA}{\mathcal{A}}

\newcommand{\diam}{\operatorname{diam}}

\newcommand{\Lip}{\operatorname{Lip}}

\newcommand{\LIP}{\operatorname{LIP}}

\newcommand{\Int}{\operatorname{Int}}

\newcommand*{\dd}{\mathrm{d}}
\newcommand{\mres}{\mathbin{\vrule height 1.6ex depth 0pt width
0.13ex\vrule height 0.13ex depth 0pt width 1.3ex}} % Restriction of measure

\newcommand\Haus{\mathcal{H}}

\def\Xint#1{\mathchoice 
	{\XXint\displaystyle\textstyle{#1}}%
	{\XXint\textstyle\scriptstyle{#1}}%
	{\XXint\scriptstyle\scriptscriptstyle{#1}}%
	{\XXint\scriptscriptstyle\scriptscriptstyle{#1}}%
	\!\int}
\def\XXint#1#2#3{{\setbox0=\hbox{$#1{#2#3}{\int}$}
	\vcenter{\hbox{$#2#3$}}\kern-.5\wd0}}

\def\dashint{\Xint-}

\def\Barint_#1{\mathchoice
          {\mathop{\vrule width 6pt height 3 pt depth -2.5pt
                  \kern -8pt \intop}\nolimits_{#1}}%
          {\mathop{\vrule width 5pt height 3 pt depth -2.6pt
                  \kern -6pt \intop}\nolimits_{#1}}%
          {\mathop{\vrule width 5pt height 3 pt depth -2.6pt
                  \kern -6pt \intop}\nolimits_{#1}}%
          {\mathop{\vrule width 5pt height 3 pt depth -2.6pt
                  \kern -6pt \intop}\nolimits_{#1}}}

\DeclarePairedDelimiter\abs{\lvert}{\rvert}

\DeclarePairedDelimiter\ceil{\lceil}{\rceil}
\DeclarePairedDelimiter\floor{\lfloor}{\rfloor}

\numberwithin{equation}{section}

\theoremstyle{plain}
\newtheorem{thm}[equation]{Theorem}

\newtheorem{lemma}[equation]{Lemma}
\newtheorem{corol}[equation]{Corollary}

\newtheorem{prop}[equation]{Proposition}

\theoremstyle{definition}

\newtheorem{defn}[equation]{Definition}

\newtheorem{notation}[equation]{Notation}

\theoremstyle{remark}
\newtheorem{rmk}[equation]{Remark}

\begin{document}

\begin{abstract}
We construct a family of purely PI unrectifiable Lipschitz differentiability spaces and investigate the possible of Banach spaces targets for which Lipschitz differentiability holds.
We provide a general investigation into the geometry of \emph{shortcut} metric spaces and characterise when such spaces are PI rectifiable, and when they are $Y$-LDS, for a given $Y$.
The family of spaces arises as an example of our characterisations.
Indeed, we show that Laakso spaces satisfy the required hypotheses.
\end{abstract}

\maketitle
\tableofcontents
\section{Introduction}\label{sec:intro}
Recent years have seen much activity in the study of first order calculus on metric measure spaces $(X,d,\mu)$ satisfying a Poincaré inequality with a doubling measure (a \emph{PI space}, see \cref{defn:PI_space}).
First introduced by \cite{Heinonen_Koskela_QC_maps_in_PI_spaces}, these spaces
allow a general theory of first order Sobolev spaces to be developed,
see the monographs \cite{HKST_Sobolev_on_metric_measure_spaces,Bjorn_Bjorn_Nonlinear_potential_theory} for more information.

A cornerstone of the theory of PI spaces is the seminal work of Cheeger \cite{cheeger1999} that provides a generalisation of Rademacher's differentiation theorem to PI spaces.
That is, up to a countable decomposition of $X$, there exists a Lipschitz function $\varphi\colon X\to \R^n$ such that the following holds.
For any Lipschitz $f\colon X\to \R$ and $\mu$-a.e.\ $x\in X$, there exists a unique linear $D\colon \R^n\to \R$ such that
\[f(y)-f(x) = D(\varphi(y)-\varphi(x)) + o(d(y,x)).\]
Cheeger's theorem was later strengthened by
Cheeger and Kleiner \cite{cheeger_kleiner_PI_RNP_LDS},
who proved that this differentiability theory holds for any Lipschitz $f\colon X\to Y$, for any Banach space $Y$ with the Radon-Nikodym property.
These results have found numerous applications in biLipschitz non-embeddability problems
\cite[Section 14]{cheeger1999}, \cite[Corollary 1.7 and 1.8]{cheeger_kleiner_PI_RNP_LDS}, \cite{lee2006lp,MR2892612}.

Cheeger's work stimulated many new lines of research.
Notably, Keith \cite{keith} shows that the conclusion of Cheeger's theorem continues to hold if one replaces the local Poincaré inequality with an infinitesimal condition regarding pointwise Lipschitz constants;
this fits more closely with the infinitesimal conclusion of the theorem.
Bate \cite{bate_str_of_measures_2015} gives a study of \emph{Lipschitz differentiability spaces} (LDS), metric measure spaces that satisfy the conclusion of Cheeger's theorem, without any further assumption such as a Poincaré inequality.
In particular, several characterisations of LDS are given in terms of a rich structure of rectifiable curves known as an Alberti representation of the space.
Since these results, it has been an important open question to determine to what extent the PI hypothesis is necessary for the conclusion of Cheeger's theorem to hold, see \cite[Introduction]{Heinonen_non_smooth_calculus} and \cite[Question 1.17]{Cheeger_Kleiner_Schioppa_infinitesimal_LDS}.

The work of Bate-Li \cite{bate_li_RNP_LDS} and Eriksson-Bique \cite{erikssonbique2019_PI_rect} give a partial answer to this.
They show that the \emph{RNP} differentiability theory of Cheeger-Kleiner characterises PI spaces.
More precisely, it characterises \emph{PI rectifiable} spaces, metric measure spaces that can be covered by a countable union of measurable subsets of PI spaces (see \cref{defn:PI_rect}), for which porous sets have measure zero (see \cref{defn:porous}).
The question for LDS was answered, negatively, with a remarkable, and unfortunately unpublished, counterexample of Schioppa \cite{schioppa2016_PI_unrectifiable}:
There is an LDS $(\sch,d,\mu)$ which is \emph{purely PI unrectifiable}.
That is, any biLipschitz image of a PI space into $\sch$ has $\mu$-measure zero.
(see \cref{defn:PI_rect}).
The space $\sch$ has analytic dimension $3$ and Nagata and Lipschitz dimension at least $3$.
Moreover, any Lipschitz function $f\colon \sch\to\ell_2$ is differentiable almost everywhere.

Schioppa's example, and its construction, immediately open many new questions:
Can the dimensions of such an example be reduced below 3?
Is $\ell_2$ the only target for which Lipschitz functions are differentiable?
More generally, for Banach spaces $Y,Z$, under what conditions does $Y$ valued differentiability imply $Z$ valued differentiability? For what $Y$ there are purely PI unrectifiable spaces for which $Y$ valued differentiability holds?
Schioppa's construction produces a \emph{single} space with the desired properties, from which it is not possible to extract underlying reasons why a purely PI unrectifiable space is an LDS. Indeed, Schioppa's example remains the \emph{only} example.
Is there a systematic approach to constructing such a space?

In this article we answer these questions and more.
To state our main results we introduce the following terminology.
For a Banach space $Y$, a metric measure space $(X,d,\mu)$ is a \emph{$Y$-LDS} if every Lipschitz function $f\colon X\to Y$ is differentiable $\mu$-almost everywhere.
See \cref{defn:LDS} for a precise definition.
\begin{thm}\label{thm:main_thm_Laakso}
Let $2<s<\infty$.
There is a compact $s$-Ahlfors-David regular metric measure space $\sqL=(\sqL,d,\mu)$ with the following properties:
\begin{enumerate}[(i)]
\item
$\sqL$ is purely PI unrectifiable;
\item
if $2\leq q<s$ and $Y$ is a $q$-uniformly convex Banach space,
then
$\sqL$ is a $Y$-LDS;
\item
if $s<q<\infty$,
no positive-measure $\mu$-measurable subset of $\sqL$ is an $\ell_q$-LDS;
\item
if $Y$ is a non-superreflexive Banach space, then
no positive-measure $\mu$-measurable subset of $\sqL$ is a $Y$-LDS;
\item
$\sqL$ has analytic, Nagata, and Lipschitz dimension $1$; this is minimal for
purely PI unrectifiable LDS and answers \cite[(Q1)]{schioppa2016_PI_unrectifiable};
\item
$\sqL$ biLipschitz embeds in $L^1([0,1])$, but no positive-measure
$\mu$-measurable subset admits a David-Semmes regular (in particular, biLipschitz) embedding in a Banach space with RNP.
\end{enumerate}
\end{thm}
Recall that a Banach space $Y$ is non-superreflexive if and only if it does not have a uniformly convex renorming
(\cite[Chapter 10 and 11]{Pisier_martingales_in_Banach_2016}).
Moreover, every uniformly convex space has $q$-uniformly convex renorming for some $2\leq q<\infty$ (\cite[Theorem 10.2]{Pisier_martingales_in_Banach_2016}).
Hence, the spaces $\sqL$ as in \cref{thm:main_thm_Laakso} are $Y$-LDS provided $Y$ has a renorming which is `sufficiently' uniformly convex.
Furthermore, the existence of a uniformly convexity renorming is a necessary condition for $\sqL$ to be $Y$-LDS.
The precise threshold for the convexity modulus is given by the Hausdorff dimension of $\sqL$.
Since $\ell_1$ is non-superreflexive,
\cref{thm:main_thm_Laakso} provides partial information towards \cite[(Q2)]{schioppa2016_PI_unrectifiable}.

The proof of \cref{thm:main_thm_Laakso} demonstrates, in a systematic way, \emph{how} a purely PI unrectifiable LDS can arise.
Indeed, our first contribution consists of the following, alternative point of view on the construction of Schioppa.
Being purely PI unrectifiable, one could believe that $(\sch,d,\mu)$ looks `far' from a PI space.
However, this is not exactly the case.
Indeed, our work began with the observation that, when equipped with the length distance $\widehat{d}$ generated by $d$,
$(\sch,\widehat{d},\mu)$ is in fact a PI space.
Moreover, the geometry of $(\sch,\widehat{d},\mu)$ allows one to find, for every RNP Banach space $Y$, every
Lipschitz $f\colon (X,\widehat{d})\to Y$,
and $\mu$-a.e.\ $x\in X$,
a sequence $y_j\rightarrow x$, such that
\begin{equation}\label{eq:intro_rate_collapse}
f(y_j)-f(x)=o(\widehat{d}(y_j,x)).
\end{equation}
Further, using the ideas behind \cite[Theorems 4.30 and 5.8]{schioppa2016_PI_unrectifiable}, one may obtain a \emph{quantitative} version of \cref{eq:intro_rate_collapse},
for $\ell_2$-valued Lipschitz maps.
One may then distort $\widehat{d}$, contracting the distance between all of the $(y_j,x)$, in order to make the space purely PI unrectifiable.
The fact that $(\sch,\widehat{d},\mu)$ is an $\ell_2$-LDS, together with the quantitative \cref{eq:intro_rate_collapse}, yield a direct proof that $\sch$, equipped with the contracted distance, is an $\ell_2$-LDS.
This contraction recovers the distance and result of Schioppa.

The proof of \cref{thm:main_thm_Laakso} builds upon our observation in two ways:
We develop a general theory of contracting a metric space first introduced by Le Donne, Rajala and Li \cite{ledonne_li_rajala_shortcuts_heisenberg};
We show that it is possible to perform the construction on a more flexible and familiar PI space, the Laakso space \cite{laakso2000_ADR_PI}.

These results illustrate how Lipschitz differentiability can arise in purely PI unrectifiable spaces.
Indeed, in the latter, the available curves cannot by themselves control Lipschitz oscillations and a new mechanism is necessary.
\Cref{thm:main_thm_Laakso} shows that, for pairs of points where curves lack, the geometry of the underlying metric space and target Banach space
can interact in such a way to force wild oscillation to concentrate on a null set.
\subsection{Metric spaces with \shortcuts}

Le Donne, Rajala and Li \cite{ledonne_li_rajala_shortcuts_heisenberg} introduce a method to construct Lipschitz images of an Ahlfors-David regular metric space by permitting \emph{shortcuts} to be taken between carefully chosen points.
This method answers a conjecture of Semmes by showing that the Heisenberg group is not \emph{minimal in looking down}.
Roughly speaking, by identifying pairs of points throughout the space that witness an excess of the triangle inequality, one can shorten distances between these and near by points, without affecting distances far away.
Slightly more precisely, for some $0<\delta <1$ and each $i\in\mathbb N$, one can find $\delta^i$ separated nets formed of shortcuts, and we contract the distances between shortcuts by a factor of $\eta_i\to 0$.
It is shown that the new metric space is not biLipschitz equivalent to the original space on any set of positive measure.
Moreover, it is shown that the Heisenberg group and snowflake metric spaces possess shortcuts.

In order to construct the metric space in \cref{thm:main_thm_Laakso}, we first generalise this construction by allowing the shortcuts to appear at an arbitrarily rate $\delta_i\to 0$ (for a precise comparison, see \cref{sec:comparison_LDLR}).
This generalisation is necessary so that Laakso spaces possess shortcuts, see \cref{defn:shortcuts_laakso}.
We then proceed to give a general study of the resulting \emph{shortcut metric space}.
In particular, in \cref{sec:pi_rect} we give conditions on $\eta_i$ that characterise when a shortcut metric space is PI rectifiable and purely PI unrectifiable.
By ensuring that the latter condition is satisfied, the first item of \cref{thm:main_thm_Laakso} is automatically satisfied.

In \cref{subsec:LDS_squashed} we characterise at which points the derivative of a Lipschitz function on $(X,d)$ is preserved under the shortcut construction.
As a result, when the non-contracted space is a $Y$-LDS, we find a condition characterising when the contracted space is also a $Y$-LDS.
The latter condition is not true in general.
For instance, in \cref{prop:non_LDS_dim_less_2}, we give an example of an LDS, in fact, any $s$-ADR Laakso space for $1<s<2$ works, for which any (non-trivial) shortcut metric space is not an LDS. 
More results of this type are found in the same section, \cref{section:nowhere_diff_maps}.
Carnot groups (of step $\geq 2$) have shortcuts (see \cref{subsec:Carnot_groups}), but
we do not know if they satisfy this differentiability condition.
This question is related to the `Vertical vs Horizontal' Poincar\'e inequalities in Carnot groups,
see \cite{Austin_Naor_Tessera_sharp_quantitative_non_emb_Heis,Seung_Yeon_Ryoo_quantitative_non_embeddability_groups},
but does not seem to be directly implied by such results.
However, it is satisfied when a suitable quantitative differentiation theory holds in the contracted or non-contracted space, see \cref{corol:quant_diff_implies_LDS}.
We are able to prove such an estimate in $s$-ADR Laakso spaces with $s>2$, leading to the second point of \cref{thm:main_thm_Laakso}.

\subsection{Quantitative differentiation on Laakso spaces}
Laakso spaces were introduced in \cite{laakso2000_ADR_PI} as examples of Ahlfors-David regular $1$-PI spaces whose Hausdorff dimension varies continuously in $(1,\infty)$. 
Consequently, Laakso spaces are also LDS.
It is natural to ask whether a quantitative differentiation theory holds in Laakso spaces.
One conjecture is that a generalisation of the celebrated Dorronsoro theorem \cite{dorronsoro} holds in Laakso spaces, analogously to the case in Heisenberg groups \cite{MR4171381}.
We do not know the answer to this question.
Instead we prove a weaker form of quantitative differentiation holds in $s$-ADR Laakso spaces with $s>2$, see \cref{thm:collapse_at_gates}.
This suffices to apply \cref{corol:quant_diff_implies_LDS} discussed above.

The first step of the poof of \cref{thm:collapse_at_gates} is to prove a harmonic approximation of Lipschitz functions on LDS with analytic dimension 1 (that is, the map $\varphi$ takes values in $\R$), see \cref{prop:harmonic_approximation_general}.
When applied to Laakso spaces, the symmetry of the space is inherited by the harmonic approximation, see \cref{lemma:harmonic_approximation_Laakso}.
This additional symmetry allows us to control the oscillation of a Lipschitz function by the oscillation of its harmonic approximation, and hence obtain \cref{thm:collapse_at_gates}.

\subsection{\Squashed\ Laakso spaces}

By combining the work discussed above, we show that a suitable shortcut space formed from an $s$-ADR Laakso space, $s>2$, satisfies the first two points of \cref{thm:main_thm_Laakso}.
We call the resulting metric space a \emph{\squashed\ Laakso space}, see \cref{sec:shortcut_laakso}.
The remaining points of \cref{thm:main_thm_Laakso} are proven via a direct analysis of \squashed\ Laakso spaces in \cref{section:nowhere_diff_maps,sec:LIP_dim}.
Point (iv) of \cref{thm:main_thm_Laakso} requires a new characterisation of non-superreflexive Banach spaces in terms of Laakso spaces, which may be of independent interest, see \cref{lemma:non_diff_map_non_superreflexive_building_blocks}.
The proof of this characterisation is given in \cref{subsec:non_superreflexive_bblocks}.

\subsection{Acknowledgements}
We would like to thank Sylvester Eriksson-Bique for conversations on \cite{schioppa2016_PI_unrectifiable}
and on obstructions to the construction of a similar space out of lower-dimensional cube complexes.
In particular, the idea of using orthogonality in $\R^2$ to control the Lipschitz constant in
\cref{lemma:induction_step_bad_fun_in_R}
is due to Eriksson-Bique. \par
D.B.\ was supported by the European Union's Horizon 2020 research and innovation programme (Grant agreement No.\ 948021).
P.W.\ was supported by the Warwick Mathematics Institute Centre for Doctoral Training, and gratefully acknowledges funding
from the University of Warwick and the UK Engineering and Physical Sciences Research Council (Grant number: EP/W524645/1).

\section{Preliminaries}\label{sec:prelim}
\subsection{General notation}\label{subsec:general_notation}
We set $\inf\varnothing:=\infty$ and $\sup\varnothing:=0$. \par
The set positive integers is denoted with $\N=\{1, 2,\dots\}$,
while $\N_0:=\N\cup\{0\}$.
For $n\in\N$, we set $[n]:=\{1,\dots,n\}$. \par
For two non-negative functions $f,g$ on a set, we write $f\lesssim g$
or $g\gtrsim f$ if there is $C>0$ such that $f\leq C g$,
and $f\sim g$ if $f\lesssim g\lesssim f$. \par
For a metric space $X$, we denote its distance as either $d$ or $d_X$.
Similarly, when a quantity depends on the distance of $X$ and we wish to emphasise such dependence,
we add the subscript $X$ (or $d$).
For instance, we may write $\diam_X A$ (or $\diam_d A$) for the diameter of a set $A\subseteq X$,
$\diam A:=\sup\{d(x,y)\colon x,y\in A\}$. \par
For sets $A,B\subseteq X$ and $x\in X$, we set \[d(A,B):=\inf\{d(x,y)\colon x\in A,y\in B\}\]
and $d(A,x):=d(A,\{x\})$.
For $A\subseteq X$ and $r>0$, we define \[B(A,r):=\{x\in X\colon d(A,x)\leq r\},\]
$B(x,r):=B(\{x\},r)$, and
\[U(A,r):=\{x\in X\colon d(A,x)<r\},\]
$U(x,r):=U(\{x\},r)$. \par
We call \emph{measure} on a set $\Omega$
any map $\mu\colon \calP(\Omega):=\{A\colon A\subseteq \Omega\}\to[0,\infty]$
which is
monotone, countably subadditive, and vanishes at $\varnothing$.
The restriction of $\mu$ to a set $E\subseteq \Omega$ is the measure defined as
$\mu\mres E(A):=\mu(E\cap A)$ for $A\subseteq\Omega$.
If $Z$ is a set and $f\colon E\subseteq\Omega\to Z$ a map,
we set $f_{\#}\mu(A):=\mu(f^{-1}(A))$ for $A\subseteq Z$.
A set $E\subseteq \Omega$ is $\mu$-measurable if
\begin{equation*}
\mu(A\cap E)+\mu(A\setminus E)=\mu(A)
\end{equation*}
for all $A\subseteq\Omega$.
The collection of $\mu$-measurable sets forms a $\sigma$-algebra
on which $\mu$ is countably additive; see \cite[Theorem 2.1.3]{Federer_GMT_1969}.
\par
We call \emph{metric measure space} a triplet $(X,d,\mu)$, where $(X,d)$ is a separable metric
space and $\mu$ a locally finite Borel regular measure satisfying
\begin{equation}\label{eq:inner_reg}
	\mu(V)=\sup\{\mu(K)\colon K \text{ compact and } K\subseteq V\}
\end{equation}
for all open sets $V\subseteq X$.
From \cite[Theorem 2.2.2, 2.2.3]{Federer_GMT_1969} and (the argument of) \cite[2.2.5]{Federer_GMT_1969},
we deduce that $\mu$ is a Radon measure, i.e.\
\cref{eq:inner_reg} holds for any $\mu$-measurable set $V$
and, for any set $A\subseteq X$, we have
\begin{equation*}
	\mu(A)=\inf\{\mu(V)\colon V\text{ open and }A\subseteq V\}.
\end{equation*}
\par
Recall that if $E\subseteq X$ is a non-empty set and $B\subseteq X$ Borel, then $B\cap E$ is a Borel set in $(E,d\vert_{E\times E})$;
see e.g.\ \cite[Lemma 3.3.4]{HKST_Sobolev_on_metric_measure_spaces}.
From this (and the above), it is not difficult to see that, for any $\mu$-measurable $E\subseteq X$,
the triplet $(E,d\vert_{E\times E},\mu\vert_{\calP(E)})$ is also a metric measure space.
In such cases, we will write $(E,d,\mu)$ in place of $(E,d\vert_{E\times E},\mu\vert_{\calP(E)})$. \par
If $\mu$ is a Borel measure on a separable metric space $X$, we set
\[\spt\mu:=X\setminus\bigcup\{V\colon V\text{ is open and }\mu(V)=0\}.\]
\subsection{Lipschitz differentiability spaces}
Let $X,Y$ be metric spaces.
A function $f\colon X\to Y$ is \emph{Lipschitz} if there is $L\geq 0$
such that $d_Y(f(x),f(y))\leq Ld(x,y)$ for $x,y\in X$.
The least such $L$ is denoted by $\LIP(f)$;
if $f$ is not Lipschitz, we set $\LIP(f):=\infty$.
We define \[\LIP(X;Y):=\{f\colon X\to Y\colon \LIP(f)<\infty\}\]
and $\LIP(X):=\LIP(X;\R)$.

For metric space $X,Y$ and a function $f\colon X\to Y$,
we define its \emph{pointwise Lipschitz constant at $x$} as
\begin{equation*}
\Lip(f;x):=\limsup_{y\rightarrow x}\frac{d_Y(f(x),f(y))}{d(x,y)}
=\limsup_{r\rightarrow 0}\sup_{y\in B(x,r)}\frac{d_Y(f(x),f(y))}{r}
\end{equation*}
if $x\in X$ is a limit point and $\Lip(f;x):=0$ otherwise.
If $f$ is continuous, $x\mapsto\Lip(f;x)$ is Borel measurable.
\begin{defn}\label{defn:LDS}
Let $Y$ be a Banach space.
A \emph{$Y$-Lipschitz differentiability space} ($Y$-LDS)
(or Lipschitz differentiability space, LDS, if $Y=\R$)
is a metric measure space $(X,d,\mu)$
satisfying the following condition.
There is a countable collection of \emph{(Cheeger) charts}
$(U_i,\varphi_i)$, where $U_i\subseteq X$ is $\mu$-measurable
and $\varphi_i\colon X\to\R^{n_i}$ Lipschitz
($n_i\in\N_0$), such that
$\mu(X\setminus\bigcup_i U_i)=0$, and
for every Lipschitz $f\colon X\to Y$, $i$,
and $\mu$-a.e.\ $x\in U_i$,
there is a unique linear map $D\colon\R^{n_i}\to Y$ such that
\begin{equation}\label{eq:Cheeger_diff}
\Lip(f-D\circ\varphi_i;x)=0.
\end{equation}
We call $d_xf:=d^{\varphi_i}_xf:=D$ \emph{(Cheeger) differential (or $\varphi_i$-differential)} of $f$ at $x$
and any collection $\{(U_i,\varphi_i)\}_i$ as above \emph{(Cheeger) atlas}.
\end{defn}
If $Y$ is not the zero-dimensional Banach space, then any $Y$-LDS is an LDS.
\begin{rmk}\label{rmk:isolated_points_zero_dim_charts}
Note that we allow $n_i=0$, i.e.\ charts $(U_i,\varphi_i\colon X\to\R^0=\{0\})=(U_i,0)$.
In this case, a function $f\colon X\to Y$ has a $\varphi_i$-differential at $x\in U_i$
if and only if $\Lip(f;x)=0$.
By definition of $\Lip$, the latter condition is
satisfied for every isolated point $x\in X$
and so $(\{x\in X\colon x\text{ is isolated}\},0)$ is a chart in any LDS $(X,d,\mu)$.
The converse is also true, see \cref{prop:non_RNP_LDS} (and \cite[Remark 4.10]{bate_str_of_measures_2015}).
\end{rmk}
We mostly need elementary properties of LDS and $Y$-LDS.
\begin{lemma}\label{lemma:indep_and_Banach_indep}
Let $X$ be a metric space, $x\in X$, $n\in\N$, and let $\varphi\colon X\to\R^n$ be Lipschitz.
The following are equivalent:
\begin{itemize}
\item $\Lip(\langle v,\varphi\rangle;x)>0$ for $v\in\R^n\setminus\{0\}$;
\item there is a constant $c>0$ such that for every Banach space $Y$ and linear $T\colon\R^n\to Y$ it holds
\begin{equation*}
\Lip(T\circ\varphi;x)\geq c\|T\|;
\end{equation*}
\item for every Lipschitz function $f\colon X\to\R$ there is at most one $D\in(\R^n)^*$ such that
$\Lip(f-D\circ\varphi;x)=0$;
\item for every Banach space $Y$ and function $f\colon X\to Y$ there is at most one linear map $D\colon\R^n\to Y$ such that
$\Lip(f-D\circ\varphi;x)=0$.
\end{itemize}
\end{lemma}
\begin{proof}
The main implication is that the first point implies the second;
we begin by establishing the others.
The second point implies the first considering $Y=(\R^n)^*$.
It also yields the fourth by taking $T=D_1-D_2$ for differentials $D_1,D_2$ of $f$ at $x$ and using triangle inequality (for $\Lip$).
The fourth point implies the third, which implies the first by 
taking $f=\langle v,\varphi\rangle$ for $v\in\R^n\setminus\{0\}$
and observing that, by uniqueness, $\Lip(f;x)=\Lip(f-\langle 0,\varphi\rangle;x)>0$. \par
It remains to prove that the first point implies the second.
Arguing as in \cite[Lemma 2.1]{batespeight2013}, we find, for $1\leq i\leq n$,
sequences $(x_j^i)_j\subseteq X\setminus\{x\}$, $x_j^i\rightarrow x$, such that
\begin{equation*}
	\lim_{j\rightarrow \infty}\frac{\varphi(x_j^i)-\varphi(x)}{d(x_j^i,x)}=:b_i
\end{equation*}
exists for $1\leq i\leq n$ and $b_1,\dots,b_n$ form a basis of $\R^n$.
Let $b_1^*,\dots, b_n^*\in(\R^n)^*$ denote the dual basis of $b_1,\dots, b_n$ and set $p(v):=\sum_{i=1}^n|b_i^*(v)|$ for $v\in\R^n$.
It is clear that $p(\cdot)$ is a norm on $\R^n$, and so there is $C\geq 1$ such that
$C^{-1}\|\cdot\|\leq p(\cdot)\leq C\|\cdot\|$.
Let $Y$ be a Banach space, $T\colon\R^n\to Y$ linear, and let $v\in\R^n$, $\|v\|=1$, be such that
$\|T(v)\|_Y=\|T\|$.
Then,
\begin{equation*}
	\|T\|\leq \sum_{i=1}^n |b_i^*(v)|\|T(b_i)\|_Y\leq p(v)\max_{1\leq i\leq n}\|T(b_i)\|_Y \leq C\max_{1\leq i\leq n}\|T(b_i)\|_Y.
\end{equation*}
Let $1\leq i\leq n$ be such that $\|T\|\leq C\|T(b_i)\|_Y$.
Then
\begin{equation*}
\Lip(T\circ\varphi;x)\geq\limsup_{j\rightarrow \infty}\frac{\|T\circ\varphi(x_j^i)-T\circ\varphi(x)\|_Y}{d(x_j^i,x)}
= \|T(b_i)\|_Y\geq C^{-1}\|T\|.
\end{equation*}
\end{proof}
\begin{lemma}\label{lemma:Y_LDS_vs_LDS}
Let $Y$ be a non-zero Banach space,
$(X,d,\mu)$ a $Y$-LDS,
and $\{(U_i,\varphi_i\colon X\to\R^{n_i})\}_i$ and $\{(V_j,\psi_j\colon X\to\R^{m_j})\}_j$
be Cheeger atlases of $(X,d,\mu)$ when viewed as an $\R$-LDS and $Y$-LDS, respectively.
Then, for $i,j$ such that $\mu(U_i\cap V_j)>0$, we have $n_i=m_j$ and
the differentials $d_x^{\varphi_i}\psi_j$, $d_x^{\psi_j}\varphi_i$
are invertible at $\mu$-a.e.\ $x\in U_i\cap V_j$.
\end{lemma}
\begin{proof}
Drop $i,j$ from the notation and set $W:=U\cap V$.
On $W$, we may differentiate one chart w.r.t\ the other obtaining, for $\mu$-a.e.\ $x\in W$,
linear maps $d_x^\varphi\psi\colon\R^n\to\R^m$ and $d_x^\psi\varphi\colon\R^m\to\R^n$ which satisfy
\begin{equation*}
	\Lip(\langle v,d_x^\psi\varphi\circ\psi\rangle;x)=\Lip(\langle v,\varphi\rangle;x)>0
\end{equation*}
and $\Lip(\langle w,d_x^\varphi\psi\rangle;x)>0$ for $v\in\R^n\setminus\{0\}$ and $w\in\R^m\setminus\{0\}$.
In particular, $d_x^\psi\varphi$ and $d_x^\varphi\psi$ are surjective linear maps, proving $n=m$
and hence the thesis.
\end{proof}
\Cref{lemma:Y_LDS_vs_LDS} implies that Cheeger atlases do not depend on the target and, moreover,
that differentiability depends on the choice of Cheeger atlas only up to a $\mu$-null set.
Given a chart $(U,\varphi\colon X\to\R^n)$ with $\mu(U)>0$, we call $n$
the \emph{analytic dimension} of $U$.
More generally, if $E$ is $\mu$-measurable with $\mu(E)>0$ and there are charts $(U_i,\varphi_i\colon X\to\R^n)$
with $\mu(E\setminus\bigcup_i U_i)=0$, we say that $E$ has analytic dimension $n$.
\begin{lemma}\label{lemma:diff_and_epsilon_diff}
Let $X$ be a metric space, $x\in X$, $n\in\N$, and let $\varphi\colon X\to\R^n$ be Lipschitz
such that $\Lip(\langle v,\varphi\rangle;x)>0$ for $v\in\R^n\setminus\{0\}$.
Then, for any Banach space $Y$ and function $f\colon X\to Y$, the following are equivalent:
\begin{itemize}
\item $f$ is $\varphi$-differentiable at $x$, i.e.\ there is a unique linear map $T\colon\R^n\to Y$ such that
\begin{equation*}
\Lip(f-T\circ\varphi;x)=0;
\end{equation*}
\item for every $\epsilon>0$ there is a linear map $T_\epsilon\colon\R^n\to Y$ such that
\begin{equation*}
\Lip(f-T_\epsilon\circ\varphi;x)\leq \epsilon.
\end{equation*}
\end{itemize}
\end{lemma}
A similar argument appears at the end of \cite[Proof of backward implication of Theorem 8.2]{bate_li_RNP_LDS}.
\begin{proof}
We need only to prove that the second item implies the first.
Triangle inequality and \cref{lemma:indep_and_Banach_indep} show that $(T_\epsilon)_\epsilon$ is
Cauchy as $\epsilon\rightarrow 0$.
Its limit $T\colon\R^n\to Y$ then satisfies
\begin{equation*}
\Lip(f-T\circ\varphi;x)\leq\Lip(f-T_\epsilon\circ\varphi;x)+\|T-T_\epsilon\|\Lip(\varphi;x)\rightarrow 0
\end{equation*}
as $\epsilon\rightarrow 0$.
Uniqueness follows from \cref{lemma:indep_and_Banach_indep}.
\end{proof}
\begin{defn}\label{defn:porous}
Let $(X,d)$ be a metric space and $E\subseteq X$ a set.
We say that $E$ is \emph{porous at $x\in E$ (w.r.t.\ $d$)} if
there is $\epsilon>0$ and $(y_i)\subseteq X$, $y_i\rightarrow x$, such that
$E\cap B(y_i,\epsilon d(x,y_i))=\varnothing$ for all $i\in\N$.
We call \emph{porous} those sets $E$ which are porous at every $x\in E$.
\end{defn}
Observe that, for a non-empty subset $E$ of a metric space $X$,
the map $x\in X\mapsto d(E,x)$ is $1$-Lipschitz.
\begin{lemma}\label{lemma:ptwise_Lip_and_porosity}
Let $X$ be a metric space, $E\subseteq X$ a set, and $x\in E$.
Then the following are equivalent:
\begin{itemize}
\item
$E$ is not porous at $x$;
\item
for every metric space $Y$ and $f\colon X\to Y$ Lipschitz,
it holds $\Lip(f\vert_{E};x)=\Lip(f;x)$;
\item
$\Lip(d(E,\cdot);x)=0$.
\end{itemize}
\end{lemma}
\begin{proof}
The fact that the first item implies the second is in \cite[Lemma 2.6]{bate_li_RNP_LDS}.
Then the second clearly implies the third, which implies the first by definition of porosity.
\end{proof}
\begin{rmk}\label{rmk:porosity_Borel}
From \cref{lemma:ptwise_Lip_and_porosity}, it is easy to see that
for any non-empty set $E\subseteq X$ there is a porous Borel set $E'$
such that
\begin{equation*}
	\{x\in E\colon E\text{ is porous at }x\}\subseteq E'.
\end{equation*}
Indeed, since $d(E,\cdot)=d(\overline{E},\cdot)$, we may take $E'=\{x\in\overline{E}\colon\Lip(d(E,\cdot);x)>0\}$. \par
\end{rmk}
A Borel measure $\mu$ on a metric space $X$ is called \emph{doubling} if
there is $C\geq 1$ such that
\begin{equation*}
0<\mu(B(x,2r))\leq C\mu(B(x,r))<\infty
\end{equation*}
for $x\in X$ and $0<r<\diam X$.
More generally,
a locally finite Borel measure $\mu$ on a separable metric space $X$
is \emph{asymptotically doubling} if 
$\limsup_{r\rightarrow 0}\frac{\mu(B(x,2r))}{\mu(B(x,r))}<\infty$ at $\mu$-a.e.\ $x\in \spt\mu$.
Lebesgue differentiation theorem holds w.r.t.\ any asymptotically doubling measure (see e.g.\ \cite[Section 3.4]{HKST_Sobolev_on_metric_measure_spaces})
and, in particular, almost every point of a measurable set is a Lebesgue density point.
\begin{lemma}\label{lemma:doubling_meas_and_porosity}
Let $X$ be a metric space and $\mu$ a doubling measure on $X$.
Let $E\subseteq X$ be a $\mu$-measurable set and $x\in E$ a Lebesgue density point of $E$ w.r.t.\ $\mu$.
Then $E$ is not porous at $x$. \par
In particular, for any set $E\subseteq X$, $\mu$-a.e.\ $x\in E$ is not a porosity point of $E$.
\end{lemma}
\begin{proof}
We prove the contrapositive.
Suppose $E$ is porous at $x\in E$ and let $\epsilon>0$, $(y_i)\subseteq X\setminus\{x\}$
as in \cref{defn:porous}.
Set $r_i:=d(x,y_i)$ and observe that
\begin{equation*}
\mu(E\cap B(x,(1+\epsilon)r_i))\leq \mu(B(x,(1+\epsilon)r_i))-\mu(B(y_i,\epsilon r_i)).
\end{equation*}
By \cite[Lemma 8.1.13]{HKST_Sobolev_on_metric_measure_spaces}, there are constants $C\geq 1$ and $0<q<\infty$
such that
\begin{equation*}
	\frac{\mu(B(y_i,\epsilon r_i))}{\mu(B(x,(1+\epsilon)r_i))}\geq C^{-1}\left(\frac{\epsilon r_i}{(1+\epsilon)r_i}\right)^q=C^{-1}(\epsilon/(1+\epsilon))^q,
\end{equation*}
proving
\begin{equation*}
	\liminf_{r\rightarrow 0}\frac{\mu(E\cap B(x,r))}{\mu(B(x,r))}<1.
\end{equation*}
The `In particular' part of the statement follows taking $E'$ as in \cref{rmk:porosity_Borel}
and observing that,
by Lebesgue differentiation theorem and the above, it must be $\mu(E')=0$.
\end{proof}
\begin{lemma}[{\cite[Theorem 2.4]{batespeight2013}}]\label{lemma:porous_null}
Let $(X,d,\mu)$ be an LDS and $E\subseteq X$ a porous set.
Then $\mu(E)=0$. \par
In particular, for any set $E\subseteq X$, $\mu$-a.e.\ $x\in E$ is not a porosity point of $E$.
\end{lemma}
\begin{proof}
Bate and Speight \cite{batespeight2013} consider complete and separable metric measure spaces with a locally finite Borel measure, but,
in fact, the same proof applies to any metric space with a $\sigma$-finite Borel measure.
\end{proof}
The proof of \cite[Lemma 2.10]{bate_li_RNP_LDS}
gives the following.
\begin{lemma}[{\cite[Lemma 2.10]{bate_li_RNP_LDS}}]\label{lemma:subsets_LDS}
Let $Y$ be a Banach space, $(X,d,\mu)$ a $Y$-LDS, $\{(U_i,\varphi_i\colon X\to\R^{n_i})\}_i$ a Cheeger atlas,
and let $E\subseteq X$ be a $\mu$-measurable set.
Then $(E,d,\mu)$ is a $Y$-LDS and $\{(U_i\cap E,\varphi_i)\}_i$ a Cheeger atlas.
\end{lemma}
We also need the following.
\begin{lemma}\label{lemma:continuity_Cheeger_diff}
Let $Y$ be a Banach space, $(X,d,\mu)$ a $Y$-LDS,
and $f,f_j\colon X\to Y$ Lipschitz maps.
Suppose $f_j\rightarrow f$ pointwise on $X$ and $\sup_j\LIP(f_j)<\infty$.
Then, for non-negative $g\in L^1(\mu)$, it holds
\begin{equation}\label{eq:continuity_Cheeger_diff_1}
	\int_X g\Lip(f;\cdot)\,\dd\mu \leq \liminf_{j\rightarrow \infty} \int_X g\Lip(f_j;\cdot)\,\dd\mu.
\end{equation}
In particular, for $\mu$-measurable $U\subseteq X$ and $p\in [1,\infty]$, we have
\begin{equation}\label{eq:continuity_Cheeger_diff_2}
	\|\chi_U\Lip(f;\cdot)\|_{L^p(\mu)}\leq \liminf_{j\rightarrow \infty}\|\chi_U\Lip(f_j;\cdot)\|_{L^p(\mu)}.
\end{equation}
\begin{proof}
Possibly replacing $Y$ with the closed span of $\bigcup_jf_j(X)$, we may assume $Y$ to be separable.
Hence, there is a countable $1$-norming set $\{y_n^*\colon n\in\N\}\subseteq Y^*$, i.e.\
$\sup_{n}|\langle y^*_n,y\rangle|=\|y\|_Y$ for $y\in Y$.
It is then not difficult to see that
$\Lip(f;x)=\sup_n\Lip(y_n^*\circ f;x)$
whenever
$x\in X$ is a differentiability point $f$ w.r.t.\ some chart.
Let $g\in L^1(\mu)$ be non-negative and
fix $\epsilon\in(0,1)$.
Set $\tilde{E}_n:=\{\Lip(y_n^*\circ f;\cdot)\geq (1-\epsilon)\Lip(f;\cdot)\}$, $E_n:=\tilde{E}_n\setminus\bigcup_{m<n}\tilde{E}_m$,
and observe that $(E_n)$ is a disjoint collection of Borel sets covering $\mu$-almost all of $X$.
Since $\mu$ is Radon, there are compact sets $K_n\subseteq E_n$ such that $\int_{E_n\setminus K_n}g\,\dd\mu\leq \epsilon 2^{-n}$ for $n\in\N$.
By \cref{lemma:ptwise_Lip_and_porosity,lemma:porous_null}, we have $\Lip((y_n^*\circ f)\vert_{K_n};x)=\Lip(y_n^*\circ f;x)$
at $\mu$-a.e.\ $x\in K_n$.
Hence, \cite[Lemma 7.3]{bate_EB_soultanis_fragmentwise_diff_str} applied to the
compact LDS $(K_n,d,\mu)$ and
Lipschitz maps $(y_n^*\circ f)\vert_{K_n}$, $(y_n^*\circ f_j)\vert_{K_n}$ yields
\begin{align*}
(1-\epsilon)\int_{K_n}g\Lip(f;\cdot)\,\dd\mu
&\leq
\int_{K_n}g\Lip(y_n^*\circ f;\cdot)\,\dd\mu
\leq
\liminf_{j\rightarrow\infty}
\int_{K_n}g\Lip((y_n^*\circ f_j)\vert_{K_n};\cdot)\,\dd\mu \\
&\leq
\liminf_{j\rightarrow\infty}
\int_{K_n}g\Lip(f_j;\cdot)\,\dd\mu,
\end{align*}
where we have also used $\|y_n^*\|_{Y^*}\leq 1$.
Then, summing over $n\in\N$ and using $\int_{X\setminus \bigcup_n K_n}g\,\dd\mu\leq \epsilon$, we have
\begin{equation*}
(1-\epsilon)\int_{X}g\Lip(f;\cdot)\,\dd\mu
\leq
\liminf_{j\rightarrow\infty}
\int_{X}g\Lip(f_j;\cdot)\,\dd\mu
+\LIP(f)(1-\epsilon)\epsilon,
\end{equation*}
which implies \cref{eq:continuity_Cheeger_diff_1} because $\epsilon\in (0,1)$ was arbitrary. \par
Finally, \cref{eq:continuity_Cheeger_diff_2} follows from \cref{eq:continuity_Cheeger_diff_1}
by H\"older inequality
and a standard dual characterisation of the $L^p(\mu)$-norm.
\end{proof}

\end{lemma}
Recall that a Banach space $Y$ has the \emph{Radon-Nikodym property (RNP)} if every Lipschitz function
$f\colon\R\to Y$ is differentiable Lebesgue a.e., see \cite{benyamini_lindenstraus_geom_nonlinear_FA,Pisier_martingales_in_Banach_2016}
for more on Banach spaces with RNP. \par
In the definition of $Y$-LDS, we do not require the Banach space $Y$ to have the Radon-Nikodym property,
because such assumption is apriori unnecessary.
Nonetheless, it is not difficult to see that RNP is required in all but trivial cases.
The equivalence between the first two points is due to Bate \cite[Remark 4.10]{bate_str_of_measures_2015}.
\begin{prop}\label{prop:non_RNP_LDS}
Let $X=(X,d,\mu)$ be a metric measure space.
Then the following are equivalent:
\begin{itemize}
\item
$\mu$-a.e.\ $x\in X$ is isolated;
\item
$X$ is an LDS of analytic dimension $0$;
\item
$X$ is a $Y$-LDS for every Banach space $Y$;
\item
$X$ is a $Y$-LDS for some Banach space $Y$ without the Radon-Nikodym property.
\end{itemize}
\end{prop}
For the proof of \cref{prop:non_RNP_LDS}, we need the following lemma.
\begin{lemma}\label{lemma:prelim_non_RNP_LDS}
Let $E\subseteq\R$ be a set, $Y$ a Banach space, $\varphi\colon E\to\R$ a function,
and $f\colon\R\to Y$ a Lipschitz function.
Suppose $E$ is not porous at $t_0\in E$, 
$\varphi$ and $f\circ\varphi$ are differentiable at $t_0$,
and $\varphi'(t_0)\neq 0$.
Then $f$ is differentiable at $\varphi(t_0)$.
\end{lemma}
\begin{proof}
We may assume w.l.o.g.\ $t_0=0=\varphi(t_0)$ and $\varphi'(0)=1$,
so that $D:=(f\circ\varphi)'(0)=(f\circ\varphi)'(0)/\varphi'(0)\in Y$.
Let $(x_j)\subseteq\R\setminus\{0\}$ be a sequence converging to $0$.
Since $E$ is not porous at $0\in E$, there is $(t_j)\subseteq E\setminus\{0\}$
such that
\begin{equation}\label{eq:prelim_non_RNP_LDS}
	\frac{|t_j-x_j|}{|x_j|}\rightarrow 0 \qquad\text{and}\qquad  \frac{|t_j|}{|x_j|}\rightarrow 1,
\end{equation}
where the second condition follows from the first.
Then, if we divide by $|x_j|$ all terms of the inequality
\begin{align*}
\|f(x_j)-f(0)-Dx_j\|_Y
&\leq
\LIP(f)(|x_j-t_j|+|\varphi(t_j)-t_j|)
+
\|f\circ\varphi(t_j)-f(0)-Dt_j\|_Y \\
&\qquad+
\|D(t_j-x_j)\|_Y,
\end{align*}
we see from \cref{eq:prelim_non_RNP_LDS} that
$\|f(x_j)-f(0)-Dx_j\|_Y/|x_j|\rightarrow 0$ as $j\rightarrow\infty$.
Since $(x_j)$ was arbitrary, this shows $f'(0)=D$, as claimed.
\end{proof}
\begin{proof}[Proof of \cref{prop:non_RNP_LDS}]
It is clear that the first point implies all others (see also \cref{rmk:isolated_points_zero_dim_charts})
and that the third implies the fourth.
We show that the fourth implies the second, which implies the first by \cite[Remark 4.10]{bate_str_of_measures_2015}. \par
We begin by proving the following claim:
If $(X,d,\mu)$ is a $Y$-LDS and there is a
chart $(K,\varphi\colon X\to\R^n)$ with $n\geq 1$, $K$ compact, and $\mu(K)>0$,
then $Y$ has RNP. \par
By \cref{lemma:subsets_LDS}, $\varphi$ is a chart also in
the $Y$-LDS $(K,d,\mu)$, so we may assume $K=X$.
Since necessarily $Y\neq 0$, $(X,d,\mu)$ is also an LDS,
and we are then
in the setting of \cite{bate_str_of_measures_2015}\footnote{%
In \cite{bate_str_of_measures_2015}, it is assumed that the (countable) set of isolated points
is $\mu$-null.
This condition is satisfied in our case because $X=K$ has analytic dimension $1$;
see (the first part of) \cref{rmk:isolated_points_zero_dim_charts}.
}.
Let $f\colon\R\to Y$ be Lipschitz;
by \cite[Theorem 5.21]{benyamini_lindenstraus_geom_nonlinear_FA},
we only need to show that $f$ has a point of differentiability. \par
Write $\varphi=(\varphi^1,\dots,\varphi^n)$,
set $F:=f\circ\varphi^1$,
and observe that, since $F\colon X\to Y$ is Lipschitz,
we have
\begin{equation}\label{eq:non_RNP_LDS}
F(y)-F(x)
=d_x^\varphi F(\varphi(y)-\varphi(x))+o(d(x,y)),
\end{equation}
for $\mu$-a.e.\ $x\in X$.
By \cite[Corollary 6.7]{bate_str_of_measures_2015},
for $\mu$-a.e.\ $x\in X$ and $1\leq k\leq n$, there are biLipschitz functions $\gamma^{x,k}\colon K^{x,k}\to X$
such that:
$K^{x,k}\subseteq\R$ is compact and $0\in K^{x,k}$ is a Lebesgue density point w.r.t.\ Lebesgue measure,
$\gamma^{x,k}(0)=x$,
$(\varphi\circ\gamma^{x,1})'(0), \dots, (\varphi\circ\gamma^{x,n})'(0)$
exist and form a basis of $\R^n$. \par
Fix $x\in X$ such that the above and \cref{eq:prelim_non_RNP_LDS} hold,
let $1\leq k\leq n$
be such that $(\varphi^1\circ\gamma^{x,k})'(0)\neq 0$,
and observe that,
by \cref{eq:prelim_non_RNP_LDS},
$(F\circ\gamma^{x,k})'(0)$ exists.
Then $f\circ(\varphi^1\circ\gamma^{x,k})=F\circ\gamma^{x,k}$
and $(\varphi^1\circ\gamma^{x,k})$ are differentiable at $0$,
$(\varphi^1\circ\gamma^{x,k})'(0)\neq 0$,
and $0\in K^{x,k}$ is not a porosity point of $K^{x,k}$ (by \cref{lemma:doubling_meas_and_porosity}).
Since $f$ is Lipschitz, by \cref{lemma:prelim_non_RNP_LDS} we conclude that $f$
is differentiable at $(\varphi^1\circ\gamma^{x,k})(0)$,
proving the claim. \par
We now conclude the proof of the statement.
Since $\mu$ is Radon,
there is a Cheeger atlas $\{(K_i,\varphi_i\colon X\to\R^{n_i})\}_i$, where $K_i$ is compact and $\mu(K_i)>0$.
(We may assume w.l.o.g.\ $\mu\neq 0$.)
Since, by assumption, $Y$ does not have RNP, the claim implies $n_i=0$ for every $i$.
In particular, $X$ is an LDS of analytic dimension $0$.
Then, by \cref{lemma:subsets_LDS}, $(K_i,0)$ is a chart in the complete LDS $(K_i,d,\mu)$
and so, by \cite[Remark 4.10]{bate_str_of_measures_2015}, $\mu$-a.e.\  $x\in K_i$ is isolated in $K_i$.
Finally, by \cref{lemma:porous_null}, $\mu$-a.e.\ $x\in K_i$ is isolated in $X$, concluding the proof.
\end{proof}
\subsection{Three useful lemmas}
In this subsection
we collect three lemmas which will be used repeatedly in the rest of the paper. \par
Although we state the following for sequences, in reality it is a fact about measures.
\begin{lemma}\label{lemma:good_subseq}
Let $(\alpha_i)\subseteq (0,\infty)$ be a sequence.
Then the following are equivalent:
\begin{itemize}
\item $\inf\{\alpha_i\colon i\notin I\}=0$ for any $I\subseteq\N$ with $\sum_{i\in I}\alpha_i<\infty$;
\item for $I\subseteq\N$ and $\sum_{i\in I}\alpha_i< t<\infty$, there is $J\supseteq I$ such that
\begin{equation*}
\sum_{i\in J}\alpha_i=t.
\end{equation*}
\end{itemize}
\end{lemma}
\begin{proof}
If the first item fails, then it is clear that the second is false as well.
Conversely, if the first condition holds, then we may argue as in the proof of \cite[Proposition 1.20]{Fonesca_Leoni_modern_methods}
(from `Next we claim [...]').
\end{proof}
\begin{lemma}\label{lemma:summable_subsequence}
Let $(\alpha_i)\subseteq (0,\infty)$ be a sequence.
Suppose $\inf\{\alpha_i\colon i\notin I\}=0$
whenever
$\sum_{i\in I}\alpha_i<\infty$, $I\subseteq \N$.
Then, there is a disjoint collection $(I_n)_{n}$ of finite non-empty subsets of $\N$
such that
\begin{equation*}
\sum_{n\in\N}\sum_{i\in I_n}\alpha_i=\infty \qquad \text{and} \qquad \sum_{n\in\N}\bigg(\sum_{i\in I_n}\alpha_i^p\bigg)^\sigma<\infty
\end{equation*}
for all $0<\sigma\leq 1<p<\infty$.
In particular, there is a subsequence $(\alpha_{i_j})_j$ satisfying
\begin{equation*}
\sum_{j\in\N}\alpha_{i_j}=\infty \qquad \text{and} \qquad \sum_{j\in\N}\alpha_{i_j}^p<\infty,
\end{equation*}
for $1<p<\infty$.
\end{lemma}
\begin{proof}
We first need to construct an auxiliary sequence.
Recall that for all $\epsilon>0$ and $1<p<\infty$ there is $(t_k)\subseteq (0,\infty)$
such that $\sum_k t_k=\infty$ and $\sum_kt_k^p\leq\epsilon$.
It follows that for all $0<\sigma\leq 1<p<\infty$ and $\epsilon>0$
there are $N=N(\sigma,p,\epsilon)\in\N$ and $t_1,\dots,t_N\in (0,1)$ satisfying
$\sum_{k=1}^Nt_k=1$ and $\left(\sum_{k=1}^Nt_k^p\right)^\sigma\leq\epsilon$. \par
Fix $(\sigma_n)_n\subseteq (0,1]$ and $(p_n)_n\subseteq (1,\infty)$ with $\sigma_n\rightarrow 0$ and $p_n\rightarrow 1$.
From the above, we find a strictly increasing sequence $(k_n)_{n\in\N_0}\subseteq\N_0$ with $k_0=0$
and $(t_k)_k\subseteq (0,1)$ such that
\begin{equation*}
\sum_{k=k_{n-1}+1}^{k_n}t_k=1, \qquad \bigg(\sum_{k=k_{n-1}+1}^{k_n}t_k^{p_n}\bigg)^{\sigma_n}\leq 2^{-n},
\end{equation*}
for $n\in\N$.
It is then not difficult to see that
\begin{equation}\label{eq:summable_subseq_eq_1}
\sum_{k\in\N}t_k=\infty, \qquad \sum_{n\in\N}\bigg(\sum_{k=k_{n-1}+1}^{k_n}t_k^{p}\bigg)^{\sigma}<\infty
\end{equation}
for all $0<\sigma\leq 1<p<\infty$. \par
Applying \cref{lemma:good_subseq} inductively, we find non-empty disjoint $(\tilde{J}_k)_k$
with $\sum_{i\in \tilde{J}_k}\alpha_i=t_k$ for $k\in\N$.
Then there is a collection $(J_k)_k$ of finite non-empty sets with $J_k\subseteq \tilde{J}_k$
and $\sum_{i\in J_k}\alpha_i\sim t_k$ for $k\in\N$.
Define $I_n:=\bigcup_{k=k_{n-1}+1}^{k_n}J_k$ for $n\in\N$.
By \cref{eq:summable_subseq_eq_1}, we have
for $0<\sigma\leq 1<p<\infty$
\begin{equation*}
\sum_{n\in\N}\sum_{i\in I_n}\alpha_i=\sum_{n\in\N}\sum_{k=k_{n-1}+1}^{k_n}\sum_{i\in J_k}\alpha_i
\sim \sum_{n\in\N}\sum_{k=k_{n-1}+1}^{k_n}t_k=\infty
\end{equation*}
and
\begin{equation*}
\sum_{i\in I_n}\alpha_i^p=\sum_{k=k_{n-1}+1}^{k_n}\sum_{i\in J_k}\alpha_i^p
\leq\sum_{k=k_{n-1}+1}^{k_n}\bigg(\sum_{i\in J_k}\alpha_i\bigg)^p
\sim \sum_{k=k_{n-1}+1}^{k_n}t_k^p,
\end{equation*}
which implies $\sum_{n\in\N}\left(\sum_{i\in I_n}\alpha_i^p\right)^\sigma<\infty$ and concludes the proof.
\end{proof}
Let $X=(X,d)$ be a metric space, $E\subseteq\R$ closed and non-empty, and $\gamma\colon E\to X$ a continuous function.
We define the \emph{variation} of $\gamma$ as
\begin{equation}\label{eq:variation_curve}
\var\gamma:=\sup\left\{\sum_{i=1}^nd(\gamma(t_{i-1}),\gamma(t_i))\colon n\in\N, t_0\leq \dots\leq t_n\in E\right\}.
\end{equation}
We say that $X$ is \emph{$C$-quasiconvex}, $C\geq 1$, if for every $x,y\in X$ there is a continuous $\gamma\colon [0,1]\to X$
with $\gamma(0)=x$, $\gamma(1)=y$, and $\var\gamma\leq Cd(x,y)$.
A metric space is \emph{quasiconvex} it is $C$-quasiconvex for some $C\geq 1$
and a \emph{length space} if it is $C$-quasiconvex for every $C>1$.
\begin{lemma}[{\cite[Lemma 2]{Ives_Preiss_not_too_well_diff_LIP_isom}}]\label{lemma:LIP_on_pieces_to_LIP}
Let $X$ be a quasiconvex metric space and set
\begin{equation*}
	C_0:=\inf\{C\in [1,\infty)\colon X \text{ is }C\text{-quasiconvex}\}.
\end{equation*}
Suppose $Y$ a metric space, $\calQ$ a locally finite closed cover of $X$,
and $f\colon X\to Y$ a function
such that
$L:=\sup_{Q\in\calQ}\LIP(f\vert_{Q})<\infty$.
Then $f$ is $C_0L$-Lipschitz on $X$.
In particular, if $X$ is a length space, then $f$ is $L$-Lipschitz.
\end{lemma}
\begin{proof}
Let $\epsilon>0$, $x,y\in X$, and $\gamma\colon [0,1]\to X$ be continuous with
$\gamma(0)=x$, $\gamma(1)=y$, and
$\var\gamma \leq (C_0+\epsilon)d(x,y)=:L_0$.
By \cite[Proposition 5.1.8]{HKST_Sobolev_on_metric_measure_spaces}, we may assume $\gamma$ to be $L_0$-Lipschitz.
Since $\gamma([0,1])$ is compact and $\calQ$ locally finite, there are finitely many $Q_1,\dots,Q_n\in\calQ$
covering $\gamma([0,1])$.
Set $Z_i:=\gamma^{-1}(Q_i)$ and observe that $f\circ\gamma$ is $LL_0$-Lipschitz on $Z_i$ for each $i$.
In particular, $f\circ\gamma$ is continuous on the closed sets $Z_1,\dots, Z_n$ and the latter cover $[0,1]$;
this implies that $f\circ\gamma$ is continuous.
Then, the proof of \cite[Lemma 2]{Ives_Preiss_not_too_well_diff_LIP_isom} shows that
$d_Y(f(x),f(y))=d_Y(f\circ\gamma(0),f\circ\gamma(1))\leq L_0L=(C_0+\epsilon)Ld(x,y)$.
\end{proof}
%The quasiconvexity assumption in \cref{lemma:LIP_on_pieces_to_LIP}
%cannot be removed (take e.g.\ $X=\N$ and $\calQ=\{\{n\}\colon n\in\N\}$)
%and the constant $C_0$ is optimal:
%For $\epsilon\in (0,1]$, consider the $(1/\epsilon)$-quasiconvex metric space
%$X=([0,1],d_\epsilon)$,
%$d_\epsilon(x,y):=\min\{|x-y|,1+\epsilon-|x-y|\}$,
%finite cover $\calQ:=\{[0,1/2],[1/2,1]\}$, and $(1/\epsilon)$-Lipschitz function
%$f\colon X\to \R$, $f(x):=x$.

\subsection{PI spaces}\label{subsec:PI_spaces}
\begin{defn}\label{defn:PI_space}
We say that a metric measure space $(X,d,\mu)$
supports a \emph{$p$-Poincar\'e inequality}, $p\in [1,\infty)$,
if $\mu$ is positive and finite on balls and there are
$C>0$ and $\lambda\geq 1$
such that for every Lipschitz function $f\colon X\to\R$
and ball $B=B(x,r)$
\begin{equation*}
	\dashint_B|f-f_B|\,\dd\mu\leq C\diam(B)\left(\dashint_{\lambda B}\Lip(f;\cdot)^p\,\dd\mu\right)^{\frac1p},
\end{equation*}
where $\lambda B=B(x,\lambda r)$ and $f_B:=\dashint_Bf\,\dd\mu:=\int_Bf\,\dd\mu/\mu(B)$.
We refer to $(C,\lambda)$ as the constants in the $p$-Poincar\'e inequality of $(X,d,\mu)$. \par
We call $(X,d,\mu)$ a \emph{$p$-PI space} if, in addition, $\mu$ is doubling and $(X,d)$ complete,
and a \emph{PI space} if it is a $p$-PI space for some $p\in [1,\infty)$.
\end{defn}
Although not immediately apparent from the definition, any two points $x,y\in X$
in a PI space $X$ may be connected by a rich family of curves
(in fact, this is a characterisation; see \cite{Keith_modulus_and_Poincare,Eriksson_Bique_Gong_almost_uniform_and_PI,Caputo_Cavallucci_position_function}).
In particular, PI spaces are quasiconvex; see \cite[Appendix]{cheeger1999}. \par
We refer the reader to \cite{HKST_Sobolev_on_metric_measure_spaces,Bjorn_Bjorn_Nonlinear_potential_theory}
for more information on PI spaces,
including several characterisations.
More recent equivalent conditions may be found in
\cite{erikssonbique2019_PI_rect,Eriksson_Bique_Gong_almost_uniform_and_PI,Caputo_Cavallucci_PI_and_energy_sep_sets,Caputo_Cavallucci_position_function}
and the survey \cite{Caputo_PI_survey}.
\par
For $s\in (0,\infty)$, we say that a Borel measure $\mu$ on a metric space $X$
\emph{Ahlfors-David regular of dimension $s$} (or $s$-ADR)
if there is $C\geq 1$ such that
\begin{equation}\label{eq:ADR_meas}
	C^{-1}r^s\leq \mu(B(x,r))\leq Cr^s,
\end{equation}
for all $x\in X$ and $0<r<\diam X$; see also \cref{subsec:Borel_Cantelli}.
We say that a metric measure space $(X,d,\mu)$ is $s$-ADR if $\mu$ is.
We refer to the least $C\geq 1$ as in \cref{eq:ADR_meas} as the ADR constant of $\mu$.
\par
The proof of the following lemma is inspired by (and very similar to) \cite[Lemma 4.1, Proposition 6.1(a)]{BjornbjornLerhabck_sharpcapacity}.
\begin{lemma}\label{lemma:capacitary_estimate}
Let $(X,d,\mu)$ be an $s$-ADR metric measure space and suppose
it supports a $p_0$-Poincar\'e inequality with constants $(C_P,\lambda_P)$, for some $1\leq p_0<s$.
Then, for $1\leq p<s$, there is a constant $C>0$ such that
for every metric space $Y$, $L$-Lipschitz function $f\colon X\to Y$, and $x,y\in X$,
we have
\begin{equation}\label{eq:capacity_estimate}
	d_Y(f(x),f(y))^s\leq CL^{s-p}\int_{B(x,2\lambda_P d(x,y))}\Lip(f;\cdot)^p\,\dd\mu.
\end{equation}
More precisely, $C$ depends only on $C_P,\lambda_P,s,p$, and the ADR constant of $\mu$.
\end{lemma}
\begin{rmk}\label{rmk:capacity_estimate_1}
By \cite{Keith_Zhong_self_improvement_Poincare} (see also \cite{Eriksson_Bique_self_improvement_Poincare}),
a complete metric measure space satisfies the assumptions of \cref{lemma:capacitary_estimate} if and only if
it is $s$-ADR and $s$-PI.
\end{rmk}
\begin{rmk}\label{rmk:capacity_estimate_2}
Let $X=(X,d,\mu)$ be a complete and quasiconvex $s$-ADR metric measure space, for some $s\in (1,\infty)$.
If \cref{eq:capacity_estimate} holds on $X$ (even just for $p=1$), then
for every $q>s$ and $\delta\in(0,1)$
there are $C\geq 1$ and $\tau_0\in (0,1)$ such that $X$ $(C,\delta,\tau_0,q)$-connected;
see \cite[Definition 2.16]{Eriksson_Bique_Gong_almost_uniform_and_PI}.
To prove this,
for a given closed `obstacle' $E\subseteq X$ with `density' at most $\tau_0\in (0,1)$ and $x,y\in X$,
consider the Lipschitz function
\begin{equation*}
	f(z):=\inf_{\gamma}\int_{\gamma} \chi_E\,\dd s+\tau_0\var\gamma,
\end{equation*}
where the infimum is taken over all rectifiable curves $\gamma\colon [0,1]\to X$ from $x$ to $z$.
See \cref{eq:variation_curve} for the definition of $\var\gamma$.
For measurable obstacles, one can argue as in \cite[Remark 2.15]{Eriksson_Bique_Gong_almost_uniform_and_PI}.
Then, by \cite[Theorem 2.19]{Eriksson_Bique_Gong_almost_uniform_and_PI}
we deduce that $X$ supports a $q$-Poincar\'e inequality for every $q>s$.
This cannot improved to $q=s$, as \cref{eq:capacity_estimate} is stable under gluing at a point,
while $s$-Poincar\'e inequalities are not, because points have zero $s$-capacity in $s$-ADR spaces;
see e.g.\ \cite[Corollary 5.3.11]{HKST_Sobolev_on_metric_measure_spaces}.
\end{rmk}
\cref{eq:capacity_estimate} is (at least formally) stronger for larger values of $p$
and fails for $p=s$ on every metric space with at least two points.
By
\cref{rmk:capacity_estimate_1} and \cref{rmk:capacity_estimate_2},
\cref{eq:capacity_estimate}
may be interpreted as a condition in-between an $s$-Poincar\'e inequality and $q$-Poincar\'e inequality for every $q>s$.
We also point out that \cref{eq:capacity_estimate} (even just for $p=1$) already implies $\mu\big(B(x,2\lambda_Pd(x,y))\big)\gtrsim d(x,y)^s$ for $x,y\in X$,
as can be seen taking $f=d(x,\cdot)$.
\begin{proof}
If we prove \cref{eq:capacity_estimate} for $p=p_0$, then we obtain the inequality also for $1\leq p<p_0$,
because \cref{eq:capacity_estimate} is weaker for smaller values of $p$.
Also, if $X$ has $p_0$-Poincar\'e inequality, it also has a $q_0$-Poincar\'e inequality for all $q_0\geq p_0$
(with the same constants).
Hence, it is enough to prove \cref{eq:capacity_estimate} with $p=p_0$. \par
Post-composing $f$ with the $1$-Lipschitz function $d_Y(f(x),\cdot)$, we see that
we may assume w.l.o.g.\ $Y=\R$.
Suppose
$r:=|f(x)-f(y)|>0$
and set $R:=d(x,y)>0$,
$k:= \ceil{\log_2(8LR/r)}\geq 3$.
Define $B_i:= B(x,2^{1-i}R)$ for $0\leq i\leq k$ and $B_i:= B(y,2^{i+1}R)$ for $-k\leq i<0$.
Since $B_{k}\subseteq B(x,r/4L)$ and $B_{-k}\subseteq B(y,r/4L)$, we have $|f_{B_k}-f_{B_{-k}}|\geq r/2$.
Hence,
\begin{equation*}
r
\lesssim \sum_{i=1-k}^{k}|f_{B_{i-1}}-f_{B_{i}}|
\sim
\sum_{i=1-k}^{0}\dashint_{B_i}|f-f_{B_i}|\,\dd\mu
+ \sum_{i=1}^{k}\dashint_{B_{i-1}}|f-f_{B_{i-1}}|\,\dd\mu
.
\end{equation*}
By the $p$-Poincar\'e inequality and $s$-AD regularity, we have for $-k\leq i\leq k$
\begin{equation*}
\dashint_{B_{i}}|f-f_{B_{i}}|\,\dd\mu
\lesssim \diam B_i\left(\dashint_{\lambda_P B_i}\Lip(f;\cdot)^p\,\dd\mu\right)^{1/p}
\lesssim (2^{-\abs{i}}R)^{1-s/p}\left(\int_{\lambda_P B_0}\Lip(f;\cdot)^p\,\dd\mu\right)^{1/p},
\end{equation*}
thus
\begin{equation*}
r\lesssim \sum_{i=-\infty}^k(2^{-i}R)^{1-s/p}\left(\int_{\lambda_P B_0}\Lip(f;\cdot)^p\,\dd\mu\right)^{1/p}
\sim (2^{-k}R)^{1-s/p}\left(\int_{B(x,2\lambda_P R)}\Lip(f;\cdot)^p\,\dd\mu\right)^{1/p}.
\end{equation*}
Since $2^{-k}R\sim r/L$, rearranging and raising to the power of $p$ concludes the proof.
\end{proof}

\subsection{PI rectifiability}
We adopt the following (equivalent) variant of \cite[Definition 5.1]{erikssonbique2019_PI_rect}.
\begin{defn}[{\cite[Definition 5.1]{erikssonbique2019_PI_rect}}]\label{defn:PI_rect}
Let $(X,d,\mu)$ be a metric measure space.
It is \emph{PI rectifiable}
if it can be $\mu$-almost all covered by countably many $\mu$-measurable sets $(E_i)$
for which there are
PI spaces $(Y_i,d_{Y_i},\nu_i)$ and biLipschitz embeddings $f_i\colon E_i\to Y_i$
satisfying ${f_i}_{\#}\mu\ll \nu_i$. \par
It is \emph{purely PI unrectifiable}
if there is no positive-measure $\mu$-measurable set $E\subseteq X$
such that $(E,d,\mu)$ is PI rectifiable.
\end{defn}
\begin{rmk}
The PI condition (\cref{defn:PI_space}) is stable under biLipschitz change of the distance
and multiplication of the measure by a positive measurable function which is both bounded and bounded away from zero.
From this, it is not difficult to see that in the definition of PI rectifiability
we may equivalently require
$(E_i)$ to be disjoint compact sets and
$(f_i)$ isometric embeddings\footnote{%
To prove this, one may use the following fact.
Let $(X,d)$ be a metric space, $A\subseteq X$ a non-empty subset,
$d_A$ a distance on $A$, and $C\geq 1$. Suppose $C^{-1}d_A\leq d\vert_{A\times A}\leq Cd_A$
and set, for $x,y\in X$,
\begin{align*}
\rho(x,y)
&:=
\left\{
\begin{array}{cc}
d_A(x,y), & x,y\in A \\
Cd(x,y), & \text{otherwise}
\end{array}
\right.,
\\
\widehat{d}(x,y)
&:=\inf\left\{\sum_{i=1}^n\rho(x_{i-1},x_i)\colon n\in\N, x=x_0,\dots,x_n=y\right\}.
\end{align*}
Then $\widehat{d}$ is a distance on $X$, $C^{-1}d\leq \widehat{d}\leq Cd$, and $\widehat{d}\vert_{A\times A}=d_A$.
}
with ${f_i}_{\#}\mu(A)=\nu_i\mres{f(E_i)}(A)$
for all sets $A\subseteq Y_i$%
\footnote{%
To establish equality for all sets, and not just Borel sets,
observe that ${f_i}_{\#}\mu$ and $\nu_i\mres{f(E_i)}$
are both Borel regular.
To prove this, for the former one may use the Lusin-Souslin Theorem (see \cite[Theorem 15.1]{Kechris_classical_descriptive_set_theory}),
while for the latter it follows by Borel regularity of $\nu_i$ (and $\nu_i$-measurability of $f_i(E_i)$).
}.
\end{rmk}
Let $X$ be a metric space, $C\subseteq\R$ closed and non-empty, and $\gamma\colon C\to X$ continuous.
Let $J\subseteq\R$ be the least interval containing $C$
and
$\{(a_i,b_i)\}_{i\in I}$
the unique at most countable collection of disjoint open intervals 
satisfying $C=J\setminus\bigcup_{i\in I}(a_i,b_i)$.
We then define
\begin{equation*}
	\gap\gamma:=\sum_{i\in I}d(\gamma(a_i),\gamma(b_i)).
\end{equation*}
We call \emph{curve fragment} any continuous map $\gamma\colon K\to X$
with $\var\gamma<\infty$, where $K\subseteq\R$ is a non-empty compact set;
see \cref{eq:variation_curve} for the definition of $\var\gamma$.
\begin{defn}[{\cite[Definition 3.1]{erikssonbique2019_PI_rect}}]\label{defn:C_delta_epsilon_connectivity}
Let $X=(X,d,\mu)$ be a metric measure space, $x\neq y\in X$,
$C\geq 1$, $\delta>0$, and $\epsilon>0$.
The pair $(x,y)$ is \emph{$(C,\delta,\epsilon)$-connected} (in $X$)
if for every $\mu$-measurable set $E\subseteq X$ with 
\begin{equation*}
	\mu\big(E\cap B(x,Cd(x,y))\big)<\epsilon\mu\big(B(x,Cd(x,y))\big)
\end{equation*}
there is a curve fragment $\gamma\colon K\to X$ with
$\gamma(\min K)=x$, $\gamma(\max K)=y$,
$\gamma^{-1}(E)\subseteq\{\min K,\max K\}$, and
\begin{equation*}
\begin{aligned}
	\var\gamma&\leq Cd(x,y), \\
	\gap\gamma&<\delta d(x,y).
\end{aligned}
\end{equation*}
For $E\subseteq X$ and $r>0$, we say that $X$ is \emph{$(C,\delta,\epsilon,r)$-connected along $E$}
if for every $x\in E$ and $y\in B(x,r)\setminus\{x\}$ the pair $(x,y)$ is $(C,\delta,\epsilon)$-connected in $X$.
\end{defn}
\cref{defn:C_delta_epsilon_connectivity} is closely related to \cref{defn:PI_space}.
Indeed, let $X=(X,d,\mu)$ be a complete metric measure space with $\mu\neq 0$.
Then $X$ is a PI space if and only if there are $C\geq 2$
and $\delta,\epsilon\in (0,1)$
such that every pair $(x,y)\in X\times X$ with $x\neq y$ is $(C,\delta,\epsilon)$-connected;
see \cite[Theorem 1.2]{erikssonbique2019_PI_rect}%
\footnote{%
Arguing essentially as in \cite[Lemma 3.4]{erikssonbique2019_PI_rect},
it is not difficult to see that, under the present assumption,
$0<\mu(B)<\infty$ for every ball $B$.}.
We are going to use the `qualitative' analog of this statement.
\begin{defn}[{\cite[Definition 5.5]{erikssonbique2019_PI_rect}}]\label{defn:asy_well_conn}
Let $X=(X,d,\mu)$ be a metric measure space.
We say that $X$ is \emph{asymptotically well-connected} if for $\delta\in (0,1)$
and $\mu$-a.e.\ $x\in X$ there are $C_x\geq 1$, $\epsilon_x\in (0,1)$, and $r_x>0$,
such that, for $y\in B(x,r_x)\setminus\{x\}$, the pair $(x,y)$ is $(C_x,\delta,\epsilon_x)$-connected in $X$.
\end{defn}
\begin{rmk}
\cref{defn:C_delta_epsilon_connectivity} differs slightly from \cite[Definition 3.1]{erikssonbique2019_PI_rect},
but is equivalent if the measure $\mu$ is doubling
(which is the only case we need and, moreover, is implied by $(C,\delta,\epsilon)$-connectivity
for $C\geq 2$ and $\delta,\epsilon\in (0,1)$; see \cite[Lemma 3.4]{erikssonbique2019_PI_rect}).
Similarly, also \cref{defn:asy_well_conn} differs from \cite[Definition 5.5]{erikssonbique2019_PI_rect},
but is equivalent if $\mu$ is asymptotically doubling.
\end{rmk}
In \cite{erikssonbique2019_PI_rect},
it is proven that if $X=(X,d,\mu)$ is an asymptotically well-connected
metric measure space with $\mu$ asymptotically doubling,
then it is PI rectifiable \cite[Theorem 5.6]{erikssonbique2019_PI_rect}%
\footnote{Our definition of metric measure space is slightly more general then the one adopted in \cite{erikssonbique2019_PI_rect}.
However, to apply \cite[Theorem 5.6]{erikssonbique2019_PI_rect}, it is enough to cover $\mu$-almost all of $X$
with countably many compact sets $(K_i)_i$, observe that the assumptions on $X$ are also
satisfied by $(K_i,d,\mu)$ and that the latter is a metric measure space in the sense of \cite{erikssonbique2019_PI_rect}.}.
Contrary to what claimed in \cite[Theorem 5.6]{erikssonbique2019_PI_rect}, the reverse implication is however not true.
Indeed, one might verify that if $(X,d,\mu)$ is asymptotically well-connected and $\mu$ is asymptotically doubling,
then $\mu$ vanishes on porous sets,
a property which need not hold on PI rectifiable spaces (as trivial examples in the plane show).
The following is an amended statement.
\begin{thm}[{\cite[Theorem 5.6]{erikssonbique2019_PI_rect}}]\label{thm:EB_PI_rect_asy_well}
Let $X=(X,d,\mu)$ be a metric measure space.
Then $X$ is PI rectifiable if and only if $\mu$ is asymptotically doubling
and there are countably many $\mu$-measurable sets $(E_i)_i$ such that
$\mu(X\setminus\bigcup_i E_i)=0$ and $(E_i,d,\mu)$ is asymptotically well-connected.
\end{thm}
The precise connection between PI rectifiability and differentiability is given by the following theorem.
\begin{thm}[{\cite[Theorem 1.5]{cheeger_kleiner_PI_RNP_LDS}, \cite[Lemma 3.4]{bate_li_RNP_LDS}, \cite[Theorem 1.1]{erikssonbique2019_PI_rect}}]
\label{thm:PI_rect_RNP_LDS_asy_well_conn}
Let $X=(X,d,\mu)$ be a metric measure space.
Then the following are equivalent:
\begin{itemize}
\item
$X$ is asymptotically well-connected and $\mu$ asymptotically doubling;
\item
$X$ is PI rectifiable and $\mu$ vanishes on porous sets;
\item
$X$ is a $Y$-LDS for every Banach space $Y$ with RNP.
\end{itemize}
\end{thm}
We are going to need the following.
\begin{lemma}\label{lemmma:asy_well_conn_Borel_const}
Let $X=(X,d,\mu)$ be a metric measure space.
Suppose $\mu$ is asymptotically doubling and $X$ asymptotically well-connected.
Then, for $\delta\in(0,1)$, the parameters $C_x\geq 1,\epsilon_x\in (0,1), r_x>0$ (appearing in the definition of asymptotically well-connected)
may be chosen such that $x\mapsto (C_x,\epsilon_x,r_x)$ is Borel measurable.
\end{lemma}
\begin{proof}
Assume first $\mu(X)<\infty$.
For $\mu$-measurable $E\subseteq X$, $C\geq 1$, $r>0$, $\delta\in(0,1)$, and $x,y\in X$, define
\begin{align*}
\rho_{E,C,\delta}(x,y)&:=\inf\left\{\max(\var\gamma/C,\gap\gamma/\delta)\colon \gamma\text{ is curve fragment from }x\text{ to }y\text{ avoiding }E\right\}, \\
\kappa_{E,C,r,\delta}(x)&:=\sup_{y\in U(x,r)}\frac{\rho_{E,C,\delta}(x,y)}{r}, \\
\theta_{E,C,r}(x)&:=
\left\{
\begin{array}{cc}
	\frac{\mu(B(x,Cr)\cap E)}{\mu(U(x,2Cr))},& x\in\spt\mu, \\
	0, & \text{otherwise}
\end{array}
\right.
.
\end{align*}
Concatenating curve fragments, it follows that $\rho_{E,C,r}$ satisfies triangle inequality.
Also,
considering
the trivial curve fragment $\gamma\colon\{0,1\}\to X$, $\gamma(0):=x$, $\gamma(1):=y$,
we see that $\rho_{E,C,\delta}(x,y)\leq d(x,y)/\delta$.
In particular, $\rho_{E,C,\delta}(\cdot,y)\colon X\to (0,\infty)$ is continuous for each $y\in X$,
and so $\kappa_{E,C,r,\delta}$ is lower semi-continuous.
Since $\theta_{E,C,r}\colon X\to [0,\infty)$ is upper semi-continuous,
the map $\tau_{C,R,\epsilon,\delta}\colon X\to (0,\infty]$
\begin{equation*}
\tau_{C,R,\epsilon,\delta}(x):=\sup_{E}\sup_{0<r<R}\min(\epsilon/\theta_{E,C,r}(x),\kappa_{E,C,r,\delta}(x)), \qquad x\in X,
\end{equation*}
is lower semi-continuous and hence Borel measurable.
The supremum is taken over $\mu$-measurable sets $E\subseteq X$ and
we interpret $1/0$ as $\infty$. \par
Let
$C\geq 1$, $R>0$, $\epsilon,\delta\in (0,1)$,
$x\in \{\tau_{C,R,\epsilon,\delta}<1\}$
and $y\in U(x,R)$, $y\neq x$.
We claim that that $(x,y)$ is
$(2C,\delta',\epsilon)$-connected in $(X,d,\mu)$
for all $\delta<\delta'<1$.
Let $E\subseteq X$ be $\mu$-measurable,
and let $\alpha\in (1,2)$ with $r:=\alpha d(x,y)<R$ and $\alpha\delta\leq\delta'$.
Recall that
$\min(\epsilon/\theta_{E,C,r}(x),\kappa_{E,C,r,\delta}(x))<1$.
If $\theta_{E,C,r}(x)>\epsilon$, then
\begin{equation*}
\mu(B(x,2Cd(x,y))\cap E)\geq \mu(B(x,Cr)\cap E)
>\epsilon \mu(U(x,2Cr))
\geq
\epsilon\mu(B(x,2Cd(x,y)))
\end{equation*}
and there is nothing prove.
If $\kappa_{E,C,r,\delta}(x)<1$, then there is a curve fragment $\gamma$
from $x$ to $y$, avoiding $E$, satisfying
\begin{equation*}
	\var\gamma<Cr\leq 2Cd(x,y) \qquad\text{and}\qquad \gap\gamma<\delta r\leq\delta'd(x,y).
\end{equation*}
This proves the claim. \par
To conclude the proof of the lemma (under the additional assumption $\mu(X)<\infty$),
it remains to show that, for $\delta\in(0,1)$,
the set
\begin{equation*}
	A:=\bigcap_{\substack{C\in [1,\infty)\cap\Q \\ R\in (0,\infty)\cap\Q \\ \epsilon\in (0,1)\cap\Q}} \{\tau_{C,R,\epsilon,\delta}\geq 1\}
\end{equation*}
is $\mu$-null.
Let $x\in A\cap \spt\mu$ such that $\limsup_{r\rightarrow 0}\frac{\mu(B(x,2r))}{\mu(B(x,r))}<\infty$
and let $M_x,R_x>0$ be such that $\mu(B(x,2r))\leq M_x\mu(B(x,r))$ for $0<r<R_x$.
By definition of $\tau_{C,R,\epsilon,\delta}(x)$ and $A$,
we see that for every rational
$C\geq 1$, $R>0$, $\epsilon\in(0,1)$,
there are $0<r<R$ and a $\mu$-measurable set $E\subseteq X$ such that
$\min(\epsilon/\theta_{E,C,r}(x),\kappa_{E,C,r,\delta}(x))>1/2$.
That is,
\begin{equation*}
\mu(E\cap B(x,Cr))<2\epsilon\mu(U(x,2Cr))
\end{equation*}
and there is $y\in U(x,r)$ such that $\rho_{E,C,r}(x,y)>r/2$, i.e.\
\begin{equation}\label{eq:meas_coeff_PI_rect_1}
	\gap\gamma>\delta r/2> (\delta/2)d(x,y) \qquad \text{or} \qquad \var\gamma>Cr/2>(C/2)d(x,y)
\end{equation}
for all curve fragments $\gamma$ from $x$ to $y$ which avoid $E$.
Since $\rho_{E,C,r}(x,y)\leq d(x,y)/\delta$, we have $\delta r/2<d(x,y)< r$
and therefore
\begin{equation}\label{eq:meas_coeff_PI_rect_2}
\begin{aligned}
\mu(E\cap B(x,(C/2)d(x,y)))
&\leq \mu(E\cap B(x,Cr))
< 2\epsilon \mu(B(x,2Cr))
\leq 2\epsilon \mu(B(x,4Cd(x,y)/\delta)) \\
&\leq C_1(M_x,\delta)\epsilon\mu(B(x,(C/2)d(x,y))),
\end{aligned}
\end{equation}
provided $R<C_2(R_x,\delta,C)$.
Here $C_1(M_x,\delta)$ and $C_2(R_x,\delta,C)$ are constants depending only on $M_x,\delta$
and $R_x,\delta, C$, respectively.
Since we may choose $C,R,\epsilon$, \cref{eq:meas_coeff_PI_rect_1,eq:meas_coeff_PI_rect_2}
show that there are no (rational or irrational)
$R>0, C\geq 1, \epsilon\in (0,1)$
such that $(x,y)$ is $(C,\delta/2,\epsilon)$-connected for all $y\in B(x,R)$.
Since $(X,d,\mu)$ is asymptotically well-connected and $\mu$ asymptotically doubling,
we conclude that $\mu(A)=0$. \par
If $\mu(X)=\infty$, let $(V_i)_i$ be an open cover of $X$ of sets of finite $\mu$-measure and
observe that $(V_i,d,\mu)$ satisfies the assumption of the previous step.
The thesis follows by modifying the measurable parameters on $V_i$ in the obvious way.
\end{proof}
\subsection{Ahlfors-David regular spaces and Borel-Cantelli-type lemmas}\label{subsec:Borel_Cantelli}
For a metric space $X$, $s\in [0,\infty)$, and a set $E\subseteq X$, the
\emph{$s$-dimensional Hausdorff measure of $E$} is defined as
\begin{equation*}
	\Haus^s(E):=\sup_{\delta>0}\inf\left\{\sum_{i\in\N}\diam(E_i)^s\colon \diam(E_i)<\delta \text{ and } E\subseteq\bigcup_{i\in\N}E_i\subseteq X\right\};
\end{equation*}
the map $E\mapsto\Haus^s(E)$ is a Borel regular measure on $X$;
see \cite[Section 2.10]{Federer_GMT_1969}.
\begin{rmk}\label{rmk:ADR_and_Haus}
An application of the $5r$-Vitali covering lemma (see e.g.\ \cite[$5B$-covering lemma]{HKST_Sobolev_on_metric_measure_spaces})
shows that if $\mu$ is an $s$-ADR measure on $X$ (see \cref{eq:ADR_meas}),
then $\Haus^s\sim \mu$ on Borel sets.
In particular, $\Haus^s$ is also $s$-ADR.
\end{rmk}
Let $X$ be a metric space.
We say that $X$ is $s$-ADR if $\Haus^s$ is $s$-ADR or, equivalently,
if there is an $s$-ADR measure on $X$.
For $\delta>0$ and $\calN\subseteq X$,
we say that $\calN$ is a \emph{$\delta$-net} if $B(\calN,\delta)=X$,
while it is \emph{$\delta$-separated} if $d(x,y)\geq \delta$ for $x\neq y\in\calN$
and \emph{strictly $\delta$-separated} if $d(x,y)>\delta$ for $x\neq y\in\calN$. \par
The main results of this subsection are \cref{lemma:1stBorelCantelli,lemma:2ndBorelCantelli}.
\begin{lemma}[First Borel-Cantelli lemma in ADR spaces]\label{lemma:1stBorelCantelli}
Let $X$ be an $s$-ADR metric space,
$C\geq 1$, and 
let $(\delta_i)\subseteq (0,\infty)$ be a sequence with $\delta_i\rightarrow 0$.
For $i\in\N$,
let $\calN_i\subseteq X$ be a set
satisfying
\begin{equation*}
\#(B(x,\delta_i/C)\cap\calN_i)\leq C, \qquad x\in\calN_i.
\end{equation*}
Then, for $(\alpha_i)\subseteq[0,\infty)$, $\Haus^s(\limsup_i B(\calN_i, \alpha_i\delta_i))=0$ whenever $\sum_i\alpha_i^s<\infty$.
\end{lemma}
\begin{lemma}[Second Borel-Cantelli lemma in ADR spaces]\label{lemma:2ndBorelCantelli}
Let $X$ be an $s$-ADR metric space, $C\geq 1$,
and let $(\delta_i)\subseteq (0,\infty)$ be a sequence with $\delta_i\rightarrow 0$.
For $i\in\N$,
let $\calN_i\subseteq X$ be a $C\delta_i$-net in $X$.
Let $(\alpha_i)\subseteq [0,\infty)$ and suppose that for each $i,j\in\N$
with $i<j$ (at least) one of the following holds:
\begin{itemize}
	\item $\alpha_i\delta_i\geq C^{-1}\delta_j$;
	\item $d(\calN_i,\calN_j)\geq C^{-1}\min(\delta_i,\delta_j)$.
\end{itemize}
Then $\Haus^s(X\setminus \limsup_i B(\calN_i,\alpha_i\delta_i))=0$ whenever $\sum_i\alpha_i^s=\infty$.
\end{lemma}
The proofs of \cref{lemma:1stBorelCantelli,lemma:2ndBorelCantelli} consist of simple estimates based on Ahlfors-David regularity
and an application of the corresponding Borel-Cantelli lemma.
For both,
the basic intuition is that a separated net
approximates the underlying $s$-ADR measure,
therefore its $\alpha$-enlargement in an $R$-ball
should have volume $\sim (R\alpha)^s$.
For \cref{lemma:2ndBorelCantelli}, the separation condition ensures that one net is diffused enough
w.r.t.\ the other, providing enough independence for a Borel-Cantelli-type result. \par
We will need the following refinement of the classical second Borel-Cantelli lemma.
A short proof can be found in \cite{yan_Borel_Cantelli_2006}.
\begin{thm}[{\cite[Theorem 2.1]{petrov_Borel_Cantelli_2002}}]\label{thm:Petrov_BL}
Let $(\Omega,\calF,\Prob)$ be a probability space and $(A_i)\subseteq\calF$ a sequence of events.
Suppose there are $C\geq 1$ and $n\in\N$ such that
$\Prob(A_i\cap A_j)\leq C \Prob(A_i)\Prob(A_j)$ for $j>i\geq n$.
Then $\Prob(\limsup_i A_i)\geq 1/C$ whenever $\sum_i\Prob(A_i)=\infty$.
\end{thm}
We now require the aforementioned estimates, which we collect in a few lemmas.
Throughout this subsection, the implicit constants in $\lesssim,\gtrsim$, and $\sim$, depend on $s$, the ADR constant
of $\Haus^s$, and any additional parameters indicated as subscript.
\begin{lemma}\label{lemma:ADR_sep_basic_estimate}
Let $X$ be an $s$-ADR metric space, $x_0\in X$, $R,\delta>0$,
and
let $\calN\subseteq X$ be a $\delta$-separated set.
Then $\#(\calN\cap B(x_0,R))\lesssim\max\{(R/\delta)^s,1\}$.
\end{lemma}
\begin{proof}
If $\delta>2R$, then $\#(\calN\cap B(x_0,R))\leq 1$.
Suppose now $\delta\leq 2R$.
Since
$\{B(x,\delta/3)\colon x\in\calN\}$ is pairwise disjoint,
we have
\begin{equation*}
\#(\calN\cap B(x_0,R))\delta^s\lesssim
\Haus^s\left(\bigcup\{B(x,\delta/3)\colon x\in\calN\cap B(x_0,R)\}\right)
\leq\Haus^s(B(x_0,R+2R/3))\lesssim R^s.
\end{equation*}
\end{proof}
\begin{proof}[Proof of \cref{lemma:1stBorelCantelli}]
Possibly enlarging $C$, we can assume $0<\alpha_i\leq C$ for all $i$.
Set $E:=\limsup_i B(\calN_i,\alpha_i\delta_i)$.
Let $\calN'_i$ be a maximal $C^{-1}\delta_i$-separated subset of $\calN_i$ and observe that $\calN_i\subseteq U(\calN'_i,\delta_i/C)$.
Fix $x_0\in X$, $0<R<\diam X$, and let $n\in\N$ be such that $\delta_i\leq R$ for $i\geq n$.
Then
\begin{equation}\label{eq:1stBC_eq_1}
B(x_0,R)\cap B(\calN_i,\alpha_i\delta_i)
\subseteq \bigcup_{x\in\calN_i\cap B(x_0,3CR)}B(x,2\alpha_i\delta_i)
\subseteq \bigcup_{x'\in \calN'_i\cap B(x_0,4CR)}\bigcup_{x\in B(x',\delta_i/C)\cap\calN_i} B(x,2\alpha_i\delta_i),
\end{equation}
which implies $\Haus^s(B(x_0,R)\cap B(\calN_i,\alpha_i\delta_i))\lesssim_C \#(\calN_i'\cap B(x_0,4CR)) (\alpha_i\delta_i)^s$.
Since
$\calN'_i\cap B(x_0,4CR)$ is $C^{-1}\delta_i$-separated and $\delta_i\leq R$,
\cref{lemma:ADR_sep_basic_estimate} and \cref{eq:1stBC_eq_1} imply
\begin{equation*}
	\Haus^s(B(x_0,R)\cap B(\calN_i,\alpha_i\delta_i))\lesssim_C (R\alpha_i)^s, \qquad i\geq n.
\end{equation*}
From the above equation, arguing as in the proof of the classical first Borel-Cantelli lemma
(or applying it to the probability measure $\Haus^s\mres{B(x_0,R)}/\Haus^s(B(x_0,R))$) we conclude
that $\Haus^s(E\cap B(x_0,R))=0$.
Since $x_0$ and $R$ were arbitrary, this concludes the proof.
\end{proof}
\begin{lemma}\label{lemma:BL_easy_lemma}
Let $X$ be an $s$-ADR metric space, $x_0\in X$, $R,\delta>0$, and $C\geq 1$.
Suppose $8C\delta\leq R<\diam X$ and let $\calN\subseteq X$ be a $C\delta$-net.
Then
\begin{equation*}
\Haus^s(B(x_0,R)\cap B(\calN,\alpha\delta))\gtrsim_C (R\alpha)^s,
\end{equation*}
for $\alpha\in [0,C]$.
\end{lemma}
\begin{proof}
Let $\calN'$ be a maximal $C\delta$-separated subset of $\calN$ and
set
$E:= B(x_0,R)\cap B(\calN,\alpha\delta)$,
$\calN_0:=\calN'\cap B(x_0,R-\alpha\delta)$.
Since
$\{B(x,\alpha\delta/3):x\in\calN_0\}$ is a pairwise disjoint collection of balls
contained in $E$, we have
\begin{equation}\label{eq:BL_easy_lemma_eq_1}
\Haus^s(E)\gtrsim \#\calN_0(\alpha\delta)^s.
\end{equation}
Since $\calN'$ is maximal, we have $B(\calN',2C\delta)=X$.
Then, for $y\in B(x_0,R/2)$, there is $x\in\calN'$ with $d(x,y)\leq 3C\delta$ and
hence $d(x_0,x)\leq R/2+3C\delta\leq R-\alpha\delta$.
That is, 
$\{B(x,3C\delta)\colon x\in\calN_0\}$ covers $B(x_0,R/2)$,
proving
\begin{equation}\label{eq:BL_easy_lemma_eq_2}
\#\calN_0\delta^s\gtrsim_C\Haus^s\left(\bigcup_{x\in\calN_0}B(x,3C\delta)\right)\gtrsim R^s.
\end{equation}
\Cref{eq:BL_easy_lemma_eq_1,eq:BL_easy_lemma_eq_2} conclude the proof.
\end{proof}

\begin{lemma}\label{lemma:BL2_l1}
Let $X$ be an $s$-ADR metric space, $C\geq 1$, $\delta,\delta'>0$, $x_0\in X$, and 
suppose $C^{-1}\delta\leq R <\diam X$.
Let $\calN,\calN'\subseteq X$ be $C^{-1}\delta$ and $C^{-1}\delta'$-separated, respectively.
Then, for $\alpha,\alpha'\in (0,C]$ with $\alpha\delta \geq C^{-1}\delta'$
it holds
\begin{equation*}
\Haus^s(B(x_0,R)\cap B(\calN,\alpha\delta)\cap B(\calN',\alpha'\delta'))
\lesssim_C (R\alpha \alpha')^s.
\end{equation*}
\end{lemma}
\begin{proof}
Set
\begin{equation*}
\begin{aligned}
E
&:=B(x_0,R)\cap B(\calN,\alpha\delta)\cap B(\calN',\alpha'\delta'), \\
\calN_0
&:=\calN\cap B(x_0,R+2\alpha\delta), \\
\calN_x'
&:=\calN'\cap B(x,2\alpha\delta+2\alpha'\delta'), \qquad x\in\calN_0.
\end{aligned}
\end{equation*}
From the definitions, we have
\begin{equation*}
\begin{aligned}
B(x_0,R)\cap B(\calN,\alpha\delta)
&\subseteq \bigcup\{B(x,2\alpha\delta)\colon x\in\calN_0\}, \\
B(x,2\alpha\delta)\cap B(\calN',\alpha'\delta')
&\subseteq\bigcup\{B(x',2\alpha'\delta')\colon x'\in\calN'_x\}, \qquad x\in\calN_0,
\end{aligned}
\end{equation*}
and therefore
\begin{equation}\label{Borel_Cantelli_2_l2_eq1}
E\subseteq \bigcup_{x\in\calN_0}\bigcup_{x'\in\calN_x'} B(x',2\alpha'\delta').
\end{equation}
Since,
for $x\in\calN_0$,
the set $\calN_x'\subseteq B(x,2\alpha\delta+2\alpha'\delta')$ is $C^{-1}\delta'$-separated,
\cref{lemma:ADR_sep_basic_estimate} implies that
its cardinality is at most
$\lesssim_C\max\{(\alpha\delta+\alpha'\delta')^s(\delta')^{-s},1\}$.
By assumption, $\alpha\delta+\alpha'\delta'\sim_C\alpha\delta$, $(\alpha\delta)/\delta'\gtrsim_C 1$, and therefore
\begin{equation}\label{Borel_Cantelli_2_l2_eq2}
	\#\calN_x'\lesssim_C \left(\frac{\alpha\delta}{\delta'}\right)^s, \qquad x\in\calN_0.
\end{equation}
Similarly, $\calN_0$ is a $C^{-1}\delta$-separated set in $B(x_0,R+2\alpha\delta)$
and $R+2\alpha\delta\sim_C R\gtrsim_C\delta$, yielding
$\#\calN_0\lesssim_C (R/\delta)^s$.
Finally, \cref{Borel_Cantelli_2_l2_eq1,Borel_Cantelli_2_l2_eq2} give
\begin{equation*}
\Haus^s(E)\lesssim\#\calN_0 \left(\max_{x\in\calN_0}\#\calN'_x\right)(\alpha'\delta')^s\lesssim_C(R\alpha\alpha')^s.
\end{equation*}
\end{proof}
\begin{lemma}\label{lemma:BL2_l2}
Let $X$ be an $s$-ADR metric space, $C\geq 1$, $\delta,\delta'>0$, $x_0\in X$,
and suppose $C^{-1}\max(\delta,\delta')\leq R <\diam X$.
Let $\calN,\calN'\subseteq X$ be $C^{-1}\delta$ and $C^{-1}\delta'$-separated, respectively,
and assume $d(\calN,\calN')\geq C^{-1}\min(\delta,\delta')$.
Then, for $\alpha,\alpha'\in [0,C]$
it holds
\begin{equation*}
\Haus^s(B(x_0,R)\cap B(\calN,\alpha\delta)\cap B(\calN',\alpha'\delta'))
\lesssim_C (R\alpha \alpha')^s.
\end{equation*}
\end{lemma}
\begin{proof}
Set
$E:=B(x_0,R)\cap B(\calN,\alpha\delta)\cap B(\calN',\alpha'\delta')$
and
assume w.l.o.g.\ $\delta'\leq\delta$ and $\alpha,\alpha'>0$. \par
Suppose first
$\max\{\alpha',(\alpha\delta)/\delta'\}<1/4C$.
Then $B(x,2\alpha\delta)\cap B(x',2\alpha'\delta')=\varnothing$ for $x\in\calN$ and $x'\in\calN'$,
for otherwise
we would find $x,x'$ with $d(x,x')\leq 2\alpha\delta+2\alpha'\delta'<(4/4C)\delta'=C^{-1}\delta'$,
contradicting $d(\calN,\calN')\geq C^{-1}\delta'$.
Since
\begin{equation*}
	E\subseteq B(x_0,R)\cap \bigcup_{\substack{x\in\calN \\ x'\in\calN'}}B(x,2\alpha\delta)\cap B(x',2\alpha'\delta'),
\end{equation*}
we have $E=\varnothing$ and the thesis is trivially satisfied. \par
We can now assume $\max\{\alpha',(\alpha\delta)/\delta'\}\geq 1/4C$.
Since \cref{lemma:BL2_l1} covers the case $\alpha\delta\geq (1/4C)\delta'$,
it only remains to consider the case $\alpha\delta< (1/4C)\delta'$ and $\alpha'\geq (1/4C)$.
Set $\calN_0:=\calN\cap B(x_0,R+2\alpha\delta)$ and observe that
it is a $C^{-1}\delta$-separated subset of $B(x_0,R+2C^2R)$.
Since $\delta\lesssim_C R$,
by \cref{lemma:ADR_sep_basic_estimate}
we conclude $\#\calN_0\lesssim_C (R/\delta)^s$.
The inclusion $E\subseteq \bigcup_{x\in\calN_0}B(x,2\alpha\delta)$
then implies
\begin{equation*}
\Haus^s(E)\lesssim_C (R/\delta)^s(\alpha\delta)^s\sim_C (R\alpha\alpha')^s,
\end{equation*}
where in the last equality we have used $\alpha'\sim_C 1$.
\end{proof}
\begin{proof}[Proof of \cref{lemma:2ndBorelCantelli}]
Replacing $\calN_i$ with a maximal $C^{-1}\delta_i$-separated subset and increasing $C\geq 1$ we can
assume $\calN_i$ to be a $C^{-1}\delta_i$-separated $C\delta_i$-net.
We can also assume $\alpha_i<C$ for $i\in\N$,
for otherwise the thesis holds trivially. \par
Set $E:=\limsup_iB(\calN_i,\alpha_i\delta_i)$, fix $x_0\in X$, $0<R<\diam X$,
and
let $n\in\N$ be such that $8C\delta_i\leq R$ for $i\geq n$.
Observe that for $n\leq i<j$,
depending on the separation condition holding for $i,j$, we can apply
either \cref{lemma:BL2_l1} or \cref{lemma:BL2_l2} to deduce
\begin{equation*}
\Haus^s(B(x_0,R)\cap B(\calN_i,\alpha_i\delta_i)\cap B(\calN_j,\alpha_j\delta_j))\lesssim_C (R\alpha_i\alpha_j)^s.
\end{equation*}
By \cref{lemma:BL_easy_lemma}, we also have $\Haus^s(B(x_0,R)\cap B(\calN_i,\alpha_i\delta_i))\gtrsim_C (R\alpha_i)^s$ for $i\geq n$.
Hence, setting $\Prob:=\Haus^s\mres{B(x_0,R)}/\Haus^s(B(x_0,R))$, the above shows that there is a constant $C_1\geq 1$
independent of $x_0$ and $R$ such that
\begin{equation*}
\Prob(B(\calN_i,\alpha_i\delta_i)\cap B(\calN_j,\alpha_j\delta_j))\leq C_1 \Prob(B(\calN_i,\alpha_i\delta_i))\Prob(B(\calN_j,\alpha_j\delta_j))
\end{equation*}
for $j>i\geq n$.
Since $\sum_i\Prob(B(\calN_i,\alpha_i\delta_i))\gtrsim_C\sum_{i\geq n}\alpha_i^s=\infty$,
\cref{thm:Petrov_BL} gives $\Prob(E)\geq 1/C_1$.
Recalling the definition of $\Prob$ and the fact that $x_0$ and $0<R<\diam X$ were arbitrary, we have
\begin{equation*}
\Haus^s(B(x_0,R)\setminus E)\leq (1-1/C_1)\Haus^s(B(x_0,R)),
\end{equation*}
for all $x_0$ and $R$.
Since $(1-1/C_1)<1$, the above inequality implies that $X\setminus E$ does not have any
Lebesgue
density point w.r.t.\ the doubling measure $\Haus^s$.
\end{proof}

\section{Shortcut metric spaces}\label{sec:metric_sp_with_shortcuts}
Given a metric space $(X,d)$ and a sequence $\eta_i \in (0,1]$, we now construct a metric space $(X,d_\eta)$ by allowing \shortcuts\ in the metric at scales $\delta_i \to 0$;
the cost of each \jumpset\ is $\eta_i d$.
The notion of \shortcuts\ is inspired by \cite{ledonne_li_rajala_shortcuts_heisenberg};
see \cref{sec:comparison_LDLR} for a comparison of the definitions.
\begin{defn}[Metric space with \shortcuts]\label{defn:shortcuts}
Let $(X,d)$ be a metric space.
Suppose there are a sequence $(\delta_i)\subseteq (0,\infty)$ with $\delta_i\rightarrow 0$, non-empty families $(\calJ_i)$ of disjoint subsets of $X$,
and constants $a,a_0,b,M>0$,
such that:
\begin{itemize}
\item $a_0\delta_i\leq d(x,y)\leq a\delta_i $ for $x,y\in S\in\calJ_i$, $x\neq y$;
\item $d(S,S')\geq b\delta_j$ for distinct $S\in\calJ_i$, $S'\in\calJ_j$ with $i\leq j$;
\item for $S\in\calJ_i$, $x,y\notin U(S,b\delta_i)$ and $z,w\in S$ it holds
\begin{equation}\label{eq:shortcut_condition}
d(x,y)\leq d(x,z)+d(w,y);
\end{equation}
\item $2\leq \# S\leq M$ for $S\in\calJ_i$;
\item $B(\cup\calJ_i,M\delta_i)=X$.
\end{itemize}
We then say that $(X,d)$ is a \emph{metric space with \shortcuts} $\{(\calJ_i,\delta_i)\}_i$
and call \emph{\jumpset}\ any $S\in\calJ:=\bigcup_i\calJ_i$.
\end{defn}
If a metric space $X$ has \shortcuts\ $\{(\calJ_i,\delta_i)\}_i$, then $\calJ$ is a collection of disjoint sets.
\begin{rmk}\label{rmk:subshortcuts}
Let $X$ be a metric space with \shortcuts\ $\{(\calJ_i,\delta_i)\}_i$,
$(i_j)\subseteq\N$ strictly increasing and, for $j\in\N$ and $S\in\calJ_{i_j}$,
let $\widehat{S}\subseteq S$ be a subset with at least two elements.
Then
\begin{equation*}
\big(\{\widehat{S}\colon S\in\calJ_{i_j}\}, \delta_{i_j}\big)
\end{equation*}
are \shortcuts\ in $X$.
Hence, it is more interesting to find large families of \shortcuts, although for
our results this does not really matter.
\end{rmk}
\Cref{eq:shortcut_condition}
is the most important condition in \cref{defn:shortcuts}.
Informally, points in $S\in\calJ$ are, in some sense, indistinguishable when 
looked at from afar.
For an illustrative example,
consider a connected graph $G=(V,E)$ containing a complete bipartite graph $K_{2,M}$,
let $S$ be the part of $K_{2,M}$ with $M$ elements,
and endow $G$ with the length distance $d_G$.
Then, for $x\notin U_{d_G}(S,1)=S$ and $z,w\in S$, it holds
\begin{equation*}
d_G(x,z)=d_G(x,w),
\end{equation*}
and \cref{eq:shortcut_condition} follows.
With this example in mind, many of the results of this section should become more transparent. \par
\begin{notation}
Throughout the remainder of this subsection, we consider a fixed a metric space
$X=(X,d)$ having \shortcuts\ $\{(\calJ_i,\delta_i)\}_i$ with constants $a_0,a,b,$ and $M$.
We also fix a sequence $\eta=(\eta_i)_i\subseteq (0,1]$.
\end{notation}
We can distort the distance $d$ in the natural way.
\begin{defn}\label{defn:shortcut_dist}
	For $x,y\in X$, define 
\begin{equation*}
\rho_\eta(x,y)
:=
\left\{
\begin{array}{cc}
	\eta_i d(x,y),& x,y\in S\in\calJ_i \\
	d(x,y),& \text{otherwise}
\end{array}
\right.
.
\end{equation*}
The \emph{\jumpset\ distance} between $x$ and $y$ is
	\begin{equation}
	\label{eq:contracted_distance}
d_\eta(x,y)
:=\inf\left\{\sum_{i=0}^n\rho_\eta(x_{i-1},x_i)\colon n\in\N, x=x_0,\dots,x_n=y\right\}.
\end{equation}
The map $d_\eta\colon X\times X\to [0,\infty)$ is a pseudo-distance.
We will shortly show that it is actually a distance (\cref{prop:single_jump}).
The space $(X,d_\eta)$ is the \emph{\jumpset\ metric space} generated by $(X,d)$ and $\eta$.
\end{defn}

\begin{notation}
To avoid ambiguity, we always use the subscript $\eta$ when a quantity depends on the distance $d_\eta$.
For instance, the Lipschitz constant w.r.t.\ $d_\eta$ of a function $f\colon X\to Y$ will be denoted as $\LIP_\eta(f)$.
\end{notation}
The following lemma is analogous to \cite[Lemma 3.2]{ledonne_li_rajala_shortcuts_heisenberg}.
\begin{lemma}\label{lemma:decreasing_increasing}
Let $x,y\in X$ with $d_\eta(x,y)<d(x,y)$.
Then, for $\epsilon>0$, there are $n\in\N$, $i_1,\dots,i_n\in\N$,
$p_j^-, p_j^+\in S_j\in\calJ_{i_j}$, and $1\leq k\leq n$ such that
\begin{equation}\label{eq:decreasing_increasing_eq_1}
d_\eta(x,y)+\epsilon \geq
d(x,p_1^-) + \rho_\eta(p_1^-,p_1^+) +d(p_1^+,p_2^-) +\cdots +\rho_\eta(p_n^-,p_n^+)+ d(p_n^+,y),
\end{equation}
\begin{equation}\label{eq:decreasing_increasing_eq_2}
	i_1\geq \cdots\geq i_k, \qquad i_k\leq \cdots \leq i_n,
\end{equation}
with $p_1^\pm,\cdots,p_n^\pm$ distinct and $S_j\neq S_{j+1}$ for $1\leq j<n$.
\end{lemma}
\begin{proof}
Assume w.l.o.g.\ $d_\eta(x,y)+\epsilon<d(x,y)$
and let
$m\in\N$, $x=x_0,\dots,x_m=y\in X$ be such that $\sum_{i=1}^m\rho_\eta(x_{i-1},x_i)\leq d_\eta(x,y)+\epsilon$
and
$x_i\neq x_j$ for $i\neq j$.
Observe that, since the total `cost' of $(x_i)_{i=0}^m$ w.r.t\ $\rho_\eta$ is less than $d(x,y)$,
there are consecutive points which belong to the same \jumpset. \par
By triangle inequality of $d$, we may discard terms from $(x_i)_{i=0}^m$ to ensure that, for $0<i<m$ and $S\in\calJ$,
if $x_{i-1}\notin S$ and $x_i\in S$, then $x_{i+1}\in S$.
When restricted to a \jumpset\ $\rho_\eta$ satisfies triangle inequality,
so we may further assume that no \jumpset\ contains more than two consecutive points of $(x_i)_{i=0}^m$.
Then, possibly adjoining $x$ and $y$ at the beginning and end of $(x_i)_{i=0}^m$,
we find
$n\in\N$, $i_1,\dots,i_n$, $p_j^-,p_j^+\in S_j\in\calJ_{i_j}$
such that $S_j\neq S_{j+1}$ for $1\leq j<n$,
$p_1^\pm,\dots,p_n^\pm$ are distinct,
and $(x_i)_{i=0}^m=(x,p_1^-,\dots,p_n^+,y)$.
\par
We now show by induction on $n\in\N$ that, for any
discrete path $(x,p_1^-,\dots,p_n^+,y)$ as above,
we may discard some terms and obtain one as in the thesis. \par
For $1\leq n\leq 2$ \cref{eq:decreasing_increasing_eq_2} clearly holds,
so we may assume $n\geq 3$.
If \cref{eq:decreasing_increasing_eq_2} holds for no $k$,
it is then not difficult to see that
there is $1<l<n$ such that $i_{l-1}\leq i_l$ and $i_{l+1}<i_l$.
The second point of \cref{defn:shortcuts} implies
$d(p_{l-1}^+,S_l)\geq b\delta_{i_l}$, $d(p_{l+1}^-,S_l)\geq b\delta_{i_l}$
and, from the third, we obtain
\begin{equation*}
d(p_{l-1}^+,p_l^-)+\rho_\eta(p_l^-,p_l^+)+d(p_l^+,p_{l+1}^-)
\geq 
d(p_{l-1}^+,p_{l+1}^-).
\end{equation*}
Hence, discarding $(p_l^-,p_l^+)$ yields a discrete path with no larger `cost' w.r.t.\ $\rho_\eta$.
Then, the first part of the proof shows that from
$(x,p_1^-,\dots,p_{l-1}^+,p_{l+1}^-,\dots,p_n^+,y)$
we may obtain a new discrete path 
$(x,q_1^-,\dots,q_t^+,y)$ with $t\in\N$, $t<n$,
satisfying all conditions of the thesis except possibly for \cref{eq:decreasing_increasing_eq_2}.
Since $t<n$, the induction hypothesis concludes the proof.
\end{proof}
The following marks our point of departure from \cite{ledonne_li_rajala_shortcuts_heisenberg}.
\begin{prop}[Single jump]\label{prop:single_jump}
Let $x,y\in X$ with $d_\eta(x,y)<d(x,y)$.
Then, for $\epsilon>0$, there are $p_-\neq p_+\in S\in\calJ$ such that
\begin{equation}\label{eq:single_jump}
d_\eta(x,y)(1+a/b+\epsilon)\geq d(x,p_-)+\rho_\eta(p_-,p_+)+d(p_+,y).
\end{equation}
In particular, $d_\eta$ is a distance on $X$.
\end{prop}
\begin{proof}
We first prove that, for $\epsilon>0$, if $i_j$, $S_j$, $p_j^\pm$, and $k$ are as in \cref{lemma:decreasing_increasing},
then
\begin{equation*}
d_\eta(x,y)+\epsilon \geq (1+a/b)^{-1}\left(d(x,p_k^-)+\rho_\eta(p_k^-,p_k^+)+d(p_k^+,y)\right).
\end{equation*}
Later, we will show that $d_\eta(x,y)>0$ whenever $x\neq y$, which, combined with the above, immediately gives the thesis. \par
From $S_j\neq S_{j+1}$ and $i_j\geq i_{j+1}$ for $1\leq j<k$, we have
\begin{equation*}
\sum_{j=1}^{k-1}d(p_j^+,p_{j+1}^-)
\geq b\sum_{j=1}^{k-1}\delta_{i_j}
\geq (b/a)\sum_{j=1}^{k-1}d(p_j^-,p_j^+)
\end{equation*}
and similarly $\sum_{j=k}^{n-1}d(p_j^+,p_{j+1}^-)\geq (b/a)\sum_{j=k}^{n-1}d(p_{j+1}^-,p_{j+1}^+)$.
Set $t:=a/(a+b)$
and observe that
\begin{align*}
\sum_{j=1}^{k-1}d(p_j^+,p_{j+1}^-)
&\geq t(b/a)\sum_{j=1}^{k-1}d(p_j^-,p_j^+)+(1-t)\sum_{j=1}^{k-1}d(p_j^+,p_{j+1}^-) \\
&=\frac{b}{a+b}\sum_{j=1}^{k-1}d(p_j^-,p_j^+)+d(p_j^+,p_{j+1}^-)
\geq (1+a/b)^{-1} d(p_1^-,p_k^-)
\end{align*}
and $\sum_{j=k}^{n-1}d(p_j^+,p_{j+1}^-)\geq (1+a/b)^{-1}d(p_{k}^+,p_n^+)$.
Inequality \eqref{eq:decreasing_increasing_eq_1} and the above give
\begin{align*}
d_\eta(x,y)+\epsilon
&\geq d(x,p_1^-)+\sum_{j=1}^{n-1}\left(\rho_\eta(p_j^-,p_j^+)+d(p_j^+,p_{j+1}^-)\right) +\rho_\eta(p_n^-,p_n^+)+d(p_n^+,y) \\
&\geq d(x,p_1^-)+(1+a/b)^{-1}\left(d(p_1^-,p_k^-)+d(p_k^+,p_n^+)\right)+\rho_\eta(p_k^-,p_k^+)+d(p_n^+,y) \\
&\geq (1+a/b)^{-1}\left(d(x,p_k^-)+\rho_\eta(p_k^-,p_k^+)+d(p_k^+,y)\right),
\end{align*}
as claimed. \par
Suppose now there are $x\neq y\in X$ with $d_\eta(x,y)=0$
and let $\epsilon_j\downarrow 0$.
By the first part of the proof,
we find $i_j\in\N$ and $p_j^-\neq p_j^+\in S_j\in \calJ_{i_j}$
such that $d(x,p_j^-)+\rho_\eta(p_j^-,p_j^+)+d(p_j^+,y)\leq \epsilon_j$ for $j\in\N$.
Since $\epsilon_j\rightarrow 0$,
$(i_j)$ is unbounded, and so it has a (not relabelled) divergent subsequence.
But then
\begin{equation*}
	d(x,y)\leq d(x,p_j^-)+d(p_j^-,p_j^+)+d(p_j^+,y)\leq \epsilon_j + a \delta_{i_j}\rightarrow 0,
\end{equation*}
a contradiction.
\end{proof}
%In \cref{prop:single_jump}, when proving that $d_\eta$ is a distance, we actually established the following fact.
\begin{lemma}\label{lemma:same_topology}
For $r>0$ there is $\rho(r)>0$ such that $B_\eta(x,\rho(r))\subseteq U_X(x,r)$ for $x\in X$.
Hence, $d$ and $d_\eta$ induce the same topology.
\end{lemma}
\begin{proof}
Set $C:= 1+2a/b$.
Since $\delta_i\rightarrow 0$ and $\eta_i>0$ for all $i$, for $r>0$ there is $\rho=\rho(r)\in (0,r/2C)$
such that $\rho_\eta(x,y)\leq C \rho$ implies $d(x,y)\leq r/2$, for all $x,y\in X$.
Let $x\in X$, $y\in B_\eta(x,\rho)$, and $p_-\neq p_+\in S\in\calJ_i$ such that $d(x,p_-)+\rho_\eta(p_-,p_+)+d(p_+,y)\leq Cd_\eta(x,y)$.
Since $\rho_\eta(p_-,p_+)\leq C\rho$, it holds $d(p_-,p_+)\leq r/2$ and so
$d(x,y)\leq C\rho + r/2<r$.
\end{proof}
Recall that a metric space is \emph{proper} if closed and bounded subsets are compact.
\begin{lemma}\label{lemma:some_metric_properties}
The following hold:
\begin{itemize}
	\item $X$ is complete if and only if $(X,d_\eta)$ is complete;
	\item $X$ is proper if and only if $(X,d_\eta)$ is proper.
\end{itemize}
\end{lemma}
\begin{proof}
From \cref{lemma:same_topology}, $X$ and $(X,d_\eta)$ have the same Cauchy sequences and topology;
this proves the first point.
For the second point, it is enough to show that sets which are bounded w.r.t.\ $d_\eta$ are also bounded w.r.t.\ $d$.
This follows easily from \cref{prop:single_jump} and $\sup_i\delta_i<\infty$.
\end{proof}
\begin{defn}
For metric space $Z,Y$, we say that a map $f\colon Z\to Y$ is \emph{David-Semmes regular} (or DS-regular)
if
there is $C\geq 1$ such that $\LIP(f)\leq C$ and for every ball $B_Y(y,r)\subseteq Y$
there are $z_1,\dots, z_n\in Z$, $n\leq C$, satisfying $f^{-1}(B_Y(y,r))\subseteq\bigcup_{i=1}^n B_Z(z_i,Cr)$.
\end{defn}
Let $f\colon Z\to Y$ be DS-regular. It follows from the definition that
\begin{equation}\label{eq:DS_reg_maps_and_Haus}
f_{\#}\Haus_Z^s(E)\underset{C,s}{\sim} \Haus_Y^s\mres f(Z)(E)
\end{equation}
for $E\subseteq Y$ and $s\in [0,\infty)$.
Note also that if $A\subseteq Z$ is non-empty, then $f\vert_{A}$ is also DS-regular, with comparable constant.
Moreover, if $\mu$ is an $s$-ADR measure on $Z$ and in addition $f$ is surjective, then
$f_\#\mu$ is $s$-ADR on $Y$.
\begin{defn}
A metric space $Z$ is \emph{(metric) doubling} if there is $N\in\N$
such that every ball can be covered by at most $N$ balls of half the radius.
This is equivalent to the following: For every $\epsilon\in (0,1)$, there is $N(\epsilon)\in\N$
such that if $W\subseteq B_Z(z,r)$ is $\epsilon r$-separated, then $\#W\leq N(\epsilon)$. \par
\end{defn}
For instance, a metric space $Z$ supporting a doubling measure is metric doubling and, if $Z$ is complete,
the converse is also true \cite{Luukkainen_Saksman_complete_doubling_has_doubling_measure}. \par
One may verify that a DS-regular image of a doubling metric space is also metric doubling.
\begin{lemma}\label{lemma:doubling_implies_DS_reg}
Suppose $X$ is metric doubling. Then the identity map $(X,d)\to (X,d_\eta)$ is David-Semmes regular.
\end{lemma}
\begin{proof}
Set $C:=1+2a/b$.
Since $d_\eta\leq d$, the identity map $(X,d)\to (X,d_\eta)$ is $1$-Lipschitz.
Let $x\in X$ and $r>0$.
Set $J:=\{S\in\calJ_i\colon a\delta_i\geq r, S\cap B(x,Cr)\neq\varnothing\}$.
Then
\begin{equation*}
	B_\eta(x,r)\subseteq B(x,(C+1)r)\cup\bigcup_{z\in S\in J}B(z,Cr).
\end{equation*}
Indeed, let $y\in B_\eta(x,r)$. If $d_\eta(x,y)=d(x,y)$, then $y\in B(x,(C+1)r)$.
Otherwise, let $p_-,p_+\in S\in\calJ$ be such that $d(x,p_-)+\rho_\eta(p_-,p_+)+d(p_+,y)\leq Cr$.
Observe that $p_-\in B(x,Cr)$, so $S\cap B(x,Cr)\neq\varnothing$.
If $S\in J$, then $y\in B(p_+,Cr)$.
Otherwise, $d(x,y)\leq Cr+r$ and $y\in B(x,(C+1)r)$, as required. \par
Since $\#S\leq M$ for each $S\in\calJ$ (see \cref{defn:shortcuts}),
to conclude the proof of David-Semmes regularity it only remains to show that the cardinality of $J$ can be bounded independently of $x$ and $r$.
For each $S\in J$, pick $z_S\in B(x,Cr)\cap S$ and set $Z:=\{z_S\colon S\in J\}$. Observe that $\# Z = \# J$.
From the separation condition in \cref{defn:shortcuts} and the definition of $J$, we see that $d(z,z')\geq (b/a)r$ for $z,z'\in Z$, $z\neq z'$.
That is, $Z$ is a $(b/a)r$-separated set contained in $B(x,Cr)$. Since $X$ is metric doubling, there is a constant $N>0$,
depending only on the doubling constant of $X$ and $Ca/b$, such that $\#J\leq N$.
\end{proof}
The following variant of \cref{lemma:doubling_implies_DS_reg} will be used in \cref{subsec:LDS_squashed}.
\begin{lemma}\label{lemma:DS_reg_no_doubling}
Suppose there is a measure $\mu$ on $X$ such that $(X,d,\mu)$ is a metric measure space
and $\mu$ vanishes on sets porous w.r.t.\ $d$. \par
Then there are countably many $\mu$-measurable sets $(E_i)_i$ with $\mu(X\setminus\bigcup_i E_i)=0$ such that 
the identity map $(E_i,d)\to (E_i,d_\eta)$ is David-Semmes regular for each $i$ (and $\eta$).
\end{lemma}
\begin{proof}
Set $C:=1+2a/b$ and $\epsilon:=b/4Ca$.
Since $\mu$ vanishes on porous sets, by \cite[Theorem 3.6(ii),(iv)]{Mera_Moran_Preiss_Zajivcek_porosity_sigma_por_and_meas},
$\mu$ is asymptotically doubling.
Then, by \cite[Lemma 8.3]{bate_str_of_measures_2015}, we find countably many metric doubling $\mu$-measurable sets $E_j$
with $\mu(X\setminus\bigcup_j E_j)=0$.
For $R>0$ and $j$, define $g_{j,R}\colon E_j\to [0,1]$ as
\begin{equation*}
g_{j,R}(x):=\sup_{y\in U(x,R)\setminus\{x\}}\frac{d(E_j,y)}{d(x,y)},\qquad x\in E_j,
\end{equation*}
and observe that it is lower semi-continuous on $E_j$.
Also, by \cref{lemma:ptwise_Lip_and_porosity}, we see that $g_{j,R}(x)\rightarrow 0$ as $R\rightarrow 0$ for every $j$ and $\mu$-a.e.\ $x\in E_j$.
Hence, $E_{j,k}:=\{x\in E_j\colon g_{j,1/k}(x)\leq \epsilon/2\}$ defines a countable collection of $\mu$-measurable sets which cover $\mu$-almost all of $X$.
Since $E_{j,k}$ is doubling in $X$, it is in particular separable, and
we may find countably many $\mu$-measurable sets $(E_{j,k,l})_l$ such that $E_{j,k}=\bigcup_{l}E_{j,k,l}$
and $\diam_\eta E_{j,k,l}\leq \diam_X E_{j,k,l}\leq 1/Ck$ for all $l$. \par
We claim that the identity $(E_{j,k,l},d)\to (E_{j,k,l},d_\eta)$ is DS-regular.
We argue as in \cref{lemma:doubling_implies_DS_reg}, with minor adaptation.
Let $x\in E_{j,k,l}$ and $r>0$.
If $r\geq 1/Ck$, then $B_\eta(x,r)\cap E_{j,k,l}=E_{j,k,l}=B_X(x,r)\cap E_{j,k,l}$ and there is nothing to prove.
Suppose $0<r<1/Ck$, define $J\subseteq\calJ$ as in the proof of \cref{lemma:doubling_implies_DS_reg},
and recall that
\begin{equation*}
B_\eta(x,r)\subseteq B_X(x,(C+1)r)\cup\bigcup_{z\in S\in J} B_X(z,Cr).
\end{equation*}
Observe that to conclude the proof of the claim (and the lemma), we need only to show that
the cardinality of $J$ is bounded by a constant independent of $x$ and $r$. \par
Let $Z$ as in the proof of \cref{lemma:doubling_implies_DS_reg}, recall
that $\#Z=\#J$,
$Z\subseteq B_X(x,Cr)$, and that $Z$ is $(b/a)r$-separated in $X$.
Recall also that $Cr<1/k$ and $g_{j,1/k}(x)\leq \epsilon/2$.
Since $Z\subseteq B_X(x,Cr)\subseteq U_X(x,1/k)$, for $z\in Z$ there is $w_z\in E_j$ such that
$d(w_z,z)\leq \epsilon d(x,z)\leq (b/4a)r$.
Since $Z$ is $(b/a)r$-separated in $X$, we deduce
that $W:=\{w_z\colon z\in Z\}$
is $(b/2a)r$-separated set in $X$ and contained in
$B_X(x,C(1+\epsilon)r)\cap E_j$.
Then the cardinality of $W$ is bounded by a constant depending only on the doubling constant of $E_j$
and $2aC(1+\epsilon)/b$.
Finally,
the separation condition of $Z$ implies also
$\#W=\#Z=\#J$, concluding the proof of the claim and hence the lemma.
\end{proof}
It turns out that, regardless of the contraction $\eta$, the separation condition in \cref{defn:shortcuts} is preserved, albeit with worse constants.
\begin{lemma}\label{lemma:sep_in_squashed}
There is a constant $c\in(0,1)$ such that,
for $i\leq j$, $S\in\calJ_i$ and $S'\in\calJ_j$ distinct,
it holds
\begin{equation}\label{eq:sep_in_squashed}
d_\eta(S,S')\geq c\delta_j.
\end{equation}
We may take $c=b/(1+a/b)$.
\end{lemma}
\begin{proof}
Let $x\in S$ and $y\in S'$ be such that $d_\eta(x,y)=d_\eta(S,S')$.
If $d_\eta(x,y)=d(x,y)$, then \cref{eq:sep_in_squashed} follows from \cref{defn:shortcuts}.
Suppose $d_\eta(x,y)<d(x,y)$, fix $C>1+a/b$, and let $p_-\neq p_+\in S''\in\calJ_k$ be such that
\begin{equation}\label{eq:sep_in_squashed_eq_1}
	d(x,p_-)+\rho_\eta(p_-,p_+)+d(p_+,y)\leq C d_\eta(x,y).
\end{equation}
Suppose first $1\leq k\leq j$.
If $S'\neq S''$, then $d(p_+,y)\geq d(S'',S')\geq b\delta_j$. Otherwise, $d(x,p_-)\geq d(S,S')\geq b\delta_j$.
In either case, from \cref{eq:sep_in_squashed_eq_1} we deduce $d_\eta(x,y)\geq C^{-1}b\delta_j$.
Suppose now $k>j$. Then
$d(x,y)\leq d(x,p_-)+d(p_+,y)$ and $d(x,y)\geq b\delta_j$, from which $d_\eta(x,y)\geq C^{-1}b\delta_j$ follows.
\end{proof}
Given $x,y\in S\in\calJ$, it is immediate that $d_\eta(x,y)\leq \rho_\eta(x,y)$, but it is not apriori clear how much smaller $d_\eta(x,y)$ can be.
We show that it cannot be too much smaller.
\begin{lemma}\label{lemma:diam_squashed_jump_sets}
There is a constant $c>0$
such that,
for $x,y\in S\in\calJ$,
it holds $d_\eta(x,y)\geq c\rho_\eta(x,y)$.
We may take $c=\min(2b,a_0)/a(1+a/b)$.
\end{lemma}
\begin{proof}
We can assume $d_\eta(x,y)<d(x,y)$.
Let $i\in\N$ be such that $S\in\calJ_i$.
Fix $C>1+a/b$ and let $p_-\neq p_+\in S'\in\calJ_j$  be such that
$d(x,p_-)+\rho_\eta(p_-,p_+)+d(p_+,y)\leq Cd_\eta(x,y)$.
If $1\leq j\leq i$ and $S\neq S'$, then
$d(S,S')\geq b\delta_i$
and so $Cd_\eta(x,y)\geq 2b\delta_i\geq 2(b/a) \rho_\eta(x,y)$.
If $S=S'$, then $\rho_\eta(p_-,p_+)\geq (a_0/a)\rho_\eta(x,y)$.
Finally, if $j>i$, then $d(x,p_-)+d(p_+,y)\geq d(x,y)\geq \rho_\eta(x,y)$.
\end{proof}
\begin{corol}\label{corol:01_law_shortcuts}
Let $(X,d,\mu)$ be an $s$-ADR metric measure space with \shortcuts\ $\{(\calJ_i,\delta_i)\}_i$.
Let $(\alpha_i)$ be a non-negative sequence and set
\begin{equation*}
	E:=\limsup_{i\rightarrow\infty} B_X(\cup\calJ_i,\alpha_i\delta_i), \qquad
	E_\eta:=\limsup_{i\rightarrow\infty} B_\eta(\cup\calJ_i,\alpha_i\delta_i).
\end{equation*}
Then $E$ and $E_\eta$ have full $\mu$-measure if $\sum_i\alpha_i^s=\infty$, while are $\mu$-null if $\sum_i\alpha_i^s<\infty$.
\end{corol}
\begin{proof}
We need only to consider $E_\eta$, since $E=E_\eta$ for $\eta=(1,1,\dots)$.
Recall that, by \cref{lemma:doubling_implies_DS_reg}, $(X,d_\eta,\mu)$ is $s$-ADR.
Let $c>0$ be the constant of \cref{lemma:sep_in_squashed},
set $C:=\max(M,2/c)$, and define $\calN_i:=\cup\calJ_i$ for $i\in\N$.
Then \cref{lemma:sep_in_squashed} and the last two conditions of \cref{defn:shortcuts}
imply that $(\calN_i)_i$ satisfies the assumptions of \cref{lemma:1stBorelCantelli,lemma:2ndBorelCantelli}
with our choice of $C$, concluding the proof.
\end{proof}
\begin{lemma}\label{lemma:nbhds_of_jump_sets}
There are $R_0>0$ and $C\geq 1$ such that
for any $S\in\calJ_i$ and $0<R<R_0$, it holds
\begin{equation*}
B_\eta(S,R\delta_i)\subseteq B_X(S,CR\delta_i).
\end{equation*}
We may take $R_0=b/(1+a/b)$ and $C=(1+a/b)^2$.
\end{lemma}
\begin{proof}
Let $z\in S$ and $x\in B_\eta(z,R\delta_i)$.
Set $C:=1+a/b$, let $\epsilon>0$ be such that $R(C+\epsilon)<b$, and pick $p_-\neq p_+\in S'\in \calJ_j$
such that $d(x,p_-)+\rho_\eta(p_-,p_+)+d(p_+,z)\leq (C+\epsilon)d_\eta(x,z)$.
If $1\leq j \leq i$ and $S\neq S'$, then $b\delta_i\leq d(p_+,z)\leq (C+\epsilon)R\delta_i$, a contradiction.
If $S=S'$, then $p_-\in S$ and so $d(x,S)\leq (C+\epsilon) d_\eta(x,z)$.
Lastly, if $j>i$,
we have $d(p_+,z)\geq b\delta_j\geq (b/a)d(p_-,p_+)$, which yields $d(x,z)\leq (C+\epsilon)(1+a/b)d_\eta(x,z)$.
Hence, $d(x,S)\leq C(C+\epsilon)R\delta_i$ for all $\epsilon>0$.
\end{proof}
\begin{lemma}\label{lemma:must_be_jump}
There is a constant $c>0$ such that
the following holds for any
$i\in\N$, $R>0$,
and $C\geq 1$
with $C(R+1)\eta_i<c$.
Let $p_-\neq p_+\in S\in\calJ_i$ and $x\in B_X(p_-,R\eta_i\delta_i\eta_i)$, $y\in B_X(p_+,R\delta_i\eta_i)$.
Then $d_\eta(x,y)<d(x,y)$ and for any $q_-,q_+\in S'\in\calJ$ satisfying
\begin{equation}\label{eq:must_be_jump_eq_1}
d(x,q_-)+\rho_\eta(q_-,q_+)+d(q_+,y)\leq Cd_\eta(x,y),
\end{equation}
we have $q_\pm =p_\pm$.
We may take $c=\min(a_0,b)/\max(4,a)(1+a/2b)$.
\end{lemma}
\begin{rmk}\label{rmk:must_be_jump_there_is_jump}
We stress that \cref{lemma:must_be_jump} does not
guarantee the validity of \cref{eq:must_be_jump_eq_1} with $q_\pm=p_\pm$,
but only that if \cref{eq:must_be_jump_eq_1} holds for some $q_\pm$, then it must be $q_\pm=p_\pm$.
However, combined with \cref{prop:single_jump}, it does imply \cref{eq:must_be_jump_eq_1} with $q_\pm=p_\pm$ and $C:=1+a/b$.
Indeed, if $C(R+1)\eta_i<c$, then $(C+\epsilon)(R+1)\eta_i<c$ for all sufficiently small $\epsilon>0$,
so from \cref{prop:single_jump} and \cref{lemma:must_be_jump} we deduce
\begin{equation*}
d(x,p_-)+\rho_\eta(p_-,p_+)+d(p_+,y)\leq (C+\epsilon)d_\eta(x,y),
\end{equation*}
for all $x\in B_X(p_-,R\delta_i\eta_i)$, $y\in B_X(p_+, R\delta_i\eta_i)$, and $\epsilon>0$ small.
\end{rmk}
\begin{proof}
Let $i\in\N$, $R>0$, $C\geq 1$, $p_-\neq p_+\in S\in\calJ_i$, and $x\in B_X(p_-,R\delta_i\eta_i)$, $y\in B_X(p_+,R\delta_i\eta_i)$.
We show that there is $c=c(a,a_0,b)>0$ such that, if the conclusion fails, then $C(R+1)\eta_i\geq c$. \par
Since $d_\eta\leq d$ and $\rho_\eta(p_-,p_+)\leq a\delta_i\eta_i$, we have
\begin{equation}\label{eq:must_be_jump_eq_2}
d_\eta(x,y)\leq (2R+a)\delta_i\eta_i,
\end{equation}
and so,
if $d_\eta(x,y)=d(x,y)$, then
\begin{equation*}
(2R+a)\delta_i\eta_i\geq d(x,y)\geq d(p_-,p_+)-d(x,p_-)-d(p_+,y)\geq a_0\delta_i - 2R\delta_i\eta_i,
\end{equation*}
which implies $(4R+a)\eta_i\geq a_0$. \par
Assume $d_\eta(x,y)<d(x,y)$
and that $q_-, q_+\in S'\in\calJ_j$ satisfy
\cref{eq:must_be_jump_eq_1}.
Suppose $1\leq j\leq i$ and either $q_-\neq p_-$
or $q_+\neq p_+$;
w.l.o.g.\ $q_-\neq p_-$.
Then
\begin{equation*}
d(x,q_-)\geq d(q_-,p_-)-d(x,p_-)\geq \min(a_0,b)\delta_i-R\delta_i\eta_i,
\end{equation*}
\cref{eq:must_be_jump_eq_1,eq:must_be_jump_eq_2} imply
$((2C+1)R+Ca)\eta_i\geq \min(a_0,b)$. \par
It remains to consider the case $j>i$.
We note that
\begin{align*}
d(q_-,q_+)
&\leq a\delta_j\leq \frac{a}{2b}\big(d(q_-,p_-)+d(q_+,p_+)\big) \\
&\leq \frac{a}{2b}\big(d(q_-,x)+d(x,p_-)+d(q_+,y)+d(y,p_+)\big) \\
&\leq \frac{a}{2b}\big(d(x,q_-)+d(q_+,y)+2R\delta_i\eta_i\big),
\end{align*}
and by triangle inequality
\begin{equation}\label{eq:must_be_jump_eq_3}
d(x,y)\leq (1+a/2b)\big(d(x,q_-)+d(q_+,y)\big) + \frac{a}{2b}2R\delta_i\eta_i.
\end{equation}
Since $p_-\neq p_+$, it holds $d(x,y)\geq a_0\delta_i-2R\delta_i\eta_i$,
and so \cref{eq:must_be_jump_eq_3,eq:must_be_jump_eq_1,eq:must_be_jump_eq_2} imply
$(1+a/2b)(2(C+1)R+Ca)\eta_i\geq a_0$. \par
If $c>0$ is defined as in the statement, then $C(R+1)\eta_i\geq c$
in any of the above cases.

\end{proof}
\begin{lemma}\label{lemma:jump_points_are_far}
There is a constant $C\geq 1$
such that the following holds for $i\in\N$ and $R>0$ with $(R+1)\eta_i\leq C^{-1}$.
For $S\in\calJ_i$ and $0<r\leq R$, there are closed sets $\{B_z(r)\colon z\in S\}$
satisfying:
\begin{itemize}
\item
$B_\eta(S,r\delta_i\eta_i)=\bigcup_{z\in S}B_z(r)$;
\item
$B_X(z,r\delta_i\eta_i)\subseteq B_z(r)\subseteq B_X(z,Cr\delta_i\eta_i)$;
\item
$d_\eta(B_z(r),B_{z'}(r))\geq C^{-1}\delta_i\eta_i$ and $d_X(B_z(r),B_{z'}(r))\geq C^{-1}\delta_i$ for $z\neq z'\in S$.
\end{itemize}
Moreover, for $0<r<r'\leq R$ and $z\in S$, it holds $B_{z}(r)\subseteq B_{z}(r')$.
We may take 
\begin{equation*}
C=\max(4,a)(1+a/b)^4/\min(a_0,b).
\end{equation*}
\end{lemma}
\begin{proof}
Set $C:=\max(4,a)(1+a/b)^4/\min(a_0,b)\geq 1$ and let $R>0$, $i\in\N$, and $S\in\calJ_i$ be as in the statement.
Define $C_0:=(1+a/b)^2$,
$B_z(r):=B_\eta(S,r\delta_i\eta_i)\cap B_X(z,C_0r\delta_i\eta_i)$, $z\in S$,
and
observe that $\{B_z(r)\colon z\in S\}$ is a collection of closed sets (\cref{lemma:same_topology})
satisfying the second point of the statement, because $C\geq C_0$.
Since $r\eta_i<C^{-1}< b/(1+a/b)$, \cref{lemma:nbhds_of_jump_sets}
shows that $\{B_z(r)\colon z\in S\}$ covers $B_\eta(S,r\delta_i\eta_i)$.
Let $z\neq z'\in S$.
Let $c_0\in (0,1)$ be defined as in \cref{lemma:must_be_jump} and note
that $(C_0r+1)\eta_i<C_0C^{-1}<c_0/(1+a/b)$.
Then, by \cref{rmk:must_be_jump_there_is_jump}, for $x\in B_z(r)$ and $y\in B_{z'}(r)$ we have
\begin{equation*}
	(1+a/b)d_\eta(x,y)\geq d(x,z)+\rho_\eta(z,z')+d(z',y)\geq a_0\delta_i\eta_i,
\end{equation*}
which implies $d_\eta(B_z(r),B_{z'}(r))\geq C^{-1}\delta_i\eta_i$ because $C^{-1}\leq a_0/(1+a/b)$.
If $x,y$ are as above, then
\begin{equation*}
d(x,y)\geq d(z,z')-d(x,z)-d(z',y)\geq a_0\delta_i-2C_0R\delta_i\eta_i
\geq (a_0-2C_0C^{-1})\delta_i\geq C^{-1}\delta_i
\end{equation*}
concludes the proof.
\end{proof}
The following was already proven in \cite[Theorem 1.3]{ledonne_li_rajala_shortcuts_heisenberg}
under slightly more restrictive assumption (see \cref{sec:comparison_LDLR}).
We include it here for comparison with, and as warm-up to, \cref{thm:PI_rectifiability} and \cref{thm:pure_PI_unrectifiability}.
Our proof relies on the Borel-Cantelli-type lemmas of \cref{subsec:Borel_Cantelli} and
the properties of $d_\eta$ developed so far.
\begin{prop}\label{prop:biLipschitz_pieces}
Let $(X,d,\mu)$ be an $s$-ADR metric measure space with \shortcuts\ $\{(\calJ_i,\delta_i)\}_i$
and $\eta=(\eta_i)\subseteq (0,1]$.
Then the following are equivalent:
\begin{itemize}
\item there is a positive-measure $\mu$-measurable subset of $X$
	on which $d$ and $d_\eta$ are biLipschitz equivalent;
\item $\inf\{\eta_i\colon i\in\N\}>0$;
\item $d$ and $d_\eta$ are biLipschitz equivalent on $X$.
\end{itemize}
\end{prop}
\begin{proof}
We need only to prove that the first point implies the second.
We prove the contrapositive.
Suppose $\inf\{\eta_i\colon i\in\N\}=0$, let $(\eta_{i_j})_j$ and
$(\alpha_j)\subseteq [1,\infty)$ be such that $\alpha_j\eta_{i_j}\rightarrow 0$ and $\sum_j (\alpha_j\eta_{i_j})^s=\infty$.
Set $E_0:=\limsup_{j}B_X(\cup\calJ_{i_j},\alpha_j\eta_{i_j}\delta_{i_j})$ and
fix a positive-measure $\mu$-measurable set $E\subseteq X$.
By \cref{lemma:doubling_implies_DS_reg}, the measure $\mu$ is $s$-ADR (in particular doubling)
also on $(X,d_\eta)$. Since,
by \cref{corol:01_law_shortcuts}, $E_0$
has full measure,
we may then find a Lebesgue density point $x$ of $E$ w.r.t.\ $d_\eta$ with $x\in E\cap E_0$. \par
Since $x\in E_0$, there are a subsequence $(i_{j_k})_k$ and $p_k^-\in S_k\in\calJ_{i_{j_k}}$
such that $d(x,p_k^-)\leq 2\alpha_{j_k}\eta_{i_{j_k}}\delta_{i_{j_k}}$ for $k\in\N$.
Let $C_0\geq 1$ denote the constant of \cref{lemma:jump_points_are_far}
and observe that, since $(\alpha_{j_k}+1)\eta_{i_{j_k}}\rightarrow 0$,
we may assume w.l.o.g.\ $(2\alpha_{j_k}+1)\eta_{i_{j_k}}\leq C_0^{-1}$ for all $k\in\N$.
Let $\{B_{z,k}\colon z\in S_k\}$
be as in \cref{lemma:jump_points_are_far}
with $r=R=2\alpha_{j_k}$
(and $i=i_{j_k}$, $S=S_k$),
let $p_k^+\in S_k\setminus\{p_k^-\}$, set $B_k^-:=B_{p_k^-,k}$ and $B_k^+:=B_{p_k^+,k}$,
and observe that $x\in B_k^-$.
From
$d_\eta(p_k^+,x)\leq (2\alpha_{j_k}+a)\eta_{i_{j_k}}\delta_{i_{j_k}}\leq (2+a)\alpha_{j_k}\eta_{i_{j_k}}\delta_{i_{j_k}}$,
we have
\begin{equation}\label{eq:biLipschitz_pieces_eq_1}
B_k^+\subseteq B_\eta(S_k,2\alpha_{j_k}\eta_{i_{j_k}}\delta_{i_{j_k}})\subseteq B_\eta(x, (4+a)\alpha_{j_k}\eta_{i_{j_k}}\delta_{i_{j_k}}).
\end{equation}
By \cref{lemma:jump_points_are_far}, we have
$\mu(B_k^+)\sim (\alpha_{j_k}\eta_{i_{j_k}}\delta_{i_{j_k}})^s$.
Then, since $x$ is a Lebesgue density point of $E$ w.r.t.\ $d_\eta$,
from the above and \cref{eq:biLipschitz_pieces_eq_1}
we see that $E\cap B_k^+\neq\varnothing$ for all sufficiently large $k$.
Assume w.l.o.g.\ $E\cap B_k^+\neq\varnothing$ for all $k\in\N$
and
let $y_k\in E\cap B_k^+$.
Then \cref{eq:biLipschitz_pieces_eq_1} implies
$d_\eta(x,y_k)\leq (4+a)\alpha_{j_k}\eta_{i_{j_k}}\delta_{i_{j_k}}$,
while
\cref{lemma:jump_points_are_far} and $x\in B_k^-$
give
$d(x,y_k)\geq C_0^{-1}\delta_{i_{j_k}}$.
We finally have
\begin{equation*}
\limsup_{E\ni y\rightarrow x}\frac{d(x,y)}{d_\eta(x,y)}
\geq
\limsup_{k\rightarrow\infty}\frac{d(x,y_k)}{d_\eta(x,y_k)}
\geq
\lim_{k\rightarrow\infty}\frac{C_0^{-1}}{(4+a)\alpha_{j_k}\eta_{i_{j_k}}}=\infty.
\end{equation*}
\end{proof}

\section{Shortcut metric spaces and PI (un)rectifiability}
\label{sec:pi_rect}

In this section we determine the sequences $\eta=(\eta_i)_i\subseteq (0,1]$ for which a shortcut space $(X,d_\eta,\mu)$ is PI rectifiable or purely PI unrectifiable, and show that there is no other possibility.

\subsection{Pure PI unrectifiability}
The main theorem regarding pure PI unrectifiability of shortcut spaces is the following.
\begin{thm}\label{thm:pure_PI_unrectifiability}
Let $(X,d,\mu)$ be an $s$-ADR metric measure space with \shortcuts\
and $\eta=(\eta_i)_i\subseteq (0,1]$.
Suppose $\inf\{\eta_i\colon i\notin I\}=0$ whenever $\sum_{i\in I}\eta_i^s<\infty$, $I\subseteq\N$.
Then $(X,d_\eta,\mu)$ is purely PI unrectifiable. \par
In particular, if $\eta_i\rightarrow 0$, then $(X,d_\eta,\mu)$ is purely PI unrectifiable whenever $\sum_{i\in\N}\eta_i^s=\infty$.
\end{thm}
For the proof of \cref{thm:pure_PI_unrectifiability}, we need the following lemma.
\begin{lemma}\label{lemma:fragments_to_distant_sets_have_gap}
Let $X$ be a metric space, $E, F\subseteq X$ sets, and let $\gamma\colon C\to E\cup F$
be a continuous function from a non-empty closed set $C\subseteq\R$.
Suppose $\gamma(C)\cap E\neq\varnothing$ and $\gamma(C)\cap F\neq\varnothing$.
Then $\gap\gamma\geq d(E,F)$.
\end{lemma}
\begin{proof}
We may assume w.l.o.g.\ $d(E,F)>0$, $0\in C\subseteq [0,\infty)$, and $\gamma(0)\in E$.
Set
\begin{equation*}
	\alpha:=\max\{t\in C\colon \gamma([0,t]\cap C)\subseteq E\}.
\end{equation*}
It is not difficult to see that $C\cap (\alpha,\alpha+\epsilon)=\varnothing$ for some $\epsilon>0$.
Also, by assumption $\gamma^{-1}(F)\cap C\neq\varnothing$, and therefore $C\setminus [0,\alpha]$ is a non-empty closed set which is bounded from below.
We can therefore define $\beta:=\min (C\setminus [0,\alpha])$.
By definition of $\alpha$, it must be $\gamma(\beta)\in F$, and thus
$\gap\gamma\geq d(\gamma(\alpha),\gamma(\beta))\geq d(E,F)$.
\end{proof}
\begin{proof}[Proof of \cref{thm:pure_PI_unrectifiability}]
By \cref{lemma:summable_subsequence}, there is an subsequence $(\eta_{i_j})_j$ with $\eta_{i_j}\rightarrow 0$ and $\sum_j\eta_{i_j}^s=\infty$.
Define
\begin{equation*}
	E_0:=\limsup_{j\rightarrow\infty}B_\eta(\cup\calJ_{i_j},\delta_{i_j}\eta_{i_j}),
\end{equation*}
and note that by \cref{corol:01_law_shortcuts} it has full measure.
Let $E\subseteq X$ be a positive-measure $\mu$-measurable set, $x\in E_0\cap E$ a Lebesgue density point of $E$ w.r.t.\ $d_\eta$,
$C_0\geq 1$ the constant in \cref{lemma:jump_points_are_far},
and $0<\delta\leq 1/C_0(4+a)$.
Let $C\geq 1$ and $\epsilon,r>0$.
We claim that there is $y\in B_\eta(x,r)\cap E$, $y\neq x$, such that
$(x,y)$ is not $(C,\delta,\epsilon)$-connected in $(E,d_\eta,\mu)$; see \cref{defn:C_delta_epsilon_connectivity}.
By \cref{thm:EB_PI_rect_asy_well}, this will conclude the proof.
\par
Let $R\geq C(4+a)$ and $j\in\N$ be such that $d_\eta(x,\cup\calJ_{i_j})\leq \delta_{i_j}\eta_{i_j}$, $(R+1)\eta_{i_j}\leq C_0^{-1}$,
and $(4+a)\delta_{i_j}\eta_{i_j}\leq r$.
Let $S\in\calJ_{i_j}$ be such that $d_\eta(x,S)\leq 2\delta_{i_j}\eta_{i_j}$ and let $\{B_z(t)\colon z\in S\}$
be as in \cref{lemma:jump_points_are_far} for $0< t\leq R$.
Since $x$ is a Lebesgue density point of $E$ w.r.t.\ $d_\eta$,
we may assume that $j\in\N$ was taken so large that $B_z(2)\cap E\neq\varnothing$ for all $z\in S$.
Indeed, if this failed for infinitely many $j$,
\begin{equation*}
B_X(z,2\delta_{i_j}\eta_{i_j})\subseteq B_z(2)\subseteq B_\eta(x,4\delta_{i_j}\eta_{i_j})
\end{equation*}
and $s$-AD regularity of $\mu$ on $(X,d)$ and $(X,d_\eta)$ would contradict the fact that $x$
is a Lebesgue density point of $E$ w.r.t.\ $d_\eta$. \par
Let $z_0\in S$ be such that $x\in B_{z_0}(2)$, pick $z\in S\setminus\{z_0\}$, $y\in B_z(2)\cap E$, and
observe that $d_\eta(x,y)\leq (4+a)\delta_{i_j}\eta_{i_j}\leq r$, i.e.\ $y\in B_\eta(x,r)\cap E$. \par
Let $K\subseteq\R$ be a non-empty compact set and $\gamma\colon K\to X$ a continuous
map with $\gamma(\min K)=x$ and $\gamma(\max K)=y$.
If $\var_\eta\gamma>R\delta_{i_j}\eta_{i_j}$, the choices of $R$ and $y$ imply
\begin{equation*}
	\var_\eta\gamma>C(4+a)\delta_{i_j}\eta_{i_j}\geq Cd_\eta(x,y).
\end{equation*}
Suppose $\var_\eta\gamma\leq R\delta_{i_j}\eta_{i_j}$.
Then $\gamma(K)\subseteq B_\eta(S,R\delta_{i_j}\eta_{i_j})$, $\gamma(K)\cap B_{z_0}(R)\neq\varnothing$,
and
$\gamma(K)\cap \bigcup\{B_{z'}(R)\colon z'\in S\setminus\{z_0\}\}\neq\varnothing$.
From \cref{lemma:fragments_to_distant_sets_have_gap,lemma:jump_points_are_far}, we deduce
\begin{equation*}
	\gap_\eta\gamma\geq d_\eta\left(B_{z_0}(R),\bigcup\{B_{z'}(R)\colon z'\in S\setminus\{z_0\}\}\right)\geq C_0^{-1}\delta_{i_j}\eta_{i_j}\geq \delta d_\eta(x,y).
\end{equation*}
This proves the claim and hence the thesis.
\end{proof}

\subsection{PI rectifiability}
In this subsection we prove the following.
\begin{thm}\label{thm:PI_rectifiability}
Let $(X,d,\mu)$ be a PI rectifiable $s$-ADR metric measure space with \shortcuts\
and $\eta=(\eta_i)_i\subseteq (0,1]$.
Suppose there is $I\subseteq\N$ with $\sum_{i\in I}\eta_i^s<\infty$ and $\inf\{\eta_i\colon i\notin I\}>0$.
Then $(X,d_\eta,\mu)$ is PI rectifiable. \par
In particular, if $\eta_i\rightarrow 0$, then $(X,d_\eta,\mu)$ is PI rectifiable whenever $\sum_{i\in\N}\eta_i^s<\infty$.
\end{thm}

\begin{rmk}
The hypotheses on the sequence $\eta$ in \cref{thm:pure_PI_unrectifiability} and \cref{thm:PI_rectifiability} exactly complement each other.
Therefore, for a given PI rectifiable $s$-ADR metric measure space $(X,d,\mu)$ with shortcuts, these results characterise when the shortcut metric space is PI rectifiable or purely PI unrectifiable.
\end{rmk}

\begin{defn}\label{defn:bad}
Given a metric space $X$ with \shortcuts\ $\{(\calJ_i,\delta_i)\}_i$
and $\eta=(\eta_i)_i\subseteq (0,1]$, we set
\begin{equation}
	\Bad:=\bigcup_{\alpha\in\N}\limsup_{i\rightarrow\infty} B_\eta(\cup\mathcal{J}_i,\alpha\delta_i\eta_i).
\end{equation}
One may verify that
$\Bad$ may be equivalently defined replacing the set $B_\eta(\cup\calJ_i,\alpha\delta_i\eta_i)$
with $B_X(\cup\calJ_i,\alpha\delta_i\eta_i)$.
\end{defn}

For \cref{thm:PI_rectifiability}, we need only to focus
on points \emph{not} in $\Bad$.
\begin{lemma}\label{lemma:prop_PI_rect_set}
Let $X$ be a metric space with \shortcuts\ $\{(\calJ_i,\delta_i)\}_i$, let $a,b$ be as in \cref{defn:shortcuts},
and $\eta=(\eta_i)_i\subseteq (0,1]$.
For any constant $C>1+a/b$,
$x\notin \Bad$, and $\epsilon>0$,
there is $R>0$ with the following property.
If $y\in B_\eta(x,R)$ and $p_-\neq p_+\in S\in \mathcal{J}$ satisfy
\begin{equation*}
d(x,p_-)+\rho_\eta(p_-,p_+)+d(p_+,y)\leq Cd_\eta(x,y),
\end{equation*}
then
\begin{equation*}
\rho_\eta(p_-,p_+)\leq\epsilon d_\eta(x,y).
\end{equation*}
\end{lemma}
\begin{proof}
Define $\alpha:=aC/\epsilon$.
Since $x\in X\setminus \Bad$, there is $i_0\in\N$ such that
\begin{equation*}
	d_\eta(x,\cup\mathcal{J}_i)\geq \alpha\delta_i\eta_i,\qquad i\geq i_0.
\end{equation*}
Let $R>0$ be such that $CR<a_0\delta_i\eta_i$ for $1\leq i< i_0$,
where $a_0$ is as in \cref{defn:shortcuts}.
Let $y$ and $p_-\neq p_+\in S\in \calJ_i$ be as in the statement.
Then $a_0\delta_i\eta_i\leq CR$ and,
from the choice of $R$, it follows that $i\geq i_0$ and thus
$\alpha\delta_i\eta_i\leq d_\eta(x,p_-) \leq Cd_\eta(x,y)$.
Then 
\begin{equation*}
\rho_\eta(p_-,p_+)\leq a\delta_i\eta_i\leq (a/\alpha)Cd_\eta(x,y)=\epsilon d_\eta(x,y).
\end{equation*}
\end{proof}
\begin{lemma}\label{lemma:PI_rect_good_set}
Let $(X,d,\mu)$ be an $s$-ADR metric measure space with \shortcuts.
There are constants $C_0\geq 1$ and $c_0\in (0,1)$ such that the following holds.
Let $\eta=(\eta_i)_i\subseteq (0,1]$ be a sequence,
$E\subseteq X\setminus\Bad$ a $\mu$-measurable set,
and suppose $(X,d,\mu)$ is $(C,\delta,\epsilon,r)$-connected along $E$
for some $C\geq 1$, $\delta,\epsilon\in(0,1)$, and $r>0$.
Then, for every Lebesgue density point $x$ of $E$ w.r.t.\ $d_\eta$,
there is $r_x>0$
such that $(x,y)$ is $(\overline{C},\overline{\delta},\overline{\epsilon})$-connected
in $(X,d_\eta,\mu)$ for all $y\in B_\eta(x,r_x)\setminus\{x\}$,
where
$\overline{C}=C_0 C$, $\overline{\delta}=C_0\delta$,
and
$\overline{\epsilon}=c_0\delta^s\epsilon$.
\end{lemma}
\begin{proof}
Let $a,b$ be as in \cref{defn:shortcuts},
fix $\tilde{C}_0>1+a/b$, set $C_0:=7\tilde{C}_0$, and let $c_0\in (0,1)$ to be determined later.
Let $R_x>0$ be the radius given by \cref{lemma:prop_PI_rect_set} applied to $x$ with
$\epsilon\equiv \delta$ and $C\equiv \tilde{C}_0$.
Fix $0<\alpha<\min(\delta,1/2)$.
Since $x$ is a Lebesgue point of $E$ w.r.t.\ $d_\eta$, there is $0<r_x<\min(r,R_x)/\tilde{C}_0(1+\alpha)$ such
that for every $y\in B_\eta(x,r_x)\setminus\{x\}$ there is $z\in B_\eta(y,\alpha d_\eta(x,y))\cap E$.
We claim that the thesis holds with this value of $r_x$, $C_0$, and $c_0$ sufficiently small. \par
Let $y$ and $z$ be as above, set $\rho_y:=d_\eta(x,y)$, $\rho_z:=d_\eta(x,z)$,
and observe that $(1-\alpha)\rho_y\leq\rho_z\leq (1+\alpha)\rho_y$.
Let $A\subseteq X$ be a $\mu$-measurable set with
\begin{equation*}
\mu(A\cap B_\eta(x,C_0 C\rho_y))<c_0\delta^s\epsilon\,\mu(B_\eta(x,C_0 C\rho_y)).
\end{equation*}
Suppose first $\rho_z=d(x,z)$.
Then $B_X(x,C\rho_z)\subseteq B_\eta(x,C(1+\alpha)\rho_y)\subseteq B_\eta(x,C_0 C\rho_y)$
and $\rho_z\geq \rho_y/2$ imply
for $c_0\in(0,1)$ small enough
\begin{equation*}
\mu(A\cap B_X(x,C\rho_z))< c_0\delta^s\epsilon\,\mu(B_\eta(x,C_0 C\rho_y))
\lesssim (c_0\epsilon)(C\rho_z)^s\leq \epsilon\mu(B_X(x,C\rho_z)),
\end{equation*}
where we have also used the fact that $\mu$ is $s$-ADR also on $(X,d_\eta)$, see \cref{lemma:doubling_implies_DS_reg}.
Since $x\in E$ and $d(x,z)=\rho_z\leq (1+\alpha)\rho_y<r$, the connectivity assumption
on $E$ yields a curve fragment $\gamma_z$ from $x$ to $z$, satisfying $\var_X\gamma_z\leq C \rho_z$,
$\gap_X\gamma_z<\delta \rho_z$, and which may meet $A$ only at the endpoints of its domain.
The curve fragment $\gamma$, obtained following $\gamma_z$ and then jumping from $z$ to $y$, satisfies
\begin{equation*}
\begin{aligned}
\var_\eta\gamma&\leq C\rho_z+\alpha\rho_y\leq (C(1+\alpha)+\alpha)\rho_y<C_0C\rho_y, \\
\gap_\eta\gamma&<\delta\rho_z+\alpha\rho_y\leq (2+\alpha)\delta\rho_y< C_0\delta\rho_y,
\end{aligned}
\end{equation*}
and may meet $A$ at its endpoints or in $z$.
It is not difficult to see that restricting the domain of $\gamma$ we can find a curve fragment $\tilde{\gamma}$
from $x$ to $y$ having slightly larger gap and variation (in $d_\eta$) which does not intersect $z$. \par
Suppose now $\rho_z<d(x,z)$, let $p_-\neq p_+\in S\in \calJ$ be such that
\begin{equation*}
	d(x,p_-)+\rho_\eta(p_-,p_+)+d(p_+,z)\leq \tilde{C}_0\rho_z,
\end{equation*}
and recall that, by the choice of $r_x$, we have $\rho_\eta(p_-,p_+)\leq \delta\rho_z$.
We first construct curve fragments $\gamma_x$, $\gamma_z$ from $x$ to $p_-$ and $z$ to $p_+$, respectively,
with controlled variation and gap.
Let $(w,p)$ be one of the pairs $(x,p_-)$, $(z,p_+)$.
Suppose $t:=d(w,p)\geq \delta \rho_z$.
Since $t\leq \tilde{C}_0\rho_z$ and $d_\eta(x,w)\leq \rho_z$, we have
$B_X(w,Ct)\subseteq B_\eta(x, (\tilde{C}_0C+1)\rho_z)\subseteq B_\eta(x,C_0 C\rho_y)$ and
therefore
\begin{equation}\label{eq:PI_rect_good_set_0}
\begin{aligned}
\mu(A\cap B_X(w,Ct))
&<c_0\delta^s\epsilon\,\mu(B_\eta(x,C_0 C\rho_y))
\lesssim (c_0\epsilon) (C\delta\rho_z)^s \\
&\lesssim (c_0 \epsilon) (Ct)^s\leq \epsilon\mu(B_X(w,Ct)),
\end{aligned}
\end{equation}
provided $c_0\in (0,1)$ is taken small enough.
Then, $t\leq \tilde{C}_0\rho_z\leq \tilde{C}_0(1+\alpha)\rho_y<r$, $w\in E$, the connectivity assumption on $E$,
and \cref{eq:PI_rect_good_set_0} ensure the
existence of a curve fragment $\gamma_w$
from $w$ to $p$ satisfying
\begin{equation}\label{eq:PI_rect_good_set_1}
\begin{aligned}
\var_\eta\gamma_w
&\leq\var_X\gamma_w\leq C t\leq \tilde{C}_0 C\rho_z, \\
\gap_\eta\gamma_w
&\leq\gap_X\gamma_w<\delta t\leq \tilde{C}_0 \delta \rho_z,
\end{aligned}
\end{equation}
and, moreover, $\gamma_w$ may meet $A$ only at the endpoints of its domain.
If, instead, $t=d(w,p)<\delta\rho_z$, let $\gamma_w\colon\{0,1\}\to X$
be given by $\gamma_w(0):=w$, $\gamma_w(1):=p$,
and observe that it
satisfies \cref{eq:PI_rect_good_set_1} and avoids $A$ in the same way.
Then, the curve fragment $\gamma\colon K\to X$ which follows $\gamma_x$, then
$\gamma_z$ with reverse orientation,
and, if $z\neq y$, finally jumps from $z$ to $y$,
satisfies
\begin{equation}\label{eq:PI_rect_good_set_2}
\begin{aligned}
\var_\eta\gamma&\leq (2\tilde{C}_0 C+\delta)\rho_z+\alpha\rho_y<C_0C\rho_y, \\
\gap_\eta\gamma&\leq (2\tilde{C}_0 +1)\delta\rho_z+\alpha\rho_y<C_0\delta\rho_y,
\end{aligned}
\end{equation}
where we have used \cref{eq:PI_rect_good_set_1}, $\rho_\eta(p_-,p_+)\leq\delta\rho_z$, and $d_\eta(z,y)\leq\delta\rho_y$.
Note that $\gamma$ may meet $A$ only at its endpoints or at $p_-,p_+$, or $z$.
By construction, there are $t_-<t_+<t_z\in K$ such that $\gamma^{-1}\{p_-,p_+,z\}=\{t_-,t_+,t_z\}$.
Hence, there is a curve fragment $\tilde{\gamma}$, obtained restricting slightly the domain of $\gamma$,
which may meet $A$ only at its endpoints, and moreover still satisfies \cref{eq:PI_rect_good_set_2}.
\end{proof}
\begin{proof}[Proof of \cref{thm:PI_rectifiability}]
Let $I\subseteq\N$  be such that $\sum_{i\in I}\eta_i^s<\infty$ and $\inf\{\eta_i\colon i\notin I\}>0$.
If $I$ is empty or finite, then $d$ and $d_\eta$ are biLipschitz and $(X,d_\eta,\mu)$ is trivially PI rectifiable.
Assume $I$ is infinite and suppose we have established the thesis for $I=\N$.
Let $(i_j)_j\subseteq\N$ be the strictly increasing sequence with $I=\{i_j\colon j\in\N\}$
and observe that
$d_\eta$ is biLipschitz to the distance $d_{\tilde{\eta}}$ obtained from the \shortcuts\ $\{(\calJ_{i_j},\delta_{i_j})\}_j$
and the sequence $\tilde{\eta}_j:=\eta_{i_j}$.
That is, $(X,d_\eta,\mu)$ is biLipschitz to the PI rectifiable space $(X,d_{\tilde{\eta}},\mu)$.
Hence, we need only to consider the case $I=\N$. \par
By \cref{corol:01_law_shortcuts} and $I=\N$, we have $\mu(\Bad)=0$.
Let $C_0, c_0$ be as in \cref{lemma:PI_rect_good_set} and let $\delta\in(0,1)$.
Since $(X,d,\mu)$ is PI rectifiable, by \cref{thm:PI_rect_RNP_LDS_asy_well_conn} and
\cref{lemma:doubling_meas_and_porosity}, for $\mu$-a.e.\ $x\in X$
there are $(C_x,\epsilon_x,r_x)$ as in \cref{defn:asy_well_conn} with $\delta\equiv \delta/C_0$.
We may also assume $x\mapsto (C_x,\epsilon_x,r_x)$ to be Borel measurable; see \cref{lemmma:asy_well_conn_Borel_const}.
Then,
there are countably many disjoint Borel sets $E_n\subseteq X$ and constants $(C_n,\epsilon_n,r_n)$
such that $(X,d,\mu)$ is $(C_n,\epsilon_n,\delta/C_0,r_n)$-connected along $E_n$
and $\mu(X\setminus\bigcup_nE_n)=0$.
Then, by \cref{lemma:PI_rect_good_set}, for every $n$ and $\mu$-a.e.\ $x\in E_n$ there
are constants $(\tilde{C}_x,\tilde{\epsilon}_x,\tilde{r}_x)$
such that $(x,y)$ is $(\tilde{C}_x,\delta, \tilde{\epsilon}_x)$-connected in $(X,d_\eta,\mu)$ for $y\in B_\eta(x,\tilde{r}_x)\setminus\{x\}$.
By \cref{thm:PI_rect_RNP_LDS_asy_well_conn} we conclude that $(X,d_\eta,\mu)$ is PI rectifiable.
\end{proof}

\section{Shortcut metric spaces and Lipschitz differentiability}\label{subsec:LDS_squashed}
Let $Y$ be a Banach space, $(X,d,\mu)$ a $Y$-LDS with shortcuts, and $f\colon (X,d_\eta)\to Y$ Lipschitz.
In this section we give a precise description of the set of points in $X$ where the derivative of $f$ on $(X,d)$ defines a derivative of $f$ on $(X,d_\eta)$.
We also show that this is the only way in which a differentiable structure on $(X,d_\eta,\mu)$ can arise, see \cref{lemma:comparison_diff_str_doubling} and \cref{prop:char_LDS_squashed}.
We require some definitions in order to state the main result in \cref{thm:differentiability_set}.

\begin{defn}[Compatible differentiable structure]\label{defn:compatible_structure}
Let $Y$ be a Banach space,
$(X,d,\mu)$ a $Y$-LDS, and suppose $(X,d)$ has \shortcuts\ $\{(\calJ_i,\delta_i)\}_i$.
We say that a Cheeger atlas 
\begin{equation*}
\{(U_j,\varphi_j\colon X\to \R^{n_j})\}_j 
\end{equation*}
of $(X,d,\mu)$ is \emph{compatible} with the \shortcuts\ if
\begin{equation}\label{eq:compatible_structure}
\diam \varphi_j(S)=0 
\end{equation}
for each $j$ and $S\in\calJ$.
If $(X,d,\mu)$ has a compatible atlas, we then say that it is a $Y$-LDS with compatible \shortcuts,
or
that the \shortcuts\ and the differentiable structure are compatible.
\end{defn}
Under the compatibility condition of \cref{defn:compatible_structure},
the only possible differentiable structure on $(X,d_\eta,\mu)$ is the one of $(X,d,\mu)$.
\begin{lemma}\label{lemma:comparison_diff_str_doubling}
Let $(X,d,\mu)$ be an LDS
with compatible \shortcuts.
Let $Y$ be a non-zero Banach space, $E\subseteq X$ a positive-measure $\mu$-measurable set,
and suppose $(E,d_\eta,\mu)$ is a $Y$-LDS.
Then, the restriction to $E$ of any compatible atlas of $(X,d,\mu)$ is
an atlas of $(E,d_\eta,\mu)$.
\end{lemma}
\begin{proof}
The proof consists of standard reductions and
the combination of \cref{lemma:DS_reg_no_doubling} with \cite[Theorem 3.26]{Cheeger_Eriksson_Bique_thing_carpets}. \par
Since the differentiable structure of an LDS does not depend on the target (see \cref{lemma:Y_LDS_vs_LDS}),
it is enough to consider the case $Y=\R$.
Let $\{(U_j,\varphi_j\colon X\to\R^{n_j})\}_j$ be a compatible atlas,
observe that $\varphi_j$ is Lipschitz also w.r.t.\ $d_\eta$,
and let $j$ be such that $\mu(E\cap U_j)>0$.
If $n_j=0$, then by \cref{prop:non_RNP_LDS} and \cref{lemma:same_topology} $\mu$-a.e.\ $x\in U_j$
is isolated and therefore $(E\cap U_j,\varphi_j)=(E\cap U_j,0)$ is trivially a chart.
Suppose now $n_j\geq 1$.
By \cref{lemma:porous_null,lemma:DS_reg_no_doubling} and
inner regularity of $\mu$, there are countably many compact sets $K_{j,k}\subseteq E\cap U_j$
with $\mu(E\cap U_j\setminus \bigcup_k K_{j,k})=0$
and such that the identity map $\iota\vert_{K_{j,k}}\colon (K_{j,k},d)\to (K_{j,k},d_\eta)$ is DS-regular.
Then, for each $k$, $(K_{j,k},d,\mu)$ and $(K_{j,k},d_\eta,\mu)$ are compact (hence proper) LDS, $\mu$ is Radon,
$\iota\vert_{K_{j,k}}\colon (K_{j,k},d)\to (K_{j,k},d_\eta)$ is a surjective DS-regular map preserving $\mu$,
$\varphi_j\vert_{K_{j,k}}\colon (K_{j,k},d_\eta)\to\R^{n_j}$ Lipschitz,
and $\varphi_j\vert_{K_{j,k}}\circ\iota\vert_{K_{j,k}}\colon (K_{j,k},d)\to\R^{n_j}$ is a global chart on $(K_{j,k},d,\mu)$.
Thus, by \cite[Theorem 3.26]{Cheeger_Eriksson_Bique_thing_carpets}, $\varphi_j\vert_{K_{j,k}}$
is a chart on $(K_{j,k},d_\eta,\mu)$. \par
Then, for any Lipschitz $f\colon (E,d_\eta)\to\R$, any $j,k$, and $\mu$-a.e.\ $x\in K_{j,k}$, there is a unique
linear map $T_x\colon\R^{n_j}\to\R$ such that
$\Lip_\eta((f-T_x\circ\varphi_j)\vert_{K_{j,k}};x)=0$.
Since $(E,d_\eta,\mu)$ is an LDS, by \cref{lemma:porous_null,lemma:ptwise_Lip_and_porosity}
we conclude $\Lip_\eta((f-T_x\circ\varphi_j)\vert_E;x)=0$ at $\mu$-a.e.\ $x\in E\cap U_j$, and any $j$. \par
Uniqueness of the $\varphi_j$-differentials follows from
\begin{equation*}
\Lip_\eta(\langle v,\varphi_j\rangle\vert_E;x)
\geq
\Lip(\langle v,\varphi_j\rangle\vert_E;x)
=\Lip(\langle v,\varphi_j\rangle;x)>0
\end{equation*}
for $v\in\R^{n_j}\setminus\{0\}$
and $\mu$-a.e.\ $x\in E\cap U_j$;
see \cref{lemma:indep_and_Banach_indep,lemma:ptwise_Lip_and_porosity,lemma:porous_null}.
The first inequality holds because $d_\eta\leq d$.
\end{proof}

We now introduce a definition related to the set Bad given in \cref{defn:bad}.
\begin{defn}\label{defn:badf}
Let $Y$ be a Banach space,
$(X,d,\mu)$ be a $Y$-LDS with compatible \shortcuts\ $\{(\calJ_i,\delta_i)\}_i$, and $f\colon (X,d_\eta)\to Y$ a map.
For $\epsilon>0$ and $i\in\N$, gather in $\calJ_i^\epsilon(f)$ the \jumpsets\ $S\in\calJ_i$ for which
$\diam f(S)\geq \epsilon \eta_i\delta_i$.
We define
\begin{equation}
\begin{aligned}
	\Badf_\epsilon(f)&:=\limsup_{i\rightarrow\infty} B_\eta(\cup \calJ_i^\epsilon(f),\eta_i\delta_i/\epsilon), \\
	\Badf(f)&:=\bigcup_{\epsilon >0} \Badf_\epsilon(f).
\end{aligned}
\end{equation}
\end{defn}

If $0<\epsilon'<\epsilon$, then $\Badf_{\epsilon}(f)\subseteq \Badf_{\epsilon'}(f)$.
Hence,  $\Badf(f)=\bigcup_j \Badf_{1/j}(f)$
shows that $\Badf(f)$ is Borel.
We show that, if $f$ is Lipschitz on $(X,d_\eta)$, $\Badf(f)$ coincides with the set of non-differentiability points of $f$, up to a null set.

The main result of this section is the following.
\begin{thm}\label{thm:differentiability_set}
Let $Y$ be a Banach space,
$(X,d,\mu)$ a $Y$-LDS with compatible \shortcuts,
and suppose $\mu$ vanishes on sets porous w.r.t.\ $d_\eta$.
Then, for any Lipschitz map $f\colon (X,d_\eta)\to Y$, we have
\begin{itemize}
\item $f$ is $\mu$-a.e.\ differentiable on $X\setminus \Badf(f)$ w.r.t.\ $d_\eta$ and any compatible atlas;
\item $f$ is $\mu$-a.\ nowhere differentiable on any subset of $\Badf(f)$ w.r.t.\ $d_\eta$ and any compatible atlas.
\end{itemize}
Moreover, the Cheeger differentials of $f$ w.r.t.\ $d$ and $d_\eta$ agree $\mu$-a.e.\ on $X\setminus \Badf(f)$.
\end{thm}
Before proving \cref{thm:differentiability_set}, let us single out two immediate applications.
\begin{prop}\label{prop:char_LDS_squashed}
Let $Y$ be a non-zero Banach space and $(X,d,\mu)$ a
$Y$-LDS with compatible \shortcuts.
Then the following are equivalent:
\begin{itemize}
\item
$(X,d_\eta,\mu)$ is a $Y$-LDS;
\item
$\mu(\Badf(f))=0$ for every Lipschitz map $f\colon (X,d_\eta)\to Y$
and $\mu$ vanishes on sets porous w.r.t.\ $d_\eta$.
\end{itemize}
\end{prop}
\begin{proof}
Suppose $(X,d_\eta,\mu)$ is a $Y$-LDS.
Then, by \cref{lemma:porous_null},
sets porous w.r.t.\ $d_\eta$ are $\mu$-null.
By \cref{lemma:comparison_diff_str_doubling}, there is a compatible atlas of $(X,d,\mu)$ which restricts
to an atlas of $(X,d_\eta,\mu)$
and, by \cref{thm:differentiability_set}, we then have $\mu(\Badf(f))=0$ for any Lipschitz $f\colon (X,d_\eta)\to Y$.
Conversely, if the assumptions of the second point are satisfied, then $(X,d_\eta,\mu)$ is a $Y$-LDS by \cref{thm:differentiability_set}.
\end{proof}
\begin{prop}\label{prop:LDS_unrect_subsets_Bad}
Let $Y$ be a non-zero Banach space, $(X,d,\mu)$
a $Y$-LDS with compatible \shortcuts,
and suppose $\mu$ vanishes on sets porous w.r.t.\ $d_\eta$.
Let $f\colon (X,d_\eta)\to Y$ be a Lipschitz map
and $E\subseteq \Badf(f)$ a positive-measure $\mu$-measurable set.
Then $(E,d_\eta,\mu)$ is not a $Y$-LDS.
\end{prop}
\begin{proof}
Suppose by contradiction that $(E,d_\eta,\mu)$ is a $Y$-LDS.
By \cref{lemma:comparison_diff_str_doubling}, there is a compatible atlas of $(X,d,\mu)$
which restricts to an atlas of $(E,d_\eta,\mu)$.
But then, by \cref{thm:differentiability_set}, $f\vert_{E}\colon (E,d_\eta)\to Y$ is $\mu$-almost nowhere
differentiable on $E$, from which we conclude $\mu(E)=0$.
\end{proof}
We now turn to the proof of \cref{thm:differentiability_set}.
Unsurprisingly (in view of \cref{thm:PI_rect_RNP_LDS_asy_well_conn}),
the proof of \cref{thm:differentiability_set} is reminiscent
of the ones of \cref{thm:PI_rectifiability} and \cref{thm:pure_PI_unrectifiability}.
The crucial point is that $\Bad(f)$ may be much smaller than $\Bad$ or,
in other words,
Lipschitz maps $f\colon (X,d_\eta)\to Y$
may be unable to fully capture the lack of `connectivity' (in the sense of \cref{defn:asy_well_conn})
of the underlying space.
This is the fundamental phenomenon behind \cref{thm:main_thm_Laakso}.
\Cref{thm:collapse_at_gates} and
the constructions of non-differentiable maps in \cref{section:nowhere_diff_maps}
also serve to illustrative this principle. \par
The following lemma should be compared to \cref{lemma:prop_PI_rect_set}.
It will be used for the proof of the first claim in \cref{thm:differentiability_set}.
\begin{lemma}\label{lemma:unpacking_Bad_set_diff}
Let $(X,d)$ be a metric space with \shortcuts\ $\{(\calJ_i,\delta_i)\}_i$
and $C>1+a/b$, where $a,b$ are as in \cref{defn:shortcuts}.
Then, for every
$\eta=(\eta_i)_i\subseteq (0,1]$, Banach space $Y$, $f\colon (X,d_\eta)\to Y$ Lipschitz,
$x\notin \Badf(f)$, and $\epsilon>0$, there is $R>0$ with the following property.
If $y\in B_\eta(x,R)$ and $p_-\neq p_+\in S\in\calJ$ satisfy
\begin{equation*}
d(x,p_-)+\rho_\eta(p_-,p_+)+d(p_+,y)\leq Cd_\eta(x,y),
\end{equation*}
then
\begin{equation}\label{eq:unpacking_Bad_set_diff}
\|f(p_-)-f(p_+)\|_Y\leq\epsilon d_\eta(x,y).
\end{equation}
\end{lemma}
\begin{proof}
Set $L:=\LIP_\eta(f)$ and $\epsilon_0:=\epsilon\min(a_0,1/aL)/C$, where $a_0$ is as in \cref{defn:shortcuts}.
Since $x\notin\Badf(f)$, in particular $x\notin\Badf_{\epsilon_0}(f)$, we find $i_0\in\N$ such that
\begin{equation}\label{eq:unpacking_Bad_set_diff_1}
d_\eta(x,\cup\calJ_i^{\epsilon_0}(f))\geq \eta_i\delta_i/\epsilon_0,
\end{equation}
for $i\geq i_0$.
Let $R>0$ be such that $a_0\eta_i\delta_i\leq CR$ implies $i\geq i_0$.
Let $y$ and $p_-\neq p_+\in S\in\calJ_i$ be as in the statement
and observe that
\begin{equation*}
a_0\eta_i\delta_i\leq\rho_\eta(p_-p_+)\leq Cd_\eta(x,y)\leq CR,
\end{equation*}
and so, by the choice of $R$, we have $i\geq i_0$ and hence \cref{eq:unpacking_Bad_set_diff_1} holds.
If $S\notin\calJ_i^{\epsilon_0}(f)$, then
\begin{equation*}
\|f(p_-)-f(p_+)\|_Y\leq\epsilon_0\eta_i\delta_i\leq \epsilon_0(C/a_0)d_\eta(x,y)\leq \epsilon d_\eta(x,y).
\end{equation*}
Otherwise, \cref{eq:unpacking_Bad_set_diff_1} gives
\begin{equation*}
\eta_i\delta_i\leq \epsilon_0 d_\eta(x,\cup\calJ_i^{\epsilon_0}(f))\leq \epsilon_0 d(x,S)\leq \epsilon_0 C d_\eta(x,y)
\end{equation*}
and, since $f$ is $L$-Lipschitz on $(X,d_\eta)$, we finally have
\begin{equation*}
\|f(p_-)-f(p_+)\|_Y\leq La\eta_i\delta_i\leq \epsilon_0 aLCd_\eta(x,y)\leq \epsilon d_\eta(x,y).
\end{equation*}
\end{proof}

We require the following lemma for the proof of the non-differentiability claim in \cref{thm:differentiability_set}.
\begin{lemma}\label{lemma:unpacking_Bad_set_non_diff}
Let $(X,d)$ be a metric space with \shortcuts\ $\{(\calJ_i,\delta_i)\}_i$ and let $\eta=(\eta_i)\subseteq (0,1]$.
Let $n\in\N$ and suppose $\varphi\colon X\to\R^n$ is a function satisfying $\diam\varphi(S)=0$ for $S\in\calJ$.
Then, for any Banach space $Y$, $f\colon (X,d_\eta)\to Y$, and $x\in\Badf(f)$, there are a constant $c>0$
and a sequence $(x_j)\subseteq X\setminus\{x\}$, $x_j\rightarrow x$, such that
\begin{equation*}
	\limsup_{j\rightarrow\infty}\frac{\|f(x_j)-f(x)-T(\varphi(x_j)-\varphi(x))\|_Y}{d_\eta(x,x_j)}\geq c,
\end{equation*}
for every linear $T\colon\R^n\to Y$.
\end{lemma}
\begin{proof}
Since $x\in\Badf(f)$ there are $\epsilon>0$,
a strictly increasing sequence $(i_k)\subseteq\N$, and $S_k\in\calJ_{i_k}^\epsilon(f)$ such that $d_\eta(x,S_k)\leq 2\eta_{i_k}\delta_{i_k}/\epsilon$.
Since $d_\eta(z,w)\leq a\eta_{i_k}\delta_{i_k}$ for $z,w\in S_k$, possibly shrinking $\epsilon>0$, we may assume
\begin{equation}\label{eq:unpacking_Bad_set_non_diff_1}
d_\eta(x,p)\leq \eta_{i_k}\delta_{i_k}/\epsilon
\end{equation}
for $p\in S_k$ and $k\in\N$.
Also, since $(S_k)$ are pairwise disjoint, we may assume $x\notin S_k$ for $k\in\N$.
By definition of $\calJ_{i_k}^\epsilon(f)$ and $S_k\in\calJ_{i_k}^\epsilon(f)$,
there are $p_k^-\neq p_k^+\in S_k$ such that
$\|f(p_k^-)-f(p_k^+)\|_Y\geq \epsilon\eta_{i_k}\delta_{i_k}$.
For $k\in\N$, define $x_{2k-1}:=p_k^-$, $x_{2k}:=p_k^+$,
and observe that $(x_j)\subseteq X\setminus\{x\}$ and $x_j\rightarrow x$.
Let $T\colon\R^n\to Y$ be linear and note that, by triangle inequality and $\varphi(p_k^-)=\varphi(p_k^+)$,
there are $\sigma_k\in\{-1,+1\}$ such that
\begin{equation}\label{eq:unpacking_Bad_set_non_diff_2}
\|f(p_k^{\sigma_k})-f(x)-T(\varphi(p_k^{\sigma_k})-\varphi(x))\|_Y\geq\epsilon\eta_{i_k}\delta_{i_k}/2
\end{equation}
for $k\in\N$.
Since $(p_k^{\sigma_k})_k$ is a subsequence of $(x_j)$, we finally have from \cref{eq:unpacking_Bad_set_non_diff_1}, \cref{eq:unpacking_Bad_set_non_diff_2}
\begin{align*}
\limsup_{j\rightarrow\infty}\frac{\|f(x_j)-f(x)-T(\varphi(x_j)-\varphi(x))\|_Y}{d_\eta(x,x_j)} &\geq
\limsup_{k\rightarrow\infty}\frac{\|f(p_k^{\sigma_k})-f(x)-T(\varphi(p_k^{\sigma_k})-\varphi(x))\|_Y}{\eta_{i_k}\delta_{i_k}/\epsilon} \\
&\geq \frac{\epsilon^2}{2}.
\end{align*}
\end{proof}

\begin{proof}[Proof of \cref{thm:differentiability_set}]
We begin with the statement regarding points where $f$ is differentiable.

Possibly replacing $Y$ with the closed span of $f(X)$, we may assume $Y$ to be separable.
Then, for $n\in\N$, the Banach space of linear maps $\R^n\to Y$ is also separable.
Let $\{(U_j,\varphi_j\colon X\to\R^{n_j})\}_j$ be a Cheeger atlas compatible with the \shortcuts\ and fix $\epsilon>0$.
Since $f$ is in particular Lipschitz on $(X,d)$ and $(X,d,\mu)$ a $Y$-LDS,
we may cover $\mu$-almost all of $X\setminus\Badf(f)$ with countably many Borel sets $E\subseteq X\setminus\Badf(f)$
having the following properties.
There is $j$ with $E\subseteq U_j$, $T\colon\R^{n_j}\to Y$ linear, and $R_1>0$ such that:
\begin{enumerate}[(i)]
\item\label{item:thm_diff_set_item_1}
for $x\in E$ and $y\in B_X(x,R_1)$, it holds
\begin{equation*}
\|f(y)-f(x)-T(\varphi_j(y)-\varphi_j(x))\|_Y\leq \epsilon d(x,y);
\end{equation*}
\item\label{item:thm_diff_set_item_2}
$E$ is not porous w.r.t.\ $d_\eta$ at any $x\in E$.
\end{enumerate}
Set $C:=1+2a/b$, let $R>0$ be as in \cref{lemma:unpacking_Bad_set_diff} with the current choices of parameters,
and assume further $R\leq R_1/C$.
We are now ready to prove differentiability.
Let $y\in B_\eta(x,R)\cap E$.
If $d_\eta(x,y)=d(x,y)$, then $R\leq R_1$ and \itemref{item:thm_diff_set_item_1} give
\begin{equation}\label{eq:thm_diff_set_eq_1}
\|f(y)-f(x)-T(\varphi_j(y)-\varphi_j(x))\|_Y\leq \epsilon d_\eta(x,y).
\end{equation}
Suppose now $d_\eta(x,y)<d(x,y)$.
By \cref{prop:single_jump}, we find $p_-\neq p_+\in S\in\calJ$ such that
\begin{equation}\label{eq:thm_diff_set_eq_2}
d(x,p_-)+\rho_\eta(p_-,p_+)+d(p_+,y)\leq Cd_\eta(x,y);
\end{equation}
in particular, $d(x,p_-),d(y,p_+)\leq C R\leq R_1$ and $y,p_-,p_+$ are as in \cref{lemma:unpacking_Bad_set_diff}.
Thus, item \itemref{item:thm_diff_set_item_1} and \cref{eq:unpacking_Bad_set_diff} give
\begin{equation}\label{eq:thm_diff_set_eq_3}
\begin{aligned}
\|f(y)-f(x)-T(\varphi_j(y)-\varphi_j(x))\|_Y
&\leq
\|f(y)-f(p_+)-T(\varphi_j(y)-\varphi_j(p_+))\|_Y
+
\|f(p_+)-f(p_-)\|_Y \\
&\qquad +
\|f(p_-)-f(x)-T(\varphi_j(p_-)-\varphi_j(x))\|_Y \\
&\leq
\epsilon d(y,p_+)+\epsilon d_\eta(x,y)+\epsilon d(x,p_-) \\
&\leq
\epsilon(C+1)d_\eta(x,y),
\end{aligned}
\end{equation}
where we have also used $\varphi_j(p_-)=\varphi_j(p_+)$ and \cref{eq:thm_diff_set_eq_2}.
Hence, \cref{eq:thm_diff_set_eq_1} and \cref{eq:thm_diff_set_eq_3} imply
$\Lip_\eta( (f-T\circ\varphi_j)\vert_{E};x)\leq (C+1)\epsilon$ for $x\in E$.
But then, by item \itemref{item:thm_diff_set_item_2} and \cref{lemma:ptwise_Lip_and_porosity},
we have $\Lip_\eta(f-T\circ\varphi_j;x)\leq (C+1)\epsilon$ for every $x\in E$.
Since, for every $\epsilon>0$, $X\setminus\Badf(f)$ is $\mu$-almost all covered by sets such as $E$, we have proven
that for every $\epsilon>0$, $j$, and $\mu$-a.e.\ $x\in U_j\setminus\Badf(f)$, there is a linear map $T_{x,\epsilon}\colon\R^{n_j}\to Y$
such that $\Lip_\eta(f-T_{x,\epsilon}\circ\varphi_j;x)\leq\epsilon$.
\Cref{lemma:diff_and_epsilon_diff} and
\begin{equation*}
\Lip_\eta(\langle v,\varphi_j\rangle;x)\geq \Lip(\langle v,\varphi_j\rangle;x)>0
\end{equation*}
for $v\in\R^{n_j}\setminus\{0\}$ and $\mu$-a.e.\ $x\in U_j$ conclude the proof
of differentiability; see \cref{lemma:indep_and_Banach_indep}.
Lastly, the differentials w.r.t.\ the two distances agree because they are unique $\mu$-a.e.\ and $d_\eta\leq d$.

We now consider points where $f$ is not differentiable.

Let $\{(U_j,\varphi_j\colon X\to\R^{n_j})\}_j$ be a compatible atlas and $A\subseteq\Badf(f)$ a set.
Let $N\subseteq X$ be the set of porosity points of $A$ and recall that $\mu(N)=0$ by assumption.
By \cref{lemma:unpacking_Bad_set_non_diff}, for every $j$ and $x\in A\cap U_j$, we have
$\Lip_\eta(f-T\circ\varphi_j;x)>0$ for every linear $T\colon\R^{n_j}\to Y$.
But then, by \cref{lemma:ptwise_Lip_and_porosity}, we have
$\Lip_\eta((f-T\circ\varphi_j)\vert_{A};x)>0$
for every linear map $T\colon\R^{n_j}\to Y$ and $x\in A\cap U_j\setminus N$.
\end{proof}

As mentioned in the introduciton, in general the set $\Bad(f)$ may have positive $\mu$ measure, as is demonstrated by the example in \cref{prop:non_LDS_dim_less_2}
and, for different targets, by
\cref{prop:non_diff_LIP_in_lq} and \cref{prop:non_diff_map_non_superreflexive}.
However, this is not possible if the space satisfies the following quantitative differentiation hypothesis.
Note that if it holds on the non-contracted space, then it holds also on the contracted one.
\begin{corol}\label{corol:quant_diff_implies_LDS}
Let $Y$ be a Banach space and let $(X,d,\mu)$ be an $s$-ADR $Y$-LDS with compatible \shortcuts.
Suppose that, for any Lipschitz map $f\colon (X,d_\eta)\to Y$ with bounded support,
it holds
\begin{equation*}
	\sum_{S\in\calJ} (\diam f(S))^s<\infty.
\end{equation*}
Then $\mu(\Badf(f))=0$ for all Lipschitz $f\colon (X,d_\eta)\to Y$.
In particular, $(X,d_\eta,\mu)$ is a $Y$-LDS.
\end{corol}

\begin{proof}
Let $f\colon (X,d_\eta)\to Y$ be Lipschitz.
Fix $x_0\in X$ and, for $R>0$, let $f_R\colon (X,d_\eta)\to Y$ be a Lipschitz function
with support in $B_\eta(x_0,2R)$ and which agrees with $f$ on $B_\eta(x_0,R)$.
Let $\epsilon>0$ and observe that for $S\in\calJ_i^{\epsilon}(f_R)$, we have
$\mu(B_\eta(S,\eta_i\delta_i/\epsilon))\lesssim \epsilon^{-2s}(\diam f_R(S))^s$.
Hence, the hypothesis gives
\begin{equation*}
\sum_{i\in\N}\sum_{S\in\calJ_i^\epsilon(f_R)} \mu(B_\eta(S,\eta_i\delta_i/\epsilon))
\lesssim  \epsilon^{-2s}\sum_{S\in\calJ}\left(\diam f_R(S)\right)^s
<\infty,
\end{equation*}
which implies $\mu(\Badf_\epsilon(f_R))=0$ by the first Borel-Cantelli lemma.
By the definition of $\Badf_\epsilon(f)$, it follows that $\Badf_\epsilon(f)\cap B_\eta(x_0,R/2)\subseteq \Badf_\epsilon(f_R)$,
and therefore $\mu(\Badf_\epsilon(f)\cap B_\eta(x_0,R/2))=0$ for all $\epsilon,R>0$.
Thus, we have $\mu(\Badf(f))=0$ for Lipschitz $f\colon (X,d_\eta)\to Y$,
and rest of the thesis follows from \cref{prop:char_LDS_squashed}, \cref{lemma:doubling_meas_and_porosity}, and \cref{lemma:doubling_implies_DS_reg}.
\end{proof}

\section{Laakso spaces}\label{sec:Laakso_pi}
Laakso spaces were introduced in \cite{laakso2000_ADR_PI} as examples of Ahlfors-David regular $1$-PI spaces whose Hausdorff dimension varies continuously in $(1,\infty)$. 
We now describe the construction.

Let $M\in\N$, $M\geq 2$, $\theta \in (0,1/4]$, and $(N_n)\subseteq 2\N$, $N_n\geq 4$, such that
\begin{equation}\label{eq:N_n_theta_n_Laakso}
	\prod_{i=1}^n\frac{1}{N_i}\sim \theta^n,\qquad n\in\N.
\end{equation}
It is convenient to metrise $[M]^\N$ as follows.
For $x\neq y\in [M]^\N$, let $n\in\N_0$ be the greatest integer such that
$x(i)=y(i)$ for $1\leq i<n$,
and set $d(x,y)=\theta^n$. \par
Set $\tilde{X}:=[0,1]\times[M]^\N$ and endow it with the product topology (and metric).
Let $h\colon \tilde{X}\to [0,1]$ denote the projection onto the first coordinate,
which we will call \emph{height}.
For $i\in\N$, set
\begin{equation*}
W_i^h:=\left\{\sum_{j=1}^i\frac{t_j}{N_1\cdots N_j}\colon t_j\in\N_0, 0\leq t_j<N_j \text{ for } 1\leq j<i \text{ and } 1\leq t_i<N_i\right\},
\end{equation*}
whose elements will be called \emph{wormhole heights of level $i$}.
We also set
\begin{equation}\label{eq:wormhole_heights_leq_i}
	W_{\leq i}^h:=\bigcup_{j=1}^iW_j^h=\left\{\frac{m}{N_1\cdots N_i}\colon m\in\N, 0<m<N_1\cdots N_i\right\}.
\end{equation}
We now define an equivalence relation on $\tilde{X}$.
Given $x,y\in \tilde{X}$, we write $x\sim y$ if either $x=y$, or there is $i\in\N$ such that $h(x)=h(y)\in W_i^h$ and
$x,y$ differ at most in the $i$-th digit, i.e.\ $x(j)=y(j)$ for all $j\neq i$.
Let $X$ denote the quotient space and $q\colon \tilde{X}\to X$ the quotient map. \par
We call \emph{wormholes of level $i$} points $x\in X$ for which $h(x)\in W^h_i$.
Note that for $x\in X$, $q^{-1}\{x\}$ is a singleton unless $x$ is a wormhole of some level $i\in\N$, in which case $\#q^{-1}\{x\}=M$
and for $a\in [M]$ there is $y\in q^{-1}\{x\}$ with $y(i)=a$.
Hence, if $x\in X$ is not a wormhole, we can talk about its digits, denoted by $(x(j))_j$.
Instead, if $x$ is a wormhole of some level $i\in\N$, its digit of level $i$ is not defined, but the ones of level $j\neq i$ are, which we still denote as $x(j)$.
We then say that
$x,y\in X$ \emph{differ at the $i$-th digit} or \emph{differ at the digit of level $i$}
if $\tilde{x}(i)\neq\tilde{y}(i)$
for all representatives $\tilde{x},\tilde{y}\in\tilde{X}$ of $x,y$ respectively.
In this case, we also write $x(i)\neq y(i)$.
\par
We metrise $X$ by setting
\begin{equation*}
d(x,y):=\inf\{\Haus^1(\gamma):q\circ\gamma \text{ is a continuous curve from }x\text{ to }y\},
\end{equation*}
for $x,y\in X$.
\begin{rmk}
The above distance can be informally described as follows.
Let $x,y\in X$ and let
$\Delta\subseteq\N$ be the set of levels $j$ for which $x(j)\neq y(j)$.
Any path from $x$ to $y$ has to change digits at each level in $\Delta$,
which can occur only by traveling through wormholes of the corresponding level.
Paths from $x$ to $y$ can be explicitly constructed as follows.
Let $I\subseteq [0,1]$ be an interval containing $h(x),h(y)$ and intersecting $W_i^h$ for each $i\in\Delta$.
We can then travel
from $x$ to the closest boundary point of $I$, then along $I$, and lastly to $y$, changing all necessary digits along the way;
the resulting path has length $2|I|-|h(x)-h(y)|$.
It turns out the there is always an optimal path of this form.
\end{rmk}
\begin{prop}[{\cite[Proposition 1.1]{laakso2000_ADR_PI}}]\label{prop:dist_in_Laakso}
Let $x,y\in X$ and $\ell$ the minimum among lengths of intervals containing $h(x),h(y)$, and all wormhole heights necessary to travel from $x$ to $y$.
Then
\begin{equation*}
d(x,y)=2\ell - |h(x)-h(y)|.
\end{equation*}
\end{prop}
It is not difficult to verify that $d$ induces the quotient topology on $X$, $q\colon\tilde{X}\to X$ is David-Semmes regular,
and $h\colon X\to [0,1]$ is open \cite{laakso2000_ADR_PI}. \par
For $M, (N_n)$, and $d$ as above, we call \emph{Laakso space} the metric space $(X,d)$ and,
when we want to highlight the dependence on $M$ and $(N_n)$, we write
$X(M;(N_n))$.
\begin{rmk}\label{rmk:comparison_Laakso}
In \cite{laakso2000_ADR_PI}, Laakso allows for a slightly larger family of sequences $(N_n)$ than us.
Hence, what we call Laakso spaces form a strict subset of the ones considered in \cite{laakso2000_ADR_PI}.
Nonetheless, as will be clarified shortly, the spaces we consider can still achieve any Hausdorff dimension in $(1,\infty)$.
\end{rmk}
Laakso proves the following.
\begin{thm}[{\cite{laakso2000_ADR_PI}}]\label{thm:vanilla_Laakso}
Let $X=X(M;(N_n))$ be a Laakso space and $s\in (1,\infty)$ such that
$M\theta^{s-1}=1$, i.e.\ $s=1+\log M/\log(1/\theta)$.
Then $(X,d,\Haus^s)$ is a compact geodesic $s$-ADR $1$-PI space. \par
In particular, by \cite{cheeger_kleiner_PI_RNP_LDS}, $(X,d,\Haus^s)$ is a $Y$-LDS
for any Banach space $Y$ with RNP.
\end{thm}
The following lemma gives more precise information on the charts.
\begin{lemma}\label{lemma:Laakso_height_is_chart}
Let $(X,d,\Haus^s)$ be an $s$-ADR Laakso space.
Then the height function $h\colon X\to\R$ is a global Cheeger chart, satisfying $\Lip(h;\cdot)=1$ everywhere.
\end{lemma}
\begin{proof}
The equality $\Lip(h;\cdot)=1$ is clear. \par
One can prove directly that $h$ is a chart following an argument analogous to
\cite{capolli_pinamonti_speight_Laakso_maximal_der}.
For brevity, we explain how this follows from results in the literature. \par
From \cite[5.A]{Weaver_Lipschitz_algebras_and_derivations_2_ext_diff},
the module of Weaver derivations on $(X,d,\Haus^s)$ has exactly one independent derivation.
Then, by \cite[Theorem 3.24]{schioppa2016_derivations_and_alberti_representations},
$(X,d,\Haus^s)$ has exactly one independent Alberti representation and
hence, by \cite[Theorem 6.6]{bate_str_of_measures_2015}, the analytic dimension of $(X,d,\Haus^s)$ is $1$. \par
\end{proof}
To construct non-differentiable maps, it will be useful to consider the following sequence of approximations of $X=X(M;(N_n))$.
\begin{defn}
Set $\tilde{X}_n:=[0,1]\times [M]^n$, $n\in\N_0$, and let $\pi^\infty_n\colon \tilde{X}\to \tilde{X}_n$ denote the coordinate projection map.
Similarly to $\sim$ on $\tilde{X}$, we can define an equivalence relation on $\tilde{X}_n$, which then satisfies
\begin{equation*}
	x\sim y \text{ if and only if } \pi^\infty_n(x)\sim \pi^\infty_n(y) \text{ for all }n\in\N_0.
\end{equation*}
We let $X_n$ denote the quotient of $\tilde{X}_n$ under the above equivalence relation and still denote with $q$ the quotient map; this should cause no confusion.
The maps $\pi_n^\infty$ induce maps $X\to X_n$, which we still denote with $\pi^\infty_n$.
We similarly have maps $\pi^m_n\colon \tilde{X}_m\to \tilde{X}_n$, $\pi^m_n\colon X_m\to X_n$ for $0\leq n\leq m$.
We endow the spaces $X_n$ with the quotient topology and define a distance similarly to $X$.
\end{defn}
One can then recover $X$ from $(X_n)_n$ in several ways.
Indeed, it is then not difficult to verify that $X$ corresponds to the inverse limit of the inverse system given by $(X_n)_n$ and $(\pi^{n+1}_n)_n$,
and that $(X_n)_n$ Gromov-Hausdorff converges to $X$; see e.g.\ \cite{Cheeger_Kleiner_inverse_limit_PI}.
\subsection{Laakso spaces as metric graphs}\label{subsec:Laakso_as_graphs}
Let $G=(V,E)$ be a connected (abstract) graph, $\ell>0$, and endow it with its length distance $d_G$,
scaled such that $d_G(u,v)=\ell $ for each edge $\{u,v\}\in E$.
We denote with $X_G$ the metric space obtained by attaching an isometric copy of $[0,\ell]$ to each pair of
neighbouring vertices.
We can then identify $V$ with an isometric copy in $X_G$ and each $e\in E$ with the copy of $[0,\ell]$ in $X_G$ joining $\partial e$.
With these identifications, we call $(X_G,V,E)$ the \emph{metic graph} associated to $G$. \par
The following will be useful in the construction of non-differentiable functions. We omit the proof.
\begin{lemma}\label{lemma:extension_from_graph}
Let $(X_G,V,E)$ be a metric graph associated to a connected graph $G=(V,E)$, $Y$ a Banach space, and $f\colon V\to Y$ a function.
Then there is a unique extension $F\colon X_G\to Y$ of $f$ which is affine linear on each edge of $X_G$.
Moreover, $\LIP(F)=\LIP(f)$.
\end{lemma}
Let $n\in\N$. We now define a graph $G_n=(V,E)$ for which $X_n=X_{G_n}$ isometrically.
The most natural choice for $G_n$ would be a multigraph, but since we prefer working with graphs, our choice of $G_n$ will be a topological minor of the natural multigraph.
For $1\leq l\leq n$, set
\begin{equation*}
V_l:=W_l^h\times\left([M]^{l-1}\times\{0\}\times[M]^{n-l}\right)
\end{equation*}
and
\begin{equation*}
	V:=\bigcup_{l=1}^nV_l\cup\left(\left\{\frac{m}{N_1\cdots N_{n+1}}\colon m\in\N_0, 0\leq m\leq N_1\cdots N_{n+1}\right\}\setminus W^h_{\leq n}\right)\times [M]^n.
\end{equation*}
In the definition of $V_l$, the digit $0$ should be considered simply as a `symbol' outside the `alphabet' $[M]$.
For $v,v'\in V$, we say that they adjacent if $|h(v)-h(v')|=\frac{1}{N_1\cdots N_{n+1}}$
and $v(j)=v'(j)$ for all $1\leq j\leq n$ for which $v(j),v'(j)\neq 0$.
This adjacency relation defines an edge set $E$ and we set $G_n=(V,E)$.
It is then not difficult to verify that, taking $\ell = \frac{1}{N_1\cdots N_{n+1}}$, we have $X_n=X_{G_n}$ isometrically.
\subsection{Cubical covers of Laakso spaces}
\label{sec:cubical_laakso}
We now define a family of covers of a Laakso space.

\begin{defn}\label{defn:laakso_cover}
Let $X=X(M;(N_n))$ be a Laakso space, fixed for the rest of the subsection.
For $i\in\N$,
define
\begin{align*}
\calI_0&:=\{[0,1]\}, \\
\calI_i&:=\left\{\left[\frac{m}{N_1\cdots N_i}, \frac{m+1}{N_1\cdots N_i}\right]\colon m\in\N_0, 0\leq m<N_1\cdots N_i\right\}.
\end{align*}
Further, for $i,j\in\N_0$ and $I\in\calI_i$, $a\in [M]^j$,
define
\begin{equation*}
\begin{aligned}
\tilde{Q}_{I,a}&:=I\times\{a\}\times [M]^\N, \\
Q_{I,a}&:=q(I\times\{a\}\times [M]^\N), \\
\calQ_i&:= \{Q_{I,a}\colon I\in\calI_i,a\in [M]^i\}, \\
\calQ&:= \bigcup_{i\in\N_0}\calQ_i,
\end{aligned}
\end{equation*}
where we interpret $I\times\{a\}\times[M]^\N=I\times [M]^\N$ if $j=0$. \par
\end{defn}

Recall that given two sets $\Omega,Z$ and a function $f\colon \Omega\to Z$, a set $E\subseteq \Omega$ is
\emph{$f$-saturated} if $f^{-1}(f(E))=E$.
\begin{lemma}\label{lemma:cube_saturated}
Let $i,j\in\N_0$ with $j\leq i$ and $I\in\calI_i$, $a\in[M]^j$.
Then $\Int\tilde{Q}_{I,a}$ is $q$-saturated
and $\Int Q_{I,a}=q(\Int\tilde{Q}_{I,a})$.
\end{lemma}
\begin{proof}
We first prove that $\Int\tilde{Q}_{I,a}$ is $q$-saturated.
We may suppose $i\geq 1$, for otherwise $\Int\tilde{Q}_{I,a}=\Int\tilde{X}=\tilde{X}$, which is $q$-saturated.
Let $J$ denote the interior of $I$ as a subset of $[0,1]$.
Then $\Int\tilde{Q}_{I,a}=J\times\{a\}\times[M]^\N$
and $J$ does not intersect $W_{\leq i}^h$ (see \cref{eq:wormhole_heights_leq_i}).
Thus, if $x\in \Int\tilde{Q}_{I,a}$ and $y\in q^{-1}\{q(x)\}$,
then $h(y)=h(x)\in J$ and
$y(k)=x(k)$ for $k>i\geq j$,
proving that $\Int\tilde{Q}_{I,a}$ is $q$-saturated. \par
Being $q$ a quotient map (in the topological sense), $q(\Int\tilde{Q}_{I,a})$ is open in $X$ and so
$q(\Int\tilde{Q}_{I,a})\subseteq\Int Q_{I,a}$.
For the reverse inclusion, let $V\subseteq Q_{I,a}$ be open.
Then $h(V)$ is open in $[0,1]$ and contained in $I$.
Hence, $h(V)\subseteq J$ and $V$ does not contain wormholes of level $>j$,
implying $q^{-1}(V)\subseteq\Int\tilde{Q}_{I,a}$.
\end{proof}
Saturated sets have the following useful properties.
We include proofs for completeness.
\begin{lemma}\label{lemma:prop_saturated_sets}
Let $\Omega,Z$ be sets, $f\colon \Omega\to Z$ a function,
and suppose $E\subseteq\Omega$ is $f$-saturated.
Then
$f(E\cap A) = f(E)\cap f(A)$
and $f(A\setminus E) = f(A)\setminus f(E)$
for all $A\subseteq\Omega$.
\end{lemma}
\begin{proof}
The inclusion
$f(E\cap A)\subseteq f(E)\cap f(A)$ is clear.
Let $z\in f(E)\cap f(A)$ and $x\in A$ with $f(x)=z$.
Then $x\in f^{-1}\{z\}\subseteq f^{-1}(f(E))=E$, proving $z\in f(E\cap A)$
and hence $f(E\cap A)=f(E)\cap f(A)$. \par
The inclusion $f(A)\setminus f(E)\subseteq f(A\setminus E)$ is clear.
Let $z\in f(A\setminus E)$ and pick $x\in A\setminus E$ such that $f(x)=z$.
If $z\in f(E)$, then $x\in f^{-1}\{z\}\subseteq f^{-1}(f(E))=E$, a contradiction.
Thus, we have $f(A\setminus E)\subseteq f(A)\setminus f(E)$, concluding the proof.
\end{proof}
\begin{lemma}\label{lemma:Laakso_intersection_cubes}
Let $i,j\in\N_0$ with $j\leq i$ and $I,I'\in\calI_i$, $a,a'\in[M]^j$.
Then $\partial Q_{I,a}=q(\partial \tilde{Q}_{I,a})$ and
$Q_{I,a}\cap Q_{I',a'}=\partial Q_{I,a}\cap \partial Q_{I',a'}$,
whenever $(I,a)\neq (I',a')$.
\end{lemma}
\begin{proof}
We only need to apply \cref{lemma:cube_saturated,lemma:prop_saturated_sets}.
We have
\begin{align*}
\partial Q_{I,a}&=q(\tilde{Q}_{I,a})\setminus q(\Int\tilde{Q}_{I,a})
=q(\tilde{Q}_{I,a}\setminus \Int\tilde{Q}_{I,a})=q(\partial \tilde{Q}_{I,a}), \\
Q_{I,a}\cap \Int Q_{I',a'}&=q(\tilde{Q}_{I,a})\cap q(\Int\tilde{Q}_{I',a'})
=q(\tilde{Q}_{I,a}\cap\Int\tilde{Q}_{I',a'})=\varnothing,
\end{align*}
where in the second line we have used
$\tilde{Q}_{I,a}\cap \Int\tilde{Q}_{I',a'}=\varnothing$ for $(I,a)\neq (I',a')$.
\end{proof}
To conclude, we gather the most relevant properties of $(\calQ_i)_{i\in\N_0}$.
\begin{lemma}\label{lemma:Laakso_cubical_cover_basic}
Let $X=X(M;(N_n))$ be an $s$-ADR Laakso space and let $(\calQ_i)_{i\in\N_0}$
be as in \cref{defn:laakso_cover}.
For $i\in\N_0$,
$\calQ_i$ is a finite compact cover of $X$ satisfying:
\begin{enumerate}[(i)]
\item $\Haus^s(\partial Q)=0$ for $Q\in\calQ_i$;
\label{item:Laakso_cubical_cover_basic_item_1}
\item $Q\cap Q' = \partial Q\cap \partial Q'$ for $Q\neq Q'\in\calQ_i$;
\label{item:Laakso_cubical_cover_basic_item_2}
\item for $Q_0\in\calQ_i$ and $j>i$ it holds
\label{item:Laakso_cubical_cover_basic_item_3}
\begin{equation*}
Q_0=\bigcup\{Q\in\calQ_j\colon Q\subseteq Q_0\}.
\end{equation*}
\end{enumerate}
Moreover, there is $C=C(X)\geq 1$ such that
for $i\in\N_0$ and $Q\in\calQ_i$ it holds
\begin{equation}\label{eq:cubical_cover_cube_shape}
\begin{aligned}
C^{-1}\theta^i&\leq \diam Q\leq C\theta^i \\
C^{-1}\theta^{is}&\leq \Haus^s(Q)\leq C \theta^{is}
\end{aligned}
\end{equation}
Here $\theta$ is as in \cref{eq:N_n_theta_n_Laakso}.
\end{lemma}
\begin{proof}
To see that $\calQ_i$ is compact, recall that $q$ is continuous.
Item \itemref{item:Laakso_cubical_cover_basic_item_2} is the content of \cref{lemma:Laakso_intersection_cubes},
while \itemref{item:Laakso_cubical_cover_basic_item_3} follows from the definition of $\calQ_i$.
Let $i\in\N_0$ and $Q=Q_{I,a}\in\calQ_i$.
Since $q$ is David-Semmes regular, by \cref{lemma:Laakso_intersection_cubes} we have (see \cref{eq:DS_reg_maps_and_Haus}) 
\begin{equation}\label{eq:Laakso_cubical_cover_basic_eq_1}
\Haus^s_X(\partial Q)\sim q_\#\Haus^s_{\tilde{X}}(q(\partial\tilde{Q}_{I,a}))=\Haus^s_{\tilde{X}}(q^{-1}(q(\partial\tilde{Q}_{I,a})).
\end{equation}
Let $J$ denote the interior of $I$ as a subset of $[0,1]$.
Then $\partial\tilde{Q}_{I,a}=(I\setminus J)\times \{a\}\times [M]^\N$,
$\#(I\setminus J)\leq 2$, and $q^{-1}(q(\partial\tilde{Q}_{I,a}))\subseteq (I\setminus J)\times [M]^\N$.
Since $\Haus^s_{\tilde{X}}((J\setminus I)\times [M]^\N)=0$, \cref{eq:Laakso_cubical_cover_basic_eq_1} concludes the proof of \itemref{item:Laakso_cubical_cover_basic_item_1}. \par
Let $i\in\N_0$ and $Q\in\calQ_i$.
Since $\calQ_0=\{X\}$, we may assume $i\in\N$ and take $I\in\calI_i$, $a\in [M]^i$ such that $Q=Q_{I,a}$.
From \cref{prop:dist_in_Laakso}, it is clear that $\diam Q_{I,a}\sim \theta^i$.
Since $\Haus^s$ is $s$-ADR (see \cref{thm:vanilla_Laakso}), it follows $\Haus^s(Q)\lesssim \theta^{is}$.
From the choice of distance on $[M]^\N$ and \cref{eq:N_n_theta_n_Laakso}, it is not difficult to see that
$\tilde{Q}_{I,a}$ contains a ball of $\tilde{X}$ of radius $\sim \theta^i$.
Since $q\colon\tilde{X}\to X$ is DS-regular, we finally have
\begin{equation*}
	\Haus^s_X(Q)\sim q_{\#}\Haus^s_{\tilde{X}}(Q)\geq \Haus^s_{\tilde{X}}(\tilde{Q}_{I,a})\gtrsim \theta^{is}.
\end{equation*}
\end{proof}

\subsection{\Shortcuts\ in Laakso spaces}\label{sec:squashed_Laakso}
In this subsection we show that Laakso spaces have \shortcuts, see \cref{prop:Laakso_has_shortcuts}.

Let $X=X(M;(N_n))$ be a Laakso space, fixed for the subsection.

\begin{defn}\label{defn:shortcuts_laakso}
For $i\in\N$, set
\begin{align*}
\calJ_i^h
&:=\left\{\sum_{j=1}^i\frac{t_j}{N_1\cdots N_j}+\frac{1}{2N_1\cdots N_i}\colon t_j\in\N_0, 0\leq t_j<N_j \text{ for }1\leq j<i \text{ and } 1\leq t_i<N_i-1\right\} \\
&=\left\{ \frac{a+b}{2}\colon [a,b]\in \calI_i \text{ and }a,b\in W_i^h\right\}.
\end{align*}
When we wish to emphasise the dependence on $N_1,\dots, N_i$ in $\calJ_i^h$, we then write
\begin{equation}\label{eq:shortcuts_explicit_notation}
\calJ_i^h=\calJ_i^h(N_1,\dots,N_i).
\end{equation}
Since $N_{i+1}$ is even, $\calJ_i^h\subseteq W_{i+1}^h$.
We define the following set of \emph{\shortcuts\ of level $i$}
\begin{equation*}
\calJ_i:=\left\{q(\{t\}\times \{a\}\times [M]\times \{1\}^\N)\colon t\in\calJ_i^h \text{ and }a\in[M]^{i-1}\right\},
\end{equation*}
where, for $i=1$, we interpret
\begin{equation*}
\{t\}\times \{a\}\times [M]\times \{1\}^\N=\{t\}\times [M]\times \{1\}^\N.
\end{equation*}
Also define
\begin{equation*}
	\delta_i:=\frac{1}{N_1\cdots N_i}.
\end{equation*}

Finally, for $n\in\N$, consider $X_n$ defined as in \cref{subsec:Laakso_as_graphs}.
We also define $(\calJ_i)_{i=1}^n$ similarly as above also in $X_n$.
\end{defn}

We now proceed to prove that the sets $\calJ_i$ defined above are indeed shortcuts in the Laakso space $X$.

We will need the following fact, often used without mention.
Let $N_1,\dots, N_i\in\N$, $m\in\N_0$, $0\leq m <N_1\cdots N_i$, and set $t:=\tfrac{m}{N_1\cdots N_i}$.
Then, there are unique $t_1,\dots,t_i\in\N_0$, $0\leq t_j<N_j$, such that
\begin{equation*}
t=\sum_{j=1}^i\frac{t_j}{N_1\cdots N_j}.
\end{equation*}
\begin{lemma}\label{lemma:shortcuts_Laakso_part1}
Let $i\in\N$, $S\in\calJ_i$, and let
$t\in (0,1)$, $a\in [M]^{i-1}$ be such that $S=q(\{t\}\times \{a\}\times [M]\times\{1\}^\N)$.
Let $I\in\calI_i$ be the interval having $t$ as midpoint
and suppose $x\notin \Int Q_{I,a}$. \par
Then $d(x,z)=d(x,S)\geq \frac{1}{2N_1\cdots N_i}$ for $z\in S$.
In particular, 
\begin{equation*}
	d(x,y)\leq d(x,z)+d(w,y)
\end{equation*}
for $x,y\notin \Int Q_{I,a}$ and $z,w\in S$.
\end{lemma}
\begin{proof}
If $h(x)\notin \Int I$, then any geodesic from $x$ to $z\in S$ has to cross $\partial I\subseteq W_i^h$
and therefore meet a wormhole of level $i$.
At such intersection point we can modify the $i$-th digit and obtain a geodesic from $x$ to $w$,
for any $w\in S$.
Hence, $d(x,z)=d(x,S)$ for $z\in S$.
Since $h$ is $1$-Lipschitz, $d(x,S)\geq |h(x)-t|=\frac{1}{2N_1\cdots N_i}$. \par
Suppose now $x\in Q_{I,b}$ for some $b\in [M]^{i-1}$, $b\neq a$,
and note that $I$ does not contain wormholes heights of level $\leq i-1$.
Since $b\neq a$,
there is $1\leq j\leq i-1$ such that $x(j)=b(j)\neq a(j)$.
Again,
since $W_j^h\cap I=\varnothing$, any geodesic from $x$ to $z\in S$ has to cross $\partial I$
and therefore meet a wormhole of level $i$.
It follows that $d(x,S)=d(x,z)$ and $d(x,S)\geq \frac{3}{2N_1\cdots N_i}$.
\end{proof}
\begin{lemma}\label{lemma:shortcuts_Laakso_part2}
Let $i\in\N$, $S\in\calJ_i$, and let
$t\in (0,1)$, $a\in [M]^{i-1}$ be such that $S=q(\{t\}\times \{a\}\times [M]\times\{1\}^\N)$.
Let $I\in\calI_i$ be the interval having $t$ as midpoint.
Then $\Int Q_{I,a}=U_X(S,\frac{1}{2N_1\cdots N_i})$.
\end{lemma}
\begin{proof}
\cref{lemma:shortcuts_Laakso_part1} gives $U_X(S,\frac{1}{2N_1\cdots N_i})\subseteq \Int Q_{I,a}$.
Fix $x\in \Int Q_{I,a}$ and let
$p\in S$ be such that $x(j)=p(j)$ for $1\leq j\leq i$.
Since $\calJ_i^h\subseteq W_{i+1}^h$, we need to change only digits of level $j\geq i+2$.
If the interval with endpoints $\{h(x),t\}$ intersects $W_{i+2}^h$,
then $d(x,p)=|h(x)-t|$, because $t\in W_{i+1}^h$ and between wormhole heights of levels $l_1,l_2$
there are wormholes heights of level $l$
for each $l\geq \max(l_1,l_2)$.
Suppose now there are no wormholes heights of level $i+2$ between $t$ and $h(x)$.
Then $|h(x)-t|<\frac{1}{N_1\cdots N_{i+2}}$.
Let $y\in X$ be such that $h(y)\in W_{i+2}^h$, $|h(y)-h(x)|=d(h(x),W_{i+2}^h)$, and $y(j)=x(j)$ for all $j$.
Then $d(y,p)=|h(y)-t|=\frac{1}{N_1\cdots N_{i+2}}$, $d(x,p)=d(x,y)+d(y,p)$,
and $d(x,y)=|h(x)-h(y)|=|h(y)-t|-|t-h(x)|$.
Hence,
\begin{equation*}
	d(x,p)\leq 2|h(y)-t|=\frac{2}{N_1\cdots N_{i+2}}<\frac{1}{2N_1\cdots N_i},
\end{equation*}
because $N_{i+1}N_{i+2}>4$.
\end{proof}
\begin{lemma}\label{lemma:shortcuts_Laakso_separation}
Let $i\leq j$ and $S\in\calJ_i$, $S'\in\calJ_j$ be distinct.
Then $d(S,S')\geq \frac{1}{2N_1\cdots N_j}$.
\end{lemma}
\begin{proof}
Let $z\in S$, $z'\in S'$, and $t=h(z)$, $t'=h(z')$.
If $t\neq t'$, then 
\begin{equation*}
d(z,z')\geq |t-t'|\geq \frac{1}{2N_1\cdots N_j}.
\end{equation*}
Suppose now $t=t'$. Since $(\calJ_k^h)_k$ is pairwise disjoint,
in particular we have $i=j\geq 2$.
Let $a\neq a'\in [M]^{j-1}$ such that $S, S'$ are determined by $(t,a)$ and $(t,a')$ respectively,
and let $I\in\calI_j$ be the interval having $t$ as midpoint.
Let $1\leq k\leq j-1$ be such that $a(k)\neq a'(k)$ and hence $z(k)\neq z'(k)$.
Since there are no wormholes heights of level $k$ in $I$, we have
\begin{equation*}
	d(z,z')\geq 2 d(t,W_k^h)=\frac{3}{N_1\cdots N_j}.
\end{equation*}
\end{proof}
\begin{lemma}\label{lemma:shortcuts_Laakso_scale}
For $i\in\N$ and $z\neq w\in S\in\calJ_i$ it holds
$d(z,w)=\frac{1}{N_1\cdots N_i}$.
\end{lemma}
\begin{proof}
The points $z,w$ differ exactly in the $i$-th digit and $d(h(z),W_i^h)=d(h(w),W_i^h)=\frac{1}{2N_1\cdots N_i}$.
\end{proof}
\begin{lemma}\label{lemma:shortcuts_Laakso_net_prelim}
Let $i\in\N$ and $t\in W_i^h$. Then there is $I\in\calI_i$ such that $t\in\partial I\subseteq W_i^h$.
\end{lemma}
\begin{proof}
We can write $t=\sum_{j=1}^i\frac{t_j}{N_1\cdots N_j}$ with $t_j\in\N_0$, $0\leq t_j<N_j$ for $1\leq j<i$ and $1\leq t_i\leq N_i-1$.
If $t_i\leq N_i-2$, then $t':=t+\frac{1}{N_1\cdots N_i}\in W_i^h$ and $[t,t']\in\calI_i$.
If $t_i=N_i-1>1$, then $t':=t-\frac{1}{N_1\cdots N_i}\in W_i^h$ and $[t',t]\in\calI_i$.
\end{proof}
\begin{lemma}\label{lemma:shortcuts_Laakso_net}
For $i\in\N$ and $x\in X$, it holds $d(x,\cup\calJ_i)\leq \frac{3}{2N_1\cdots N_i}$.
\end{lemma}
\begin{proof}
Let $I\in\calI_i$ be such that $h(x)\in I$.
Since $\partial I\cap W_i^h\neq \varnothing$, by \cref{lemma:shortcuts_Laakso_net_prelim} there is $I'\in\calI_i$
with $\partial I'\subseteq W_i^h$ and $\partial I\cap \partial I'\neq \varnothing$.
Let $z\in X$ be such that $h(z)\in\partial I\cap \partial I'$ and $z(j)=x(j)$ for all $j$.
Then
\begin{equation*}
d(x,\cup\calJ_i)\leq d(x,z)+d(z,\cup\calJ_i)\leq \frac{3}{2N_1\cdots N_i}.
\end{equation*}
\end{proof}

\begin{prop}\label{prop:Laakso_has_shortcuts}
Let $X=X(M;(N_n))$ be an $s$-ADR Laakso space.
Then $\{(\calJ_i,\delta_i)\}_{i\geq 1}$, defined as in \cref{defn:shortcuts_laakso},
are \shortcuts\ compatible with
the differentiable structure of $(X,d,\Haus^s)$. \par
More precisely, it satisfies the conditions of \cref{defn:shortcuts}
with parameters $a=a_0=1$, $b=1/2$, and $M=M$,
and $\diam h(S)=0$ for $S\in\calJ$.
\end{prop}
\begin{proof}
This follows combining the previous lemmas.
The first condition of \cref{defn:shortcuts} is given by \cref{lemma:shortcuts_Laakso_scale},
the second by \cref{lemma:shortcuts_Laakso_separation}, the third by \cref{lemma:shortcuts_Laakso_part1} and \cref{lemma:shortcuts_Laakso_part2},
the fourth by the definition of $(\calJ_i)_i$, and the last by \cref{lemma:shortcuts_Laakso_net} (and $M\geq 2$). \par
It is a PI space \cite{laakso2000_ADR_PI} and hence
a $Y$-LDS for any Banach space $Y$ with RNP
by \cite{cheeger_kleiner_PI_RNP_LDS}.
Compatibility of the \shortcuts\ with the differentiable structure follows by definition of $(\calJ_i)_i$
and \cref{lemma:Laakso_height_is_chart}.
\end{proof}

\section{Harmonic approximation on LDS of analytic dimension 1}\label{sec:harmonic_approximation}
The main result of this subsection is \cref{prop:harmonic_approximation_general}.
Under suitable assumptions, it decomposes a Lipschitz map $f\colon X \to Y$ into a sum of Lipschitz
maps $(F_{i+1}-F_i)$ with control of the total energy of the sequence (see \cref{defn:energy}), where $F_i$ agrees with $f$ on a `grid' of the space.
This will be used later, in \cref{thm:collapse_at_gates}, to prove that
\squashed\ Laakso spaces are $Y$-LDS for some Banach spaces $Y$. \par
The proof of \cref{prop:harmonic_approximation_general} is inspired
by Pisier's martingale inequality \cite[Theorem 10.6]{Pisier_martingales_in_Banach_2016},
where we replace Jensen's inequality with a local minimality condition.

Piecewise harmonic approximations of $\ell_2$ valued maps appear in \cite{schioppa2016_PI_unrectifiable}.
Our construction differs by building the harmonic approximations in an LDS, rather than a cube complex, and so we only have access to the Cheeger derivative.
On the other hand, our construction does not rely on piecewise Euclidean approximations of a metric space,
nor Euler-Lagrange equations, and holds for general (superreflexive) Banach space targets.
We also do not need to study the
regularity of minimisers, nor find covers with regular boundaries (in the potential-theoretic sense).

Recall that a Banach space $Y$ is $q$-uniformly convex, $2\leq q<\infty$, if its
modulus of uniform convexity $\delta_{Y}(\cdot)$ satisfies $\delta_{Y}(\epsilon)\geq c\epsilon^q$ for some $c>0$
and all $\epsilon\in (0,2]$.
Equivalently, there is $K>0$ such that the inequality
\begin{equation}\label{eq:q_uniform_convexity}
\left\|\frac{x+y}{2}\right\|_Y^q +\frac{1}{K^q}\left\|\frac{x-y}{2}\right\|_Y^q
\leq \frac{\|x\|_Y^q}{2}+\frac{\|y\|_Y^q}{2}
\end{equation}
holds for $x,y\in Y$; see e.g.\ \cite[Proposition 10.31]{Pisier_martingales_in_Banach_2016}.
We let $K_q(Y)$ denote the least $K$ as in \cref{eq:q_uniform_convexity}. \par
\begin{defn}\label{defn:energy}
Let $(X,d,\mu)$ is a metric measure space, $1\leq q<\infty$, and $Y$ a Banach space.
For $f\colon X\to Y$ Lipschitz and $Q\subseteq X$ $\mu$-measurable, we define the \emph{$q$-energy} of $f$ on $Q$ as
\begin{equation*}
	E_q(f,Q):=\int_Q \Lip(f;x)^q\,\dd\mu(x).
\end{equation*}
\end{defn}
\begin{lemma}\label{lemma:q_UC_E_q}
Let $Y$ be a $q$-uniformly convex Banach space, $(X,d,\mu)$ a $Y$-LDS,
and $Q\subseteq X$ a $\mu$-measurable set of analytic dimension $1$.
Then, for any $u,v\colon X\to Y$ Lipschitz, we have
\begin{equation*}
	2E_q\left(\tfrac{u+v}{2},Q\right)+2K_q(Y)^{-q}E_q\left(\tfrac{u-v}{2},Q\right)
\leq
E_q(u,Q)+E_q(v,Q).
\end{equation*}
\end{lemma}
\begin{proof}
By definition of analytic dimension,
there are charts $(U_i,\varphi_i\colon X\to\R)$ such that
$U_i\subseteq Q$ and $\mu(Q\setminus\bigcup_i U_i)=0$.
Since $(U_i,\varphi_i)$ is $1$-dimensional, we may identify $\varphi_i$-differentials with elements of $Y$,
so that $\Lip(f;x)=\|d_xf\|_Y\Lip(\varphi_i;x)$ for any $f\colon X\to Y$ which is $\varphi_i$-differentiable at $x\in U_i$.
Hence, \cref{eq:q_uniform_convexity} implies
\begin{equation*}
2\Lip\left(\tfrac{u+v}{2};x\right)^q+2K_q(Y)^{-q}\Lip\left(\tfrac{u-v}{2};x\right)^q
\leq \Lip(u;x)^q+\Lip(v;x)^q
\end{equation*}
at $\mu$-a.e.\ $x\in Q$.
Integrating over $Q$ concludes the proof.
\end{proof}
\begin{defn}
Let $(X,d,\mu)$ be a metric measure space, $1\leq q<\infty$, and $Y$ a Banach space.
For $f\colon X\to Y$ Lipschitz and $Q\subseteq X$ $\mu$-measurable, we define the following sets of admissible functions and minimisers, respectively,
\begin{align*}
	\calA(f,Q)&:=\{u\colon \overline{Q}\to Y\colon \LIP(u)\leq\LIP(f), u\vert_{\partial Q} = f\vert_{\partial Q}\}, \\
	\calE_q(f,Q)&:=\{u\in\calA(f,Q)\colon E_q(u,Q)=\inf_{v\in\calA(f,Q)}E_q(v,Q)\}.
\end{align*}
\end{defn}
Both $\calA(f,Q)$ and $\calE_q(f,Q)$ are convex subsets of $\LIP(X;Y)$ and $\calA(f,Q)$ is never empty,
since $f\vert_{\overline{Q}}\in\calA(f,Q)$.
\begin{lemma}\label{lemma:existence_harmonic_fun_on_piece}
Let $Y$ be a $q$-uniformly convex Banach space,
$(X,d,\mu)$ a $Y$-LDS,
$Q\subseteq X$ a $\mu$-measurable set,
and $f\colon X\to Y$ a Lipschitz map. \par
Then $\calE_q(f,Q)$ is non-empty.
Moreover, if $Q$ has analytic dimension $1$,
the elements $u\in\calE_q(f,Q)$ are characterised as those $u\in\calA(f,Q)$ satisfying
\begin{equation}\label{eq:existence_harmonic_fun_on_piece_1}
	E_q(u,Q)+2(2K_q(V))^{-q}E_q(u-v,Q)\leq E_q(v,Q),\qquad v\in\calA(f,Q).
\end{equation}
\end{lemma}
\begin{proof}
Proving existence requires a simple application of the direct method of calculus of variations.
If $\partial Q=\varnothing$, any constant function is admissible and a minimiser, hence $\calE_q(f,Q)\neq\varnothing$.
Suppose $\partial Q\neq\varnothing$ and let $(u_j)$ be a minimising sequence.
Since $(u_j)$ is equiLipschitz and bounded at a point, it is pointwise bounded on $\overline{Q}$.
Using reflexivity of $Y$ (see e.g.\ \cite[Theorem 10.3]{Pisier_martingales_in_Banach_2016}), Mazur's lemma,
and arguing as in Arzel\`a-Ascoli's theorem, we may suppose $(u_j)$ converges pointwise to a function $u\colon\overline{Q}\to Y$,
which then belongs to $\calA(f,Q)$.
Moreover, by \cref{lemma:continuity_Cheeger_diff}, we have
$E_q(u,Q)\leq\lim_j E_q(u_j,Q)=\inf_{v\in\calA(f,Q)}E_q(v,Q)$, proving $u\in \calE_q(f,Q)$. \par
It remains to prove the variational characterisation of minimisers.
Let $u\in\calE_q(f,Q)$, $v\in\calA(f,Q)$ and observe that $(u+v)/2\in\calA(f,Q)$.
We may assume $E_q(v,Q)<\infty$, so that $E_q(u,Q)\leq E_q(\tfrac{u+v}{2},Q)<\infty$.
Then \cref{lemma:q_UC_E_q} and minimality of $u$ give
\begin{align*}
2K_q(Y)^{-q}E_q\left(\tfrac{u-v}{2},Q\right)
&\leq E_q(u,Q)+E_q(v,Q)-2E_q\left(\tfrac{u+v}{2},Q\right) \\
&\leq E_q(v,Q)-E_q(u,Q).
\end{align*}
Lastly, if $u\in\calA(f,Q)$ satisfies \cref{eq:existence_harmonic_fun_on_piece_1}, non-negativity of $E_q(\cdot,Q)$ implies $u\in \calE_q(f,Q)$.
\end{proof}
The following is an immediate application of \cref{lemma:LIP_on_pieces_to_LIP}.
\begin{lemma}\label{lemma:gluing_harmonic_approximations}
Let $Y$ be a Banach space, $1\leq q<\infty$,
and let $(X,d,\mu)$ be a metric measure space.
Suppose $(X,d)$ is a length space
and let
$\calQ$ be a locally finite closed cover of $X$ satisfying
$Q\cap Q' = \partial Q\cap \partial Q'$ for $Q\neq Q'\in\calQ$. \par
Then, for any Lipschitz $f\colon X\to Y$ and $(u_Q)_{Q\in\calQ}$ with $u_Q\in\calE_q(f,Q)$, $Q\in\calQ$,
there is a unique function $F\colon X\to Y$
satisfying $F\vert_{Q}=u_Q$.
Moreover, $\LIP(F)\leq\LIP(f)$.
\end{lemma}
Note that a locally finite cover $\calQ$ of $X$ is necessarily countable.
Indeed, there is an open cover $\calU$ of $X$ of sets intersecting only finitely many elements of $\calQ$.
Since $X$ is separable, it has the Lindel\"of property, and so there is a countable subcover $\calU'\subseteq \calU$.
Then $\calQ\setminus \{\varnothing\} = \bigcup_{U\in\calU'}\{Q\in\calQ\colon Q\cap U\neq \varnothing\}$ is a countable union of finite sets.
\begin{prop}\label{prop:harmonic_approximation_general}
Let $Y$ be a $q$-uniformly convex Banach space, $(X,d,\mu)$ a $Y$-LDS
of analytic dimension $1$,
and suppose $(X,d)$ is a length space. \par
For $i\in\N$,
let $\calQ_i$ be a locally finite closed cover of $X$ satisfying:
\begin{enumerate}[(i)]
\item $\mu(\partial Q)=0$ for $Q\in\calQ_i$;
\item $Q\cap Q' = \partial Q\cap \partial Q'$ for $Q\neq Q'\in\calQ_i$;
\item for $Q_0\in\calQ_i$ and $j>i$ it holds
\begin{equation*}
Q_0=\bigcup\{Q\in\calQ_j\colon Q\subseteq Q_0\}.
\end{equation*}
\end{enumerate}
Let $f\colon X\to Y$ be $L$-Lipschitz and, for $i\in\N$, let
$F_i\colon X\to Y$ be such that $F_i\vert_{Q}\in\calE_q(f,Q)$ for $Q\in\calQ_i$.
Then, for $i_0\in\N$ and $Q_0\in\calQ_{i_0}$, we have
\begin{equation*}
\sum_{i\geq i_0}E_q(F_i-F_{i+1},Q_0)\leq \tfrac{1}{2}\left(2K_q(Y)L\right)^q\mu(Q_0).
\end{equation*}
\end{prop}
\begin{proof}
Set $c:=2(2K_q(Y))^{-q}$ and assume w.l.o.g.\ $\mu(Q_0)<\infty$.
Observe that $F_{i+1}\vert_{Q}\in\calA(f,Q)$ for $Q\in\calQ_i$.
Indeed,
$\partial Q\subseteq\bigcup\{\partial Q'\in\calQ_{i+1}\colon Q'\subseteq Q\}$
implies $F_{i+1}\vert_{\partial Q}=f\vert_{\partial Q}$,
while \cref{lemma:gluing_harmonic_approximations} gives $\LIP(F_{i+1})\leq\LIP(f)$.
Then, for $i\geq i_0$ and $Q\in\calQ_i$ with $Q\subseteq Q_0$, we have
(by \cref{lemma:existence_harmonic_fun_on_piece})
\begin{align*}
	cE_q(F_i-F_{i+1},Q)\leq E_q(F_{i+1},Q)-E_q(F_i,Q),
\end{align*}
where we have also used $\mu(Q)\leq\mu(Q_0)<\infty$.
Hence, for $i\geq i_0$,
\begin{equation}\label{eq:pf_harmonic_approx_gen_telescoping_sum}
\begin{aligned}
c E_q(F_i-F_{i+1},Q_0)
=\sum_{\substack{Q\in\calQ_i\colon\\ Q\subseteq Q_0}} cE_q(F_i-F_{i+1},Q)
&\leq \sum_{\substack{Q\in\calQ_i\colon\\ Q\subseteq Q_0}} E_q(F_{i+1},Q)-E_q(F_i,Q) \\
&= E_q(F_{i+1},Q_0)-E_q(F_i,Q_0),
\end{aligned}
\end{equation}
because
the collection
$\{Q\in\calQ_i\colon Q\subseteq Q_0\}$ is an essentially disjoint countable cover of $Q_0$.
Finally, the thesis follows from
\begin{equation*}
\sum_{i\geq i_0}cE_{q}(F_i-F_{i+1},Q_0)
\leq \sum_{i\geq i_0} E_{q}(F_{i+1},Q_0)-E_{q}(F_i,Q_0)
\leq \lim_i E_q(F_i,Q_0)\leq L^q\mu(Q_0).
\end{equation*}
\end{proof}
\begin{rmk}
Note that the proof of \cref{prop:harmonic_approximation_general} relies on the exact cancellation of terms in the telescopic sum \cref{eq:pf_harmonic_approx_gen_telescoping_sum}.
This is possible due to the precise constants appearing in \cref{lemma:existence_harmonic_fun_on_piece,lemma:q_UC_E_q},
which is where we use the assumption that the space has analytic dimension 1.
\end{rmk}

\section{Quantitative differentiation on Laakso spaces}\label{subsec:collapse_in_Laakso}

In this section we first apply our harmonic decomposition given in \cref{prop:harmonic_approximation_general}
to the cubical cover of Laakso spaces defined in \cref{defn:laakso_cover}.
Because of the symmetry of Laakso spaces, the approximation satisfies an additional property
(see \cref{lemma:harmonic_approximation_Laakso}, item \itemref{item:Laakso_good_harmonic_approx_item_2}),
which is crucial to examine the behaviour of Lipschitz functions on Laakso spaces.
In particular, \cref{thm:collapse_at_gates} establishes the quantitative differentiation hypothesis of \cref{corol:quant_diff_implies_LDS} for Laakso spaces.
This will give the second item of
\cref{thm:main_thm_Laakso} in \cref{thm:squashed_Laakso_V_LDS}.

Let $X=X(M;(N_n))$ be a Laakso space.
For $i\in\N$, let $\sigma_i$ be a permutation of $[M]$.
Then the map $\tilde{X}\to\tilde{X}$ which applies $\sigma_i$ to the $i$-th digit, for every $i\in\N$, descends
to an isometry $X\to X$.
We now expand upon this fact.
\begin{lemma}\label{lemma:symmetry_of_minimisers_Laakso}
Let $X=X(M;(N_n))$ be an $s$-ADR Laakso space.
Let $i\in\N$, $I\in\calI_i$ with $\partial I\subseteq W_i^h$, and $a,b\in[M]^i$
with $a(j)=b(j)$ for $1\leq j<i$ (any $a,b\in[M]$ if $i=1$).
Let $T\colon X\to X$ denote the isometry induced by the transposition $(a(i),b(i))$ acting on the $i$-th digit.
Then, for every Banach space $Y$, $1\leq q<\infty$, $f\colon X\to Y$ Lipschitz, and $u\colon Q_{I,a}\to Y$
there hold:
\begin{itemize}
\item $u\in\calA(f,Q_{I,a})$ if and only if $u\circ T\in\calA(f,Q_{I,b})$;
\item $E_q(u,Q_{I,a})=E_q(u\circ T,Q_{I,b})$, provided $u$ (equivalently, $u\circ T$) is Lipschitz.
\end{itemize}
\end{lemma}
\begin{proof}
We first focus on the first point.
The condition on the Lipschitz constant in the definition of $\calA(f,Q_{I,a})$
is preserved by $T$ because it is an isometry.
For the boundary conditions,
by \cref{lemma:Laakso_intersection_cubes}, we know that $\partial Q_{I,b}=q(\partial I\times \{b\}\times [M])$
and for
$t\in\partial I$, $w\in[M]^{\N}$, we have
\begin{equation*}
u\circ T\circ q(t,b,w) = u\circ q(t,a,w) = f \circ q(t,a,w)
= f \circ q(t,b,w),
\end{equation*}
where we have used the definition of $T$, $u\in\calA(f,Q_{I,a})$, and $q(t,a,w)=q(t,b,w)$.
This shows that $u\circ T\in\calA(f,Q_{I,b})$ whenever $u\in\calA(f,Q_{I,a})$; the reverse statement holds by symmetry of the claim.
For the last point,
using the fact that $T$ is an isometry, we have
$\Lip(u\circ T;x) = \Lip(u;T(x))$ for $\Haus^s$-a.e.\ $x\in Q_{I,a}$ and so
\begin{equation*}
E_q(u\circ T,Q_{I,b})
=\int_{T^{-1}(Q_{I,a})}\Lip(u;T(x))^q\,\dd\Haus^s(x)
=\int_{Q_{I,a}}\Lip(u;x)^q\,\dd (T_{\#}\Haus^s)(x),
\end{equation*}
concluding the proof (because $T_{\#}\Haus^s=\Haus^s$).
\end{proof}

\begin{lemma}\label{lemma:harmonic_approximation_Laakso}
Let $X=X(M;(N_n))$ be an $s$-ADR Laakso space,
$Y$ a $q$-uniformly convex Banach space, and $f\colon X\to Y$ an $L$-Lipschitz map.
Then, for $i\in\N_0$, there are $L$-Lipschitz maps $F_i\colon X\to Y$ satisfying:
\begin{enumerate}[(i)]
\item
$F_i\vert_{Q}\in\calE_q(f,Q)$ for $Q\in\calQ_{i}$;
\label{item:Laakso_good_harmonic_approx_item_1}
\item
if $i\geq 1$,
for $I\in\calI_i$ with $\partial I\subseteq W_i^h$
and $a,b\in[M]^i$ with $a(j)=b(j)$, $1\leq j<i$,
(any $a,b\in [M]$ if $i=1$)
we have
\label{item:Laakso_good_harmonic_approx_item_2}
\begin{equation*}
F_i\circ q(t,a,w)=F_i\circ q(t,b,w),
\end{equation*}
for $t\in I$ and $w\in[M]^\N$;
\item $\sum_{i\geq i_0} E_q(F_{i+1}-F_i,Q_0)\leq \tfrac{1}{2}(2K_q(Y)L)^q\Haus^s(Q_0)$ for $i_0\in\N_0$ and $Q_0\in\calQ_{i_0}$.
\label{item:Laakso_good_harmonic_approx_item_3}
\end{enumerate}
\end{lemma}
\begin{proof}
By \cref{lemma:Laakso_cubical_cover_basic,prop:harmonic_approximation_general},
item \itemref{item:Laakso_good_harmonic_approx_item_3} is implied by \itemref{item:Laakso_good_harmonic_approx_item_1}
and it is then enough to construct $(F_i)_{i\in\N_0}$ satisfying \itemref{item:Laakso_good_harmonic_approx_item_1}
and \itemref{item:Laakso_good_harmonic_approx_item_2}.
By \cref{lemma:existence_harmonic_fun_on_piece}, for $i\in\N_0$ and $Q\in\calQ_i$
the set $\calE_q(f,Q)$ is non-empty,
while, by \cref{lemma:gluing_harmonic_approximations}, for any collection $(u_Q)_{Q\in\calQ_i}$ with $u_Q\in\calE_q(f,Q)$
there is a unique $L$-Lipschitz map $F_i\colon X\to Y$ with $F_i\vert_{Q}=u_Q$ for all $Q\in\calQ_i$. \par
Let $F_0$ be defined as above, for an arbitrary choice of $(u_{Q})_{Q\in\calQ_0}=(u_X)$
(e.g.\ $F_0:=0$)
and let $i\geq 1$.
By \cref{lemma:symmetry_of_minimisers_Laakso}, there is a collection
$(u_Q)_{Q\in\calQ_i}$ with $u_Q\in\calE_q(f,Q)$ and,
if $I\in\calI_i$ and $a,b\in [M]^i$ are as in \itemref{item:Laakso_good_harmonic_approx_item_2},
we have for $Q=Q_{I,a}, Q'=Q_{I,b}\in\calQ_i$
\begin{equation*}
u_Q\circ q(t,a,w)=u_{Q'}\circ q (t,b,w),
\end{equation*}
for $t\in I$ and $w\in [M]^\N$.
Let $F_i\colon X\to Y$ be the $L$-Lipschitz map given by \cref{lemma:gluing_harmonic_approximations} with this collection $(u_Q)_{Q\in\calQ_i}$.
It is clear that it satisfies \itemref{item:Laakso_good_harmonic_approx_item_1} and \itemref{item:Laakso_good_harmonic_approx_item_2}.
\end{proof}

\begin{thm}\label{thm:collapse_at_gates}
Suppose $2\leq q<s$.
Let $(X,d,\Haus^s)$ be an $s$-ADR Laakso space and $Y$ a $q$-uniformly convex Banach space.
Then, there is a constant $C=C(X,q)>0$ such that for
any $L$-Lipschitz function $f\colon (X,d)\to Y$ and $Q\in\calQ$ we have
\begin{equation}\label{eq:collapse_at_gates}
\sum_{\substack{S\in\mathcal{J}: \\ S\subseteq Q}}
\left(\diam f(S)\right)^s
\leq
CK_q(Y)^qL^{s}\Haus^s(Q).
\end{equation}
Here $\mathcal{J}$ is the collection of shortcuts in $X$ defined in \cref{defn:shortcuts_laakso}. \end{thm}
\begin{proof}
Let $(F_i)_{i\in\N_0}$ be as in \cref{lemma:harmonic_approximation_Laakso}
and observe that, for $i\in\N$, we have
\begin{equation}\label{eq:collapse_at_gates_eq_1}
\diam(F_{i+1}-F_i)(S)=\diam f(S), \qquad S\in\calJ_i.
\end{equation}
Indeed, for $S\in\calJ_i$ and $x\in S$ there is $Q\in\calQ_{i+1}$ with $x\in\partial Q$,
thus $F_{i+1}\vert_S=f\vert_S$, while $\diam F_i(S)=0$
by \cref{lemma:harmonic_approximation_Laakso} (item \itemref{item:Laakso_good_harmonic_approx_item_2}). \par
We now prove the statement.
Fix $i_0\in\N_0$ and $Q\in\calQ_{i_0}$.
Let $i_1\in\N$ be the least integer $i\in\N$ for which there is $S\in\calJ_i$
with $S\subseteq Q$.
Since $\diam S=\delta_{i}\sim \theta^{i}$ for $S\in\calJ_i$ (see \cref{prop:Laakso_has_shortcuts} and \cref{eq:N_n_theta_n_Laakso}),
it follows that
\begin{equation}\label{eq:collapse_at_gates_comparison_theta}
	\theta^{i_1}\lesssim\diam_X Q\sim \theta^{i_0},
\end{equation}
where we have used \cref{eq:cubical_cover_cube_shape} of \cref{lemma:Laakso_cubical_cover_basic}.
Recall that, by \cref{lemma:capacitary_estimate}, there is a constant $C_1=C_1(X,q)\geq 1$ such that
\begin{equation}\label{eq:collapse_at_gates_eq_2}
(\diam(F_{i+1}-F_i)(S))^s\leq C_1L^{s-q}\int_{B(S,C_1\delta_i)}\Lip(F_{i+1}-F_i;x)^q\,\dd\Haus^s(x)
\end{equation}
for $i\in\N$ and $S\in\calJ_i$.
Since $X$ is metric doubling and $d(S,S')\geq\delta_i/2$ for $S\neq S'\in\calJ_i$ (see \cref{lemma:shortcuts_Laakso_separation}),
the collection $\{B(S,C_1\delta_i)\colon S\in\calJ_i\}$ has uniformly bounded overlap,
\begin{equation*}
\sum_{S\in\calJ_i}\chi_{B(S,C_1\delta_i)}\lesssim_{C_1,X} 1.
\end{equation*}
Hence, for $i\geq i_1$, summing \cref{eq:collapse_at_gates_eq_2} over $S\in\calJ_i$ with $S\subseteq Q$ and recalling \cref{eq:collapse_at_gates_eq_1}, we have
\begin{equation}\label{eq:collapse_at_gates_oscillation_energy_one_scale}
\begin{aligned}
\sum_{\substack{S\in\calJ_i: \\S\subseteq Q}}(\diam f(S))^s
&\lesssim_{X,q}L^{s-q}\int_{B(Q,C_1\delta_i)}\Lip(F_{i+1}-F_i;x)^q\,\dd\Haus^s(x) \\
&\leq L^{s-q}\int_{B(Q,C_2\theta^{i_0})}\Lip(F_{i+1}-F_i;x)^q\,\dd\Haus^s(x),
\end{aligned}
\end{equation}
where we have used \cref{eq:collapse_at_gates_comparison_theta} in the last line.
Define
\begin{equation*}
	B:=\bigcup\{Q'\in\calQ_{i_1}\colon Q'\cap B(Q,C_2\theta^{i_0})\neq\varnothing\}
\end{equation*}
and observe that, by \cref{eq:collapse_at_gates_comparison_theta} and \cref{eq:cubical_cover_cube_shape},
$B(Q,C_2\theta^{i_0})\subseteq B\subseteq B(Q, C_3\theta^{i_0})$.
In particular, $\Haus^s(B)\sim\Haus^s(Q)$, again by \cref{eq:cubical_cover_cube_shape}.
Summing \cref{eq:collapse_at_gates_oscillation_energy_one_scale} over $i\geq i_1$,
we have
\begin{equation*}
\sum_{\substack{S\in\calJ:\\ S\subseteq Q}}(\diam f(S))^s
\lesssim L^{s-q}\sum_{i\geq i_1}E_q(F_{i+1}-F_i,B)
\lesssim L^s K_q(Y)^q\Haus^s(B),
\end{equation*}
where the last inequality follows summing both sides of item \itemref{item:Laakso_good_harmonic_approx_item_3} of \cref{lemma:harmonic_approximation_Laakso}
over $Q'\in\calQ_{i_1}$ with $Q'\subseteq B$.
Finally, $\Haus^s(B)\sim\Haus^s(Q)$ concludes the proof.
\end{proof}
\begin{rmk}
The conclusion of \cref{thm:collapse_at_gates} fails for $s<q$ or if the Banach space $Y$
does not have a uniformly convex renorming.
Indeed, if it were to hold, the proof of \cref{thm:squashed_Laakso_V_LDS} would show
that $(X,d_\eta,\Haus^s)$ is a $Y$-LDS for any $\eta$,
contradicting the non-differentiability results of \cref{section:nowhere_diff_maps}.
Alternatively, the maps built in \cref{section:nowhere_diff_maps} fail \cref{eq:collapse_at_gates}.
\end{rmk}

\section{\Squashed\ Laakso spaces}\label{sec:shortcut_laakso}
In this short section we combine the results of the previous section to prove the first two points of \cref{thm:main_thm_Laakso}.

\begin{defn}\label{defn:shortcut_laakso}
Let $(X,d,\Haus^s)$ be an $s$-ADR Laakso space, $\eta=(\eta_i)_i\subseteq (0,1]$, and define $d_\eta$ as
in \cref{defn:shortcut_dist}.
We call $(X,d_\eta,\Haus^s)$ a \emph{\squashed\ Laakso space}.
\end{defn}
Recall that \cref{thm:pure_PI_unrectifiability} and \cref{thm:PI_rectifiability} characterise sequences $\eta=(\eta_i)_i\subseteq (0,1]$ for
which $(X,d_\eta,\Haus^s)$ is PI rectifiable or purely PI unrectifiable.
To prove \cref{thm:main_thm_Laakso}, we will choose $\eta$ so that the latter holds.
For instance, any $\eta_i\rightarrow 0$ with $\sum_i \eta_i^s = \infty$ works.

\begin{thm}\label{thm:squashed_Laakso_V_LDS}
Suppose $2\leq q<s$.
Let $(X,d_\eta,\Haus^s)$ be a $s$-ADR \squashed\ Laakso space and $Y$ a $q$-uniformly convex Banach space.
Then $(X,d_\eta,\Haus^s)$ is a $Y$-LDS with global chart $h\colon (X,d_\eta)\to\R$.
\end{thm}
\begin{proof}
By \cref{thm:vanilla_Laakso} and \cref{prop:Laakso_has_shortcuts}, $(X,d,\Haus^s)$ is an $s$-ADR $Y$-LDS with shortcuts compatible with the
global chart $h\colon X\to\R$.
The thesis then follows from \cref{corol:quant_diff_implies_LDS} and \cref{thm:collapse_at_gates} (and $d_\eta\leq d$).
\end{proof}

\section{Almost nowhere differentiable Lipschitz maps}\label{section:nowhere_diff_maps}
In this section, we construct almost nowhere differentiable Lipschitz functions on those \squashed\ Laakso spaces which are purely PI unrectifiable.
These functions take values in a variety of Banach spaces $Y$, depending on the precise setting.
In particular, we will prove the third and fourth items of \cref{thm:main_thm_Laakso}.

\cref{thm:differentiability_set} shows that a function $f\colon (X,d_\eta)\to Y$ is almost nowhere differentiable if it
separate sufficiently many \jumpsets.
\cref{corol:01_law_shortcuts} makes `sufficiently many' precise.
In the first lemma, we reduce the construction of a single map $f\colon (X,d_\eta)\to Y$ to a sequence of building blocks $f_n\colon X\to Y$.

\begin{lemma}\label{lemma:gluing_building_blocks}
Let $(X,d,\Haus^s)$ be an $s$-ADR Laakso space,
$\eta=(\eta_i)_i\subseteq (0,1]$,
$Y$ a Banach space, $C\geq 1$, and $(I_n)_{n\in\calN}$ a disjoint family of subsets of $\N$.
Suppose there are maps $f_n\colon X\to Y$, $n\in\calN$, such that:
\begin{enumerate}[(i)]
\item
\label{item:gluing_building_blocks_1}
for $S\in \bigcup_{i\in I_n}\calJ_i$,
\begin{equation*}
C^{-1}\diam_\eta S\leq \diam f_n(S)\leq C\diam_\eta S;
\end{equation*}
\item
\label{item:gluing_building_blocks_2}
for $S\in \calJ$, $S\notin \bigcup_{i\in I_n}\calJ_i$,
\begin{equation*}
\diam f_n(S)=0;
\end{equation*}
\item
\label{item:gluing_building_blocks_3}
$\sum_{n\in\calN} \LIP_X(f_n)<\infty$ and $\sum_{n\in\calN} \sum_{i\in I_n}\eta_i^s=\infty$.
\end{enumerate}
Then no positive-measure measurable subset of $(X,d_\eta,\Haus^s)$ is a $Y$-LDS.
\end{lemma}
\begin{proof}
By \cref{prop:non_RNP_LDS}, we may assume $Y$ has RNP.
Possibly adding a constant to $f_n$, we can assume there is $x_0\in X$ such that $f_n(x_0)=0$ for all $n\in\calN$.
Hence, from \itemref{item:gluing_building_blocks_3}, $f:=\sum_{n\in\calN} f_n$ is well-defined and a Lipschitz map from $(X,d)$ to $Y$.
Set $I:=\bigcup_{n\in\calN} I_n$.
We are going to use the following properties of $f$,
implied by items \itemref{item:gluing_building_blocks_1} and \itemref{item:gluing_building_blocks_2}
\begin{equation}\label{eq:gluing_buliding_blocks_1}
\begin{aligned}
\diam f(S)
&\leq C\diam_\eta S, & S&\in\calJ, \\
\diam f(S)
&\geq C^{-1}\diam_\eta S, & S&\in\bigcup_{i\in I}\calJ_i.
\end{aligned}
\end{equation}
We first show that $f$ is Lipschitz also as a map from $(X,d_\eta)$.
Set $L:=\LIP_X(f)<\infty$ and let $x,y\in X$.
If $d_\eta(x,y)=d(x,y)$, then $\|f(x)-f(y)\|_Y\leq Ld_\eta(x,y)$.
Suppose $d_\eta(x,y)<d(x,y)$ and let $p_-\neq p_+\in S\in \calJ_i$ be such that
\begin{equation}\label{eq:gluing_buliding_blocks_2}
4d_\eta(x,y)\geq d(x,p_-)+\rho_\eta(p_-,p_+)+d(p_-,y),
\end{equation}
where we used \cref{prop:single_jump} and \cref{prop:Laakso_has_shortcuts}.
By triangle inequality, \cref{eq:gluing_buliding_blocks_1}, and \cref{eq:gluing_buliding_blocks_2}, we have
\begin{align*}
\|f(x)-f(y)\|_Y
&\leq Ld(x,p_-)+\|f(p_-)-f(p_+)\|_Y+Ld(p_-,y) \\
&\lesssim Ld(x,p_-)+\rho_\eta(p_-,p_+)+Ld(p_-,y)\lesssim (L+1) d_\eta(x,y),
\end{align*}
showing that $f\colon (X,d_\eta)\to Y$ is Lipschitz.
Observe that \cref{eq:gluing_buliding_blocks_1} implies, for all sufficiently small $\epsilon>0$,
\begin{equation*}
	\Badf_\epsilon(f)\supseteq \bigcup_{i\in I} B_\eta (\cup \calJ_i,\delta_i\eta_i/\epsilon).
\end{equation*}
From \itemref{item:gluing_building_blocks_3} and \cref{corol:01_law_shortcuts}, it follows that $\Badf_\epsilon(f)$
has full $\Haus^s$-measure for all sufficiently small $\epsilon>0$
and therefore $\Haus^s(X\setminus\Badf(f))=0$.
The thesis then follows from \cref{prop:LDS_unrect_subsets_Bad}
(and \cref{lemma:doubling_meas_and_porosity}).
\end{proof}

We may explicitly construct real valued building blocks as follows.
\begin{lemma}\label{lemma:building_block_R_valued}
Let $i\in\N$
and $X_i=X_i(M;(N_n))$.
There is a $1$-Lipschitz function $f\colon X_i\to \R$ satisfying:
\begin{enumerate}[(i)]
\item $\diam f(S)= \diam S$ for $S\in \calJ_i$; \label{item:building_block_R_valued_item1}
\item $\diam f(S)=0$ for $S\in\bigcup_{j<i} \calJ_j$; \label{item:building_block_R_valued_item2}
\item $f\circ q\vert_{I\times \{a\}}\colon I\times\{a\}\to\R$ is affine linear for $I\in\calI_{i+1}$ and $a\in [M]^i$; \label{item:building_block_R_valued_item3}
\item $f\circ q\vert_{\partial I\times [M]^i}=0$ for $I\in\calI_i$.\label{item:building_block_R_valued_item4}
\end{enumerate}
Here $(\calJ_j)_{j=1}^i$ are the families of subsets of $X_i$ defined in \cref{defn:shortcuts_laakso}.
\end{lemma}
\begin{proof}
We first construct $\tilde{f}\colon \tilde{X}_i\to\R$.
Let $I\in \calI_i$ and $a\in [M]^{i-1}$.
If $\partial I \nsubseteq W_i^h$, set $\tilde{f}\vert_{I\times\{a\}\times [M]}:=0$.
Suppose now $\partial I \subseteq W_i^h$ and let $t_0$ denote the midpoint of $I$.
We define
\begin{equation*}
	\tilde{f}(t,a',M):=-\tilde{f}(t,a',1):= \left(\tfrac{|I|}{2}-|t-t_0|\right)^+, \qquad t\in I,
\end{equation*}
and $\tilde{f}\vert_{I\times \{(a',a(i))\}}:=0$ for $a(i)\neq 1,M$,
where $t^+:=\max(0,t)$ denotes the positive part of $t\in\R$.
Since $\tilde{f}\vert_{\partial I\times \{a\}}=0$ for $I\in\calI_i$ and $a\in [M]^i$, $\tilde{f}\colon \tilde{X}_i\to\R$ is well-defined
and induces a unique function $f\colon X_i\to\R$ satisfying $f\circ q = \tilde{f}$.
It is easy to verify that $f$ has the desired properties.
\end{proof}

We also construct building blocks into $\R^2$.
\begin{lemma}\label{lemma:induction_step_bad_fun_in_R}
Let $i\in\N$ and $X_i=X_i(M;(N_n))$.
Let $f\colon X_i\to\R^2$ a function such that, for $J\in\calI_i$ and $a\in [M]^i$,
\begin{equation*}
	f\circ q\vert_{J\times\{a\}}\colon J\times\{a\}\to\R^2  \text{ is affine linear.}
\end{equation*}
Set $L:=\LIP(f)$.
Then, for $A>0$, there is a function $F\colon X_i\to\R^2$ satisfying:
\begin{enumerate}[(i)]
	\item $\diam F(S)=\diam f(S)$ for $S\in\bigcup_{j<i}\calJ_j$;
		\label{item:induction_step_bad_fun_in_R_1}
	\item $\diam F(S)=A\diam S$ for $S\in\calJ_i$;
		\label{item:induction_step_bad_fun_in_R_2}
	\item $\LIP(F)^2\leq L^2+A^2$;
		\label{item:induction_step_bad_fun_in_R_3}
	\item $F\circ q \vert_{J\times \{a\}}\colon J\times\{a\}\to\R^2$ is affine linear for $J\in\calI_{i+1}$ and $a\in [M]^i$;
		\label{item:induction_step_bad_fun_in_R_4}
	\item $F\circ q\vert_{\partial J\times [M]^i}=f\circ q \vert_{\partial J\times [M]^i}$ for $J\in\calI_i$.
		\label{item:induction_step_bad_fun_in_R_5}
\end{enumerate}
Here $(\calJ_j)_{j=1}^i$ are the families of subsets of $X_i$ defined in \cref{defn:shortcuts_laakso}.
\end{lemma}
\begin{proof}
We first establish a few properties of $f$.
We can apply \cref{lemma:LIP_on_pieces_to_LIP} to see that $f$ is in in fact Lipschitz, i.e.\ $L<\infty$.
Next, let $J\in\calI_i$ with $\partial J\subseteq W^h_i$ and pick $a\in [M]^{i-1}$.
Then, for $a_1(i),a_2(i)\in [M]$,  $f\circ q\vert_{J\times\{(a,a_1(i))\}},f\circ q\vert_{J\times\{(a,a_2(i))\}}$ are two affine functions agreeing at the endpoints of the interval $J$.
Hence, there is an affine function $\alpha_{J,a}\colon J\to\R^2$ such that
\begin{equation}\label{eq:induction_step_bad_fun_in_R_1}
	f\circ q(t,a,a_1(i))=f\circ q(t,a,a_2(i))=\alpha_{J,a}(t), \qquad t\in J
\end{equation}
and, in particular, $\diam f(S)=0$ for $S\in\calJ_i$. \par
We now describe the construction of $F$.
For each $J\in\calI_i$ and $a\in [M]^{i-1}$, we will define a function $F_{J,a}\colon J\times\{a\}\times [M]\to \R^2$
with the property 
\begin{equation}\label{eq:induction_step_bad_fun_in_R_2}
F_{J,a}\vert_{\partial J\times\{a\}\times [M]}=f\circ q\vert_{\partial J\times\{a\}\times [M]}.
\end{equation}
The function in the thesis will then be the unique $F\colon X_i\to\R^2$ satisfying
$F\circ q\vert_{J\times\{a\}\times [M]}=F_{J,a}$ for all $J\in\calI_i$ and $a\in [M]^{i-1}$;
its existence easily follows from \cref{eq:induction_step_bad_fun_in_R_2}.
Also, \eqref{eq:induction_step_bad_fun_in_R_2} implies \itemref{item:induction_step_bad_fun_in_R_5},
which in turn gives \itemref{item:induction_step_bad_fun_in_R_1}, because 
for $S\in\calJ_j$, $j<i$, there is $J\in\calI_i$ with $S\subseteq q(\partial J\times [M]^i)$.
By \cref{lemma:LIP_on_pieces_to_LIP}, we know that \itemref{item:induction_step_bad_fun_in_R_3} holds as soon as 
\begin{equation}\label{eq:induction_step_bad_fun_in_R_3}
\LIP(F_{J,a})^2\leq L^2+A^2
\end{equation}
for all $J,a$.
Hence, during the construction of $\{F_{J,a}\}$, we need only to ensure the validity of \cref{eq:induction_step_bad_fun_in_R_2}, \cref{eq:induction_step_bad_fun_in_R_3} and items
\itemref{item:induction_step_bad_fun_in_R_2}, \itemref{item:induction_step_bad_fun_in_R_4}.
 \par
We are ready to proceed with the construction.
Let $g\colon X_i\to\R$ be the $1$-Lipschitz function given by \cref{lemma:building_block_R_valued}.
Let $J\in\calI_i$.
If $\partial J \nsubseteq W_i^h$, then $J\cap \calJ_i^h=\varnothing$ and so $q(J\times [M]^i)\cap S=\varnothing$ for $S\in \calJ_i$.
We set $F_{J,a}:=f\circ q \vert_{J\times\{a\}\times [M]}$ for each $a\in [M]^{i-1}$.
With this choice, \cref{eq:induction_step_bad_fun_in_R_2}, \cref{eq:induction_step_bad_fun_in_R_3} and
items \itemref{item:induction_step_bad_fun_in_R_2}, \itemref{item:induction_step_bad_fun_in_R_4} are clearly satisfied. \par
Suppose now $\partial J\subseteq W^h_i$ and let $a\in [M]^{i-1}$.
Let $\alpha_{J,a}\colon J\to\R^2$ be as in \cref{eq:induction_step_bad_fun_in_R_1} and
choose $v_{J,a}\in\Sp^1$ such that 
\begin{equation}\label{eq:induction_step_bad_fun_in_R_4}
	\langle v_{J,a}, \alpha_{J,a}(t_1)-\alpha_{J,a}(t_2)\rangle=0, \qquad t_1,t_2\in J.
\end{equation}
(Here $\langle\cdot,\cdot\rangle$ denotes the standard scalar product on $\R^2$.)
We set $F_{J,a}:=(f\circ q+Ag\circ q v_{J,a})\vert_{J\times\{a\}\times [M]}$.
By \cref{lemma:building_block_R_valued} we know that $g$ satisfies the condition of item \itemref{item:induction_step_bad_fun_in_R_4}
and $g\circ q\vert_{\partial J\times [M]^i}=0$.
It follows that $F_{J,a}$ satisfies \itemref{item:induction_step_bad_fun_in_R_4} and \cref{eq:induction_step_bad_fun_in_R_2}.
By \cref{lemma:building_block_R_valued} and \cref{eq:induction_step_bad_fun_in_R_1} it is clear that \itemref{item:induction_step_bad_fun_in_R_2} is also satisfied.
It remains to show \cref{eq:induction_step_bad_fun_in_R_3}.
Let $t_1,t_2 \in J$, $a_1(i),a_2(i)\in [M]$, and set $x_k:=(t_k,a,a_k(i))$, $1\leq k\leq 2$.
Then \cref{eq:induction_step_bad_fun_in_R_1} and \cref{eq:induction_step_bad_fun_in_R_4} imply
\begin{align*}
\|F_{J,a}(t_1,a_1(i))-F_{J,a}(t_2,a_2(i))\|^2
&= \|f\circ q(x_1)-f\circ q(x_2)\|^2 + A^2\|g\circ q(x_1)-g\circ q(x_2)\|^2 \\
&\leq (L^2+A^2)d(q(x_1),q(x_2))^2.
\end{align*}
As discussed, this concludes the proof.
\end{proof}
\begin{lemma}\label{lemma:bad_functions_in_R2}
There is a constant $C\geq 1$ such that the following holds.
For any Laakso space $X$, $\eta=(\eta_i)_i\subseteq (0,1]$, and
non-empty $I\subseteq\N$, there is a function $f\colon X\to\R^2$ satisfying:
\begin{enumerate}[(i)]
\item $C^{-1}\leq \diam f(S)/\diam_\eta S\leq C$ for $S\in\bigcup_{i\in I}\calJ_i$;
	\label{item:bad_functions_in_R2_1}
\item $\diam f(S)=0$ for $S\in\calJ$, $S\notin\bigcup_{i\in I}\calJ_i$;
	\label{item:bad_functions_in_R2_2}
\item $\LIP_X(f)^2\leq\sum_{i\in I}\eta_i^2$.
	\label{item:bad_functions_in_R2_3}
\end{enumerate}
\end{lemma}
\begin{proof}
We first establish the following claim.
For any (possibly finite) strictly increasing sequence $(i_k)_k\subseteq\N$
there are functions $f_k\colon X_{i_k}\to\R^2$ satisfying:
\begin{enumerate}
	\item $\diam f_k(S)=\eta_{i_l}\diam S$ for $S\in\calJ_{i_l}$, $l\leq k$;
	\item $\diam f_k(S)=0$ for $S\in\bigcup_{j\leq i_k}\calJ_j$, $S\notin\bigcup_{l\leq k}\calJ_{i_l}$;
	\item $\LIP(f_k)^2\leq \sum_{l\leq k}\eta_{i_l}^2$;
	\item $f_k\circ q\vert_{J\times\{a\}}$ is affine for $J\in\calI_{i_k+1}$ and $a\in [M]^{i_k}$;
	\item $f_k\circ q\vert_{\partial J\times [M]^{i_k}}=f_{k-1}\circ\pi^{i_k}_{i_{k-1}}\circ q\vert_{\partial J\times [M]^{i_k}}$ for $J\in\calI_{i_k}$.
\end{enumerate}
If $g\colon X_{i_1}\to\R$ is the function given by \cref{lemma:building_block_R_valued},
then $f_1:=(\eta_{i_1}g,0)\colon X_{i_1}\to\R^2$ has all the desired properties.
Hence, if $(i_k)_k\subseteq\N$ is a (possibly finite) strictly increasing sequence of $N$ terms, $N\in\N\cup\{\infty\}$,
we can at least find $f_1\colon X_{i_1}\to\R^2$.
Suppose now $1<n\leq N$, $n\in\N$, and that we can find $f_1,\dots,f_{n-1}$ as in the claim.
We show that we can then construct $f_n\colon X_{i_n}\to\R^2$ satisfying (1-5).
Since $i_{n-1}<i_n$, it is not difficult to see that $f_{n-1}\circ \pi^{i_n}_{i_{n-1}}\colon X_{i_n}\to\R^2$ satisfies the assumption
of \cref{lemma:induction_step_bad_fun_in_R} with $i\equiv i_n$.
Taking $A\equiv \eta_{i_n}$, the lemma gives a function $f_n\colon X_{i_n}\to\R^2$ satisfying (1-5), hence proving the claim. \par
We can now prove the lemma.
If $I$ is finite, $I=\{i_1<\cdots <i_n\}$, we can take $f:=f_n\circ \pi^\infty_{i_n}$, where $f_n$ is as in the claim.
Suppose now $I$ is infinite and let $i_k\uparrow \infty$ be such that $I=\{i_k\colon k\in\N\}$. 
We can assume 
\begin{equation}\label{eq:bad_functions_in_R2_1}
	L:=\sum_{k\in\N}\eta_{i_k}^2<\infty,
\end{equation}
for otherwise the statement is trivial.
Let $\tilde{f}_k\colon X_{i_k}\to\R^2$ be the function provided by the claim and set $f_k:=\tilde{f}_k\circ \pi^\infty_{i_k}\colon X\to\R^2$.
By item (5), $(f_k)$ converges pointwise on the set
\begin{equation*}
	E:=q\left(\left\{\frac{m}{N_1\cdots N_n}\colon n\in\N, m\in\N_0, 0\leq m\leq N_1\cdots N_n\right\}\times [M]^\N\right).
\end{equation*}
By item (3) and \cref{eq:bad_functions_in_R2_1}, $(f_k)$ is equicontinuous.
Since $E$ is dense in $X$, it follows that $f_k$ converges pointwise on $X$ to an $L$-Lipschitz function $f\colon X\to\R^2$.
By pointwise convergence and (1-2) (or by (1-2) and (5)), we conclude that $f$ satisfies also \itemref{item:bad_functions_in_R2_1},\itemref{item:bad_functions_in_R2_2}.
\end{proof}

We are now in a position to construct non-differentiable Banach space valued functions on shortcut Laakso spaces.

For $q\in [1,\infty]$ and $n\in\N$, we denote $\ell_q^n=(\R^n,\|\cdot\|_{\ell_q})$.
We say that a Banach space $Y$ \emph{contains $\ell_q^n$ uniformly}
if there are $C\geq 1$ and, for $n\in\N$, linear operators $T_n\colon \ell_q^n\to Y$
satisfying
\begin{equation}\label{eq:contains_unif}
	C^{-1}\|x\|_{\ell_q^n}\leq\|T_n x\|_Y\leq C\|x\|_{\ell_q^n}
\end{equation}
for all $x\in\ell_q^n$ and $n\in\N$.
By a theorem of Krivine \cite{Krivine_1976_reticules}
(see \cite[10.5]{Milman_Schechtman_asy_theory_of_finite_dim_normed_spaces} for the precise statement we use),
we may equivalently require \cref{eq:contains_unif} for all $C>1$.
\begin{prop}\label{prop:non_diff_LIP_in_lq}
Suppose $2\leq s<q<\infty$.
Let $(X,d_\eta,\Haus^s)$ be an $s$-ADR \squashed\ Laakso space, assume it is purely PI unrectifiable,
and let $Y$ be a Banach space containing $\ell_q^n$ uniformly (e.g.\ $Y=\ell_q$).
Then, no positive-measure measurable subset of $(X,d_\eta,\Haus^s)$ is a $Y$-LDS.
\end{prop}
\begin{proof}
By \cref{lemma:summable_subsequence}, we find
countably many finite non-empty disjoint sets $I_n\subseteq\N$ such that
\begin{equation}\label{eq:non_diff_LIP_in_lq_1}
\sum_{n\in\N}\Big(\sum_{i\in I_n}\eta_i^q\Big)^{1/q}<\infty, \qquad
\sum_{n\in\N}\sum_{i\in I_n}\eta_i^s=\infty.
\end{equation}
For $n\in\N$, denote with $N_n$ the cardinality of $I_n$.
We will first construct a sequence of functions $\tilde{f}_n\colon X\to \ell_q^{N_n}$ satisfying items
\itemref{item:gluing_building_blocks_1}, \itemref{item:gluing_building_blocks_2}, \itemref{item:gluing_building_blocks_3}
of \cref{lemma:gluing_building_blocks}.
It will then be easy to produce a sequence $f_n\colon X\to Y$ with the same property. \par
For $n\in\N$ and
$i\in I_n$, let $g_i\colon X_i\to \R$ be given by \cref{lemma:building_block_R_valued} and set
$\tilde{f}_n:=(\eta_ig_i\circ\pi_i^\infty)_{i\in I_n}\colon X\to (\R^{I_n},\|\cdot\|_{\ell_q})\equiv \ell_q^{N_n}$.
Items \itemref{item:gluing_building_blocks_2} and \itemref{item:gluing_building_blocks_3} follow easily from
\cref{lemma:building_block_R_valued}, the definition of $\tilde{f}_n$, and \cref{eq:non_diff_LIP_in_lq_1}.
Recall that by \cref{lemma:diam_squashed_jump_sets} and \cref{prop:Laakso_has_shortcuts},
it holds 
\begin{equation*}
\diam_\eta S\sim \eta_i\delta_i\sim \eta_i \diam S, \qquad S\in \calJ_i.
\end{equation*}
Hence, also item \itemref{item:gluing_building_blocks_1} follows from \cref{lemma:building_block_R_valued} and the definition of $\tilde{f}_n$. \par
Since $N_n$ is finite and $Y$ contains $\ell_q^n$ uniformly, there are a constant $C\geq 1$
and linear operators $T_n\colon \ell_q^{N_n}\to Y$ satisfying
\cref{eq:contains_unif}.
It is then clear that the maps $f_n:=T_n\circ \tilde{f}_n\colon X\to Y$ satisfy the assumptions of \cref{lemma:gluing_building_blocks},
from which the thesis follows.
\end{proof}

\cref{prop:non_diff_LIP_in_lq} also holds if $1<s<2$.
However, for this range we can establish a much stronger conclusion,
namely that no positive-measure measurable subset of $(X,d_\eta,\Haus^s)$ is an LDS.

\begin{prop}\label{prop:non_LDS_dim_less_2}
Suppose $1<s<2$.
Let $(X,d_\eta,\Haus^s)$ be an $s$-ADR \squashed\ Laakso space and assume it is purely PI unrectifiable.
Then, no positive-measure measurable subset of $(X,d_\eta,\Haus^s)$ is an LDS.
\end{prop}
\begin{proof}
By \cref{lemma:summable_subsequence}, $\eta=(\eta_i)$ has a subsequence $(\eta_{i_j})_j$
satisfying
\begin{equation*}
	\sum_{j\in\N}\eta_{i_j}^2<\infty, \qquad \sum_{j\in\N}\eta_{i_j}^s=\infty.
\end{equation*}
Let $I:=\{i_j\colon j\in\N\}$ and $f\colon X\to\R^2$ be given by \cref{lemma:bad_functions_in_R2}.
The single map $f$ and set $I$ satisfy the assumptions of \cref{lemma:gluing_building_blocks},
concluding the proof.
\end{proof}
The following is our final non-differentiability result.
\begin{prop}\label{prop:non_diff_map_non_superreflexive}
Let $(X, d_\eta, \Haus^s)$ be an $s$-ADR \squashed\ Laakso space, suppose it is purely PI unrectifiable,
and let $Y$ be a non-superreflexive Banach space.
Then, no positive-measure measurable subset of $(X,d_\eta,\Haus^s)$ is a $Y$-LDS.
\end{prop}
The proof relies on the following theorem, whose proof is postponed to the next section.
\begin{thm}\label{lemma:non_diff_map_non_superreflexive_building_blocks}
Let $Y$ be a Banach space and $X$ a Laakso space.
Then $Y$ is non-superreflexive if and only if there is $C\geq 1$ such that
for any finite non-empty set $I\subseteq\N$ there is a $1$-Lipschitz function $f\colon X\to Y$ satisfying:
\begin{enumerate}[(i)]
	\item $\diam f(S)\geq C^{-1}\diam S$ for $S\in\bigcup_{i\in I}\calJ_i$;
		\label{item:lemma_non_diff_map_non_superreflexive_building_blocks_item_1}
	\item $\diam f(S)=0$ for $S\in\calJ$, $S\notin\bigcup_{i\in I}\calJ_i$.
\end{enumerate}
Moreover, if $Y$ is non-superreflexive, the above holds for any $C>8$.
\end{thm}
\begin{proof}[Proof of \cref{prop:non_diff_map_non_superreflexive}]
By \cref{lemma:summable_subsequence}, there is a subsequence $(\eta_{i_j})_j$ such that $\eta_{i_j}\rightarrow 0$ and $\sum_j\eta_{i_j}^s=\infty$.
For $n\in\N$, set
\begin{equation*}
	I_n:=\{i_j\colon 2^{-n}<\eta_{i_j}\leq 2^{1-n}\}
\end{equation*}
and note that they are all finite sets.
Let $\tilde{f}_n$ be the map given by \cref{lemma:non_diff_map_non_superreflexive_building_blocks} with $I\equiv I_n$
and 
define $f_n:=2^{-n}\tilde{f}_n$ (or $f_n:=0$ if $I_n=\varnothing$).
To conclude, it is enough to verify that $f_n$ and $I_n$ satisfy the assumptions of \cref{lemma:gluing_building_blocks}.
Let $n\in\N$.
If $i\in I_n$ and $S\in\calJ_i$, we have
\begin{equation*}
\diam f_n(S)\sim \eta_i \diam S\sim \diam_\eta S,
\end{equation*}
while if $i\notin I_n$, $S\in\calJ_i$, \cref{lemma:gluing_building_blocks} gives $\diam f_n(S)=0$.
Lastly, $\sum_n \LIP_X(f_n)\leq 1$ and
\begin{equation*}
	\sum_{n\in\N}\sum_{i\in I_n}\eta_i^s=\sum_{j\in\N}\eta_{i_j}^s=\infty
\end{equation*}
conclude the proof.
\end{proof}

\section{A characterisation of non-superreflexive spaces}\label{subsec:non_superreflexive_bblocks}
The goal of this section is to prove \cref{lemma:non_diff_map_non_superreflexive_building_blocks}.
We exploit well-known characterisations of non-superreflexivity in terms of biLipschitz embeddability of families of graphs,
first introduced by Bourgain \cite{Bourgain_superreflexivity_trees}; see also \cite{Baudier_superreflexivity_tree}.
More specifically, we (indirectly) use a technique introduced in \cite{Ostrovskii_Randrianantoanina_Laakso_graphs}
and later applied in \cite{Andrew_Swift_bundle_graphs} to a larger family of graphs.

The above works consider `Laakso graphs'.
However, in spite of the shared name, none of the above works deal with the Laakso spaces defined in \cref{sec:Laakso_pi} nor their finite approximations $G_n$ in \cref{subsec:Laakso_as_graphs}.
Indeed, \cite{Ostrovskii_Randrianantoanina_Laakso_graphs,Andrew_Swift_bundle_graphs} consider
the family of graphs introduced in \cite[Section 2]{lang_plaut_2001_biLip_embeddings} which is inspired by \cite{MR1924353} and
are often called `Laakso graphs'; their Gromov-Hausdorff limit is also sometimes called `Laakso space'.
However, the spaces we are considering, those of \cite{laakso2000_ADR_PI}, are not directly related to those of \cite{MR1924353}.
For instance, all of the latter graphs
are planar, while $G_n$ is non-planar for $n\geq 2$ and $M\geq 3$ (since it contains $K_{3,3}$).
Nonetheless, the techniques of \cite{Ostrovskii_Randrianantoanina_Laakso_graphs} can be used to construct the map required in \cref{lemma:non_diff_map_non_superreflexive_building_blocks}.
Instead of doing so directly, we shall reduce the construction of our map to the embedding of a graph simpler than $G_n$,
for which \cite{Andrew_Swift_bundle_graphs} can be applied. \par
We refer the reader to \cite[Definitions 2.1 and 2.4]{Andrew_Swift_bundle_graphs} for the definition of $M$-branching bundle graph
and to \cite{Diestel_graph_theory} for graph theory background.
We stress that, compared to \cite{Andrew_Swift_bundle_graphs}, we adopt a slightly different vertex set, with the only difference being a different scaling of the `height' coordinate.
We also scale the distance accordingly.

\subsection{Diamond graphs}
We define a collection of graphs which we call \emph{diamond graphs}.
Let $n,M\in\N$, $M\geq 2$, $N_1,\dots, N_{n}\in 2\N$ and set
\begin{equation}\label{eq:defn_W_j_hat}
\begin{aligned}
\widehat{W}_1^h&:=\widehat{W}_1^h(N_1,\dots,N_n):=W_1^h(N_1,\dots,N_n)\cup\{0,1\}, \\
\widehat{W}_j^h&:=\widehat{W}_j^h(N_1,\dots,N_n):=W_j^h(N_1,\dots,N_n), \\
\widehat{W}_{\leq l}^h&:=\widehat{W}_{\leq l}^h(N_1,\dots,N_n):=\bigcup_{i=1}^l\widehat{W}_i^h \\
&=\left\{\frac{m}{N_1\cdots N_l}\colon m\in\N_0,0\leq m\leq N_1\dots N_l\right\}
\end{aligned}
\end{equation}
for $1<j\leq n$ and $1\leq l\leq n$.
For $1\leq j\leq n$ and $t\in\widehat{W}_j^h$, set $w_t:=j-1$.
We then denote with 
\begin{equation}\label{eq:M_branching_diamond}
\widehat{G}(M;N_1,\dots, N_n)
\end{equation}
the $M$-branching bundle graph associated with $(w_t)$; see \cite[Definition 2.9]{Andrew_Swift_bundle_graphs}. \par
Our first main goal is to show that \cite{Andrew_Swift_bundle_graphs} implies the existence of biLipschitz embeddings of our diamond graphs
with uniformly bounded distortion.
Swift proves that suitable graphs, including our diamond graphs $\widehat{G}$, admit biLipschitz embeddings in non-superreflexive Banach space.
The distortion of such embeddings depends
on a geometric parameter $p_{\widehat{G}}$ of the graph, which we now define.\par
Let $\widehat{G}=\widehat{G}(M;N_1,\dots, N_n)$ and define for $t\in\widehat{W}^h_{\leq n}$, $t=\sum_{i=1}^n\tfrac{t_i}{N_1\cdots N_i}$, and $1\leq l\leq n$
\begin{equation}\label{eq:x_y_Swift}
\begin{aligned}
	x(t,l)&:= t-d([0,t]\cap \widehat{W}^h_{\leq l},t)=\sum_{1\leq i\leq l}\frac{t_i}{N_1 \cdots N_i}, \\
	y(t,l)&:= t+d([t,1]\cap \widehat{W}^h_{\leq l},t)=\frac{1}{N_1\cdots N_l}+\sum_{1\leq i\leq l}\frac{t_i}{N_1 \cdots N_i},
\end{aligned}
\end{equation}
and $x(t,0):=0$, $y(t,0):=1$.
We also set $x(t,l):=y(t,l):=t$ for $l>n$\footnote{%
The definitions of $x(t,l),y(t,l)$ we adopt differs slightly from the one of \cite{Andrew_Swift_bundle_graphs}, but
it is not difficult to see that they are equivalent.
}.
Observe that equality between the leftmost and rightmost sides of \cref{eq:x_y_Swift} holds even for $l=0$, provided we interpret empty sums as $0$ and empty products as $1$.
Let $p_{\widehat{G}}$ be the least $p\in\N$ such that for all $0\leq j<n$ and $t\in \widehat{W}_{\leq n}^h$
\begin{itemize}
	\item $t\geq (x(t,j)+y(t,j))/2$ implies $x(t,j+p)\geq (x(t,j)+y(t,j))/2$;
	\item $t\leq (x(t,j)+y(t,j))/2$ implies $y(t,j+p)\leq (x(t,j)+y(t,j))/2$.
\end{itemize} 
We shall verify that $p_{\widehat{G}}=1$.
\begin{rmk}\label{rmk:prelim_p_G_bundle_graphs}
For $t\in\widehat{W}_j^h$ it holds $x(t,l)=y(t,l)=t$, $j\leq l\leq n$.
We may then restrict our attention to $1\leq l+1<j\leq n$.
Also, since $N_{l+1}$ is even, $(x(t,l)+y(t,l))/2\in \widehat{W}^h_{l+1}$ for $0\leq l<n$
and in particular $t\neq (x(t,l)+y(t,l))/2$ whenever $t\notin \widehat{W}_{l+1}^h$.
\end{rmk}
\begin{lemma}\label{lemma:prelim_p_G_bundle_graphs}
Let $n,M\in\N$, $M\geq 2$, $N_1,\dots, N_{n}\in 2\N$,
and recall the definitions \cref{eq:defn_W_j_hat,eq:x_y_Swift}.
Let $1\leq l+1<j\leq n$ and $t\in\widehat{W}_j^h$, $t=\sum_{i=1}^j\tfrac{t_i}{N_1\cdots N_i}$.
Then
$t>(x(t,l)+y(t,l))/2$ if and only if $t_{l+1}\geq N_{l+1}/2$.
\end{lemma}
\begin{proof}
Under the present assumptions $t\neq (x(t,l)+y(t,l))/2$; see \cref{rmk:prelim_p_G_bundle_graphs}.
To treat the cases $l=0$ and $l>0$ in a unified way, we adopt the conventions on empty sums and product mentioned after \cref{eq:x_y_Swift}.
Observe that $t\geq (x(t,l)+y(t,l))/2$ if and only if
\begin{equation*}
\sum_{l<i\leq j}\frac{t_i}{N_1\cdots N_i}\geq \frac{1}{2N_1\cdots N_l}.
\end{equation*}
Now, if $t_{l+1}<N_{l+1}/2$, then
\begin{align*}
\sum_{l<i\leq j}\frac{t_i}{N_1\cdots N_i}
&\leq \frac{t_{l+1}}{N_1\cdots N_{l+1}}+\sum_{i=l+2}^j\frac{N_i-1}{N_1\cdots N_i}
<\frac{t_{l+1}+1}{N_1\cdots N_{l+1}}\leq \frac{1}{2N_1\cdots N_l},
\end{align*}
proving $t<(x(t,l)+y(t,l))/2$.
If instead $t_{l+1}\geq N_{l+1}/2$, then
\begin{equation*}
\sum_{l<i\leq j}\frac{t_i}{N_1\cdots N_i}\geq \frac{1}{2N_1\cdots N_l},
\end{equation*}
proving $t>(x(t,l)+y(t,l))/2$.
\end{proof}
\begin{lemma}\label{lemma:bundle_p_G_1}
Let $n,M\in\N$, $M\geq 2$, $N_1,\dots, N_{n}\in 2\N$,
consider $\widehat{G}=\widehat{G}(M;N_1,\dots, N_n)$,
and recall the definitions \cref{eq:defn_W_j_hat,eq:x_y_Swift}.
Let $0\leq l<n$ and $t\in\widehat{W}_{\leq n}^h$. Then
\begin{itemize}
\item $t\geq (x(t,l)+y(t,l))/2$ implies $x(t,l+1)\geq (x(t,l)+y(t,l))/2$;
\item $t\leq (x(t,l)+y(t,l))/2$ implies $y(t,l+1)\leq (x(t,l)+y(t,l))/2$.
\end{itemize}
Equivalently, $p_{\widehat{G}}=1$.
\begin{proof}
Let $1\leq j\leq n$ be such that $t\in\widehat{W}_j^h$ and write $t=\sum_{i=1}^j\frac{t_i}{N_1\cdots N_i}$.
Assume w.l.o.g.\ $j>l+1$ and hence $t\neq (x(t,l)+y(t,l))/2$; see \cref{rmk:prelim_p_G_bundle_graphs}.
The thesis will follow from \cref{eq:x_y_Swift}, \cref{lemma:prelim_p_G_bundle_graphs} and some simple computations.
We adopt the conventions on empty sums and product mentioned after \cref{eq:x_y_Swift}.
If $t>(x(t,l)+y(t,l))/2$, then $t_{l+1}\geq N_{l+1}/2$ and therefore
\begin{equation*}
x(t,l+1)=\sum_{1\leq i\leq l+1}\frac{t_i}{N_1\cdots N_i}
\geq
\frac{1}{2N_1\cdots N_l} + \sum_{1\leq i\leq l} \frac{t_i}{N_1\cdots N_1}
=(x(t,l)+y(t,l))/2.
\end{equation*}
If instead $t<(x(t,l)+y(t,l))/2$, then $t_{l+1}<N_{l+1}/2$ and therefore
\begin{equation*}
y(t,l+1)=\frac{1}{N_1\cdots N_{l+1}}+\sum_{1\leq i\leq l+1}\frac{t_i}{N_1\cdots N_i}
\leq
\frac{1}{2N_1\cdots N_l} + \sum_{1\leq i\leq l}\frac{t_i}{N_1\cdots N_i}=(x(t,l)+y(t,l))/2.
\end{equation*}
Recalling the definition of $p_{\widehat{G}}$ (see the paragraph after \cref{eq:x_y_Swift}), we immediately have $p_{\widehat{G}}=1$.
\end{proof}
\end{lemma}
\begin{thm}[{\cite[Theorem 5.4]{Andrew_Swift_bundle_graphs}}]\label{thm:embedding_diamonds}
Let $n,M\in\N$, $M\geq 2$, $N_1,\dots, N_n\in 2\N$, and let $\widehat{G}=\widehat{G}(M;N_1,\dots,N_n)$.
Then, for any non-superreflexive Banach space $Y$ and $\epsilon>0$, there is a map $f\colon \widehat{G}\to Y$ such that
\begin{equation}\label{eq:embedding_diamonds}
	(8+\epsilon)^{-1}d_{\widehat{G}}(u,v)\leq \|f(u)-f(v)\|_Y\leq d_{\widehat{G}}(u,v),
\end{equation}
for $u,v$ in $\widehat{G}$.
\end{thm}
\begin{proof}
The thesis follows from \cite[Theorem 5.4]{Andrew_Swift_bundle_graphs} and \cite{Brunel_Sucheston_J_convexity}, as we now describe.
We refer to \cite[Definition 5.1]{Andrew_Swift_bundle_graphs} or \cite[Definition 2.1]{Ostrovskii_Randrianantoanina_Laakso_graphs} for the definition of equal-sign-additive (ESA) basis
and \cite[Chapter 11]{Pisier_martingales_in_Banach_2016} for finite representability. \par
Since $Y$ is non-superreflexive, there is a non-reflexive Banach space $Z_1$ which is finitely representable in $Y$.
By \cite{Brunel_Sucheston_J_convexity} (see the paragraph after \cite[Theorem 2.3]{Ostrovskii_Randrianantoanina_Laakso_graphs} for details),
there is a Banach space $Z$ with an equal-sign-additive (ESA) basis which is finitely representable in $Z_1$
and hence in $Y$.
Thus, \cite[Theorem 5.4]{Andrew_Swift_bundle_graphs} and \cref{lemma:bundle_p_G_1} ensure the existence of a map $g\colon \widehat{G}\to Z$ satisfying
\begin{equation*}
	\frac{1}{8}d_{\widehat{G}}(u,v)\leq \|g(u)-g(v)\|_Z\leq d_{\widehat{G}}(u,v),
\end{equation*}
for $u,v$ in $\widehat{G}$.
Since $\widehat{G}$ is finite, the thesis follows by definition of finite representability.
\end{proof}
To prove \cref{lemma:non_diff_map_non_superreflexive_building_blocks}, we need to define suitable `projections' between diamond graphs
and establish some basic properties. \par
Let $n,M\geq 2$, $N_1,\dots,N_n\in 2\N$ and let $(\widehat{V},\widehat{E})=\widehat{G}=\widehat{G}(M;N_1,\dots, N_n)$ denote the corresponding diamond graph.
Let $k\in\N$ and $I=\{i_1<\cdots <i_k\}$ be a subset of $[n-1]$.
For $t\in \widehat{W}^h_{\leq n}$, set $w_{I,t}:=\#\{j\in [k]\colon i_j\leq w_t\}$, 
\begin{equation*}
	\widehat{V}_I:=\bigcup\left\{{\{t\}\times [M]^{w_{I,t}}\colon t\in \widehat{W}_{\leq n}^h}\right\}
\end{equation*}
and
\begin{equation}\label{eq:diamond_graph_restriction_0}
\pi_I\colon \widehat{V}\to\widehat{V}_I, \qquad \pi_I(t,a):=(t,a\vert_{I}).
\end{equation}
Set $i_0:=0$, $i_{k+1}:=n$, and for $1\leq j\leq k+1$
\begin{equation}\label{eq:diamond_graph_restriction_1}
\widehat{W}_{I,j}^h:=\bigcup_{l=i_{j-1}+1}^{i_j}\widehat{W}_l^h,
\qquad
N_{I,j}:=\prod_{i=i_{j-1}+1}^{i_j}N_i.
\end{equation}
Observe that $N_{I,1},\dots, N_{I,k+1}\in 2\N$ and
\begin{equation*}
	\widehat{W}_{I,j}^h=\widehat{W}_j^h(N_{I,1},\dots, N_{I,k+1}).
\end{equation*}
It is then clear that $w_{I,t}=j-1$ for $t\in \widehat{W}_{I,j}^h$, $1\leq j\leq k+1$.
Hence, $\widehat{V}_I$ coincides with the vertex set of $\widehat{G}(M;N_{I,1},\dots,N_{I,k+1})$.
Let $\widehat{E}_I$ denote the smallest edge set for which $\pi_I\colon \widehat{G}\to (\widehat{V}_I,\widehat{E}_I)$ is a graph homomorphism
and define
\begin{equation}\label{eq:diamond_graph_restriction_2}
	\widehat{G}_I:=\widehat{G}(M;N_1,\dots,N_n)_I:=(\widehat{V}_I,\widehat{E}_I).
\end{equation}
\begin{lemma}\label{lemma:diamond_graph_restriction_is_diamond}
Let $n,M\in\N$, $n,M\geq 2$, $N_1,\dots, N_n\in 2\N$, and let $\widehat{G}=\widehat{G}(M;N_1,\dots, N_n)$.
Let $I=\{i_1<\cdots<i_{k}\} \subseteq [n-1]$ be a non-empty set and define
$N_{I,1},\dots, N_{I,k+1}\in 2\N$ and $\widehat{G}_I$ as in \cref{eq:diamond_graph_restriction_1} and \cref{eq:diamond_graph_restriction_2} respectively.
Then 
\begin{equation*}
	\widehat{G}_I=\widehat{G}(M;N_1,\dots, N_n)_I=\widehat{G}(M;N_{I,1},\dots, N_{I,k+1}).
\end{equation*}
\end{lemma}
\begin{proof}
We use the notation $\preceq$ from \cite[Section 2]{Andrew_Swift_bundle_graphs}.
Let $u=(t,a),v=(s,b)\in \widehat{V}_I$.
We may assume $t\in\widehat{W}_{l(t)}^h(N_1,\dots, N_n)$, $s\in\widehat{W}_{l(s)}^h(N_1,\dots,N_n)$ with 
$1\leq l(t)\leq l(s)\leq n$.
Suppose $u,v$ are adjacent in $\widehat{G}(M;N_{I,1},\dots,N_{I,k+1})$, i.e.\
$|t-s|=\tfrac{1}{N_{I,1}\cdots N_{I,k+1}}=\tfrac{1}{N_1\cdots N_n}$ and $a\preceq b$ or $b\preceq a$.
Since $l(t)\leq l(s)$, we have $w_{I,l(t)}\leq w_{I,l(s)}$, and therefore it must be $a\preceq b$.
The conditions $l(t)\leq l(s)$ and $a\preceq b$ guarantee the existence of $a_0,b_0$ with $a_0\preceq b_0$ or $b_0\preceq a_0$
and $\pi_I(t,a)=(t,a_0)$, $\pi_I(s,b)=(s,b_0)$,
proving that $u,v$ are adjacent also in $\widehat{G}_I$.
Suppose now $u,v$ are adjacent in $\widehat{G}_I$, i.e.\
there are $(t,a_0),(s,b_0)$ adjacent in $\widehat{G}$ such that $u=\pi_I(t,a_0)$ and $v=\pi_I(s,b_0)$.
Then $|t-s|=\tfrac{1}{N_{I,1}\cdots N_{I,k+1}}=\tfrac{1}{N_1\cdots N_n}$, $a_0\preceq b_0$,
hence $a\preceq b$ and therefore $u,v$ are adjacent in $\widehat{G}(M;N_{I,1},\dots,N_{I,k+1})$.
\end{proof}
Recall the notation introduced in \cref{eq:shortcuts_explicit_notation}.
\begin{lemma}\label{lemma:inclusion_jump_heights}
Let $n,M\in\N$, $n,M\geq 2$, and $N_1,\dots, N_n\in 2\N$. 
Let $I=\{i_1<\cdots <i_k\}\subseteq [n-1]$ be a non-empty set and define $N_{I,1},\dots, N_{I,k+1}$ as in \cref{eq:diamond_graph_restriction_1}.
Then, for $1\leq j\leq k$, we have 
\begin{equation*}
\calJ_{i_j}^h(N_1,\dots, N_{i_j})\subseteq \calJ_j^h(N_{I,1},\dots,N_{I,j}).
\end{equation*}
\end{lemma}
\begin{proof}
Let $t\in\calJ_{i_j}^h(N_1,\dots,N_{i_j})$ and $t_i\in\N_0$, $0\leq t_i<N_i$ for $1\leq i<i_j$ and $0<t_{i_j}<N_{i_j}-1$ such that
\begin{equation*}
	t=\sum_{i=1}^{i_j}\frac{t_i}{N_1\cdots N_i}+\frac{1}{2N_1\cdots N_{i_j}}.
\end{equation*}
There are unique $\tau_l\in\N_0$, $0\leq\tau_l<N_{I,l}$, satisfying
\begin{equation*}
	\frac{\tau_l}{N_{I,l}}=\sum_{i=i_{l-1}+1}^{i_l}\frac{t_i}{N_1\cdots N_i}
\end{equation*}
for $1\leq l\leq j$.
Since $0<t_{i_j}<N_{i_j}-1$, we have $\tau_j>0$ and
\begin{equation*}
	\tau_j=\left[\sum_{i=i_{j-1}+1}^{i_j}\frac{t_i}{N_1\cdots N_i}\right](N_{i_{j-1}+1}\cdots N_{i_j})<N_{i_{j-1}+1}\cdots N_{i_j}-1,
\end{equation*}
i.e.\ $\tau_j<N_{I,j}-1$.
Hence, $t$ can be written as
\begin{equation*}
	t=\sum_{l=1}^j\frac{\tau_l}{N_{I,1}\cdots N_{I,l}}+\frac{1}{2N_{I,1}\cdots N_{I,j}}
\end{equation*}
with $0\leq \tau_l<N_{I,l}$ and $0<\tau_j<N_{I,j}-1$, proving $t\in\calJ^h_j(N_{I,1},\dots, N_{I,j})$.
\end{proof}
Let $n,M\in\N$, $n,M\geq 2$, $N_1,\dots, N_n\in 2\N$, and $\widehat{G}=\widehat{G}(M;N_1,\dots,N_n)$.
We then set, for $1\leq i\leq n-1$,
\begin{equation}\label{eq:diamond_graph_jump_sets}
\widehat{\calJ}_i:=\widehat{\calJ}_i(M;N_1,\dots,N_i):=\big\{\{t\}\times\{a\}\times [M]\colon t\in\calJ_i^h(N_1,\dots, N_i) \text{ and } a\in [M]^{i-1}\big\}.
\end{equation}
\begin{lemma}\label{lemma:diamond_graphs_distance_jump_sets}
Let $n,M\in\N$, $n,M\geq 2$, $N_1,\dots, N_n\in 2\N$, and $\widehat{G}=\widehat{G}(M;N_1,\dots,N_n)$.
For $1\leq i\leq n-1$, $u\neq v\in S\in\widehat{\calJ}_i$, we have
\begin{equation*}
	d_{\widehat{G}}(u,v)=\frac{1}{N_1\cdots N_i}.
\end{equation*}
\end{lemma}
\begin{proof}
Write $u=(t,a,u(i))$, $v=(t,a,v(i))$, for some $a\in [M]^{i-1}$ and $t\in\calJ_i^h\subseteq\widehat{W}_{i+1}^h$.
Recall that in the definition of $\calJ_i^h$ we require $t\pm\tfrac{1}{2N_1\cdots N_i}\in W_i^h$
and so, for $1\leq j\leq i$,
\begin{equation*}
	d_\R(t,\widehat{W}_j^h)\geq\frac{1}{2N_1\cdots N_i}=d_\R(t,\widehat{W}_i^h)=d_\R(t,W_i^h).
\end{equation*}
Since $u,v$ differ exactly in the $i$-th digit,
from \cite[Proposition 2.9]{Andrew_Swift_bundle_graphs} we then have
\begin{equation*}
	d_{\widehat{G}}(u,v)=\frac{2}{2N_1\cdots N_i}=\frac{1}{N_1\cdots N_i}.
\end{equation*}
\end{proof}
\begin{lemma}\label{lemma:diamond_graphs_projection_jump_sets}
Let $n,M\in\N$, $n,M\geq 2$, $N_1,\dots, N_n\in 2\N$, and $\widehat{G}=\widehat{G}(M;N_1,\dots, N_n)$.
Let $I\subseteq [n-1]$ be a non-empty set and let $\widehat{G}_I$ and $\pi_I\colon \widehat{G}\to\widehat{G}_I$ be as in
\cref{eq:diamond_graph_restriction_2} and \cref{eq:diamond_graph_restriction_0} respectively.
Then the map $\pi_I$ is $1$-Lipschitz and
for $1\leq i\leq n-1$, $u,v\in S\in \widehat{\calJ}_i$, it holds
\begin{equation*}
d_{\widehat{G}_I}(\pi_I(u),\pi_I(v))=
\left\{
\begin{array}{cc}
	d_{\widehat{G}}(u,v), & i\in I \\
	0, & i\notin I 
\end{array}
\right.
.
\end{equation*}
\end{lemma}
\begin{proof}
Since $\pi_I$ is a graph homomorphism, it follows that it is $1$-Lipschitz.
Let $I=\{i_1<\cdots <i_k\}$ and $N_{I,1},\dots, N_{I,k+1}$ be as in \cref{eq:diamond_graph_restriction_1}.
We can assume $u\neq v$ and write
$u=(t,a,u(i))$, $v=(t,a,v(i))$, for some $a\in [M]^{i-1}$ and $t\in\calJ_i^h(N_1,\dots, N_i)$.
Suppose $i\in I$ and let $1\leq j\leq k$ be such that $i=i_j$.
Then $\pi_I(S)=\{(t,a\vert_{I})\}\times [M]$ with $a\vert_{I}\in [M]^{j-1}$ and $t\in\calJ_j^h(N_{I,1},\dots,N_{I,j})$
(by \cref{lemma:inclusion_jump_heights}).
Hence, $\pi_I(u)\neq\pi_I(v)\in\pi_I(S)\in\widehat{\calJ}_j(M;N_{I,1},\dots,N_{I,j})$ and so
\cref{lemma:diamond_graphs_distance_jump_sets} and \cref{eq:diamond_graph_restriction_1} yield
\begin{equation*}
d_{\widehat{G}_I}(\pi_I(u),\pi_I(v))=\frac{1}{N_{I,1}\cdots N_{I,j}}=\frac{1}{N_1\cdots N_{i_j}}=d_{\widehat{G}}(u,v).
\end{equation*}
If $i\notin I$, 
from $1\leq i\leq n-1$ we find $1\leq j\leq k+1$ such that $i_{j-1}<i<i_j$ and therefore
$\pi_I(u)=(t,a\vert_{I})=\pi_I(v)$.
\end{proof}
\subsection{Projection of Laakso graphs onto diamond graphs}
We use the following notation for the rest of the section.
Fix a Laakso space $X=X(M;(N_n))$ and, for $n\in\N$, let
$G_n$ be the graph having $X_n=X_n(M;N_1,\dots,N_{n+1})$ as metric graph (see \cref{subsec:Laakso_as_graphs}),
and set $\widehat{G}_n=\widehat{G}(M;N_1,\dots,N_{n+1})$ (\cref{eq:M_branching_diamond}).
We let $\widehat{\pi}_n$ denote the obvious projection from the vertex set of $G_n$ onto the one of $\widehat{G}_n$.
For a non-empty $I\subseteq [n]$, set $\widehat{G}_{n,I}:=(\widehat{G}_n)_I$ (\cref{eq:diamond_graph_restriction_2})
and $\widehat{\pi}_{n,I}:=\pi_I\circ \widehat{\pi}_n$, where $\pi_I$ is as in \cref{eq:diamond_graph_restriction_0}.
\begin{lemma}\label{lemma:projection_of_Laakso_onto_diamond}
Let $n\in\N$ and $I\subseteq[n]$ be a non-empty set.
Let $G_n$, $\widehat{G}_{n,I}$, and $\widehat{\pi}_{n,I}\colon G_n\to\widehat{G}_{n,I}$ be as above.
Then the map $\widehat{\pi}_{n,I}$ is $1$-Lipschitz and for $1\leq i\leq n$ and $u,v\in S\in\calJ_i$, we have
\begin{equation*}
d_{\widehat{G}_{n,I}}(\widehat{\pi}_{n,I}(u),\widehat{\pi}_{n,I}(v))
=
\left\{
\begin{array}{cc}
	d_{G_n}(u,v), & i\in I \\
	0, & i\notin I 
\end{array}
\right.
.
\end{equation*}
\end{lemma}
\begin{proof}
It is not difficult to see that $\widehat{\pi}_n\colon G_n\to\widehat{G}_n$ is a graph homomorphism and therefore $1$-Lipschitz.
Thus, \cref{lemma:diamond_graphs_projection_jump_sets} shows that $\widehat{\pi}_{n,I}$ is $1$-Lipschitz.
Let $S$ be as in the statement and assume w.l.o.g.\ $u\neq v\in S$.
From the definitions of $G_n$, $\widehat{G}_n$, and $\widehat{\pi}_n\colon G_n\to\widehat{G}_n$, we have
$\widehat{\pi}_n(u)\neq\widehat{\pi}_n(v)\in \widehat{\pi}_n(S)\in\widehat{\calJ}_i(M;N_1,\dots,N_i)$ (\cref{eq:diamond_graph_jump_sets}).
Hence, \cref{lemma:diamond_graphs_distance_jump_sets} and (the argument of) \cref{lemma:shortcuts_Laakso_scale} give
\begin{equation*}
d_{\widehat{G}_n}(\widehat{\pi}_n(u),\widehat{\pi}_n(v))=\frac{1}{N_1\cdots N_i}=d_{G_n}(u,v).
\end{equation*}
Finally, \cref{lemma:diamond_graphs_projection_jump_sets} concludes the proof.
\end{proof}
\begin{proof}[Proof of \cref{lemma:non_diff_map_non_superreflexive_building_blocks}]
Suppose first $Y$ is non-superreflexive,
set $n:=\max(I)$, and let $G_n,\widehat{G}_{n,I}$, and $\widehat{\pi}_{n,I}\colon G_n\to\widehat{G}_{n,I}$ be as above.
By \cref{lemma:diamond_graph_restriction_is_diamond} and \cref{thm:embedding_diamonds}, we deduce the
existence of a map $\widehat{f}\colon \widehat{G}_{n,I}\to Y$ for which \cref{eq:embedding_diamonds} holds.
Set $\tilde{f}:=\widehat{f}\circ\widehat{\pi}_{n,I}\colon G_n\to Y$ and observe that, by \cref{lemma:projection_of_Laakso_onto_diamond}, it is $1$-Lipschitz
and satisfies both items of \cref{lemma:non_diff_map_non_superreflexive_building_blocks} (on the graph) for $1\leq i\leq n$.
With small abuse of notation, let $\tilde{f}\colon X_n\to Y$ denote the extension of $\tilde{f}\vert_{G_n}$ to $X_n$ given by \cref{lemma:extension_from_graph},
and finally set $f:=\tilde{f}\circ \pi_n^\infty\colon X\to Y$.
Since $\diam \pi_n^\infty(S)=0$ for $S\in\bigcup_{i>n}\calJ_i$, the map $f$ is as claimed.
\par
For the converse direction, we use the differentiability theory of Cheeger-Kleiner \cite{cheeger_kleiner_PI_RNP_LDS}
and a characterisation of superreflexivity via ultrapowers.
We refer the reader to \cite{Heinrich_ultraproducts_in_Bsp_theory} for background on ultrafilters, ultrapowers,
and their applications to Banach space theory. \par
Suppose $Y$ is a Banach space, $C\geq 1$, and that, for $n\in\N$,
there is a $1$-Lipschitz map $f_n\colon X\to Y$ satisfying \itemref{item:lemma_non_diff_map_non_superreflexive_building_blocks_item_1} with $I=[n]$.
We may assume w.l.o.g.\ $\sup_n\|f_n\|_\infty<\infty$.
Let $\calU$ be a non-principal ultrafilter on $\N$, let $Y^\calU$ denote the corresponding ultrapower of $Y$,
and define $f\colon X\to Y^\calU$
as $f(x):=[(f_n(x))]$,
where, for bounded $(y_n)\subseteq Y$, $[(y_n)]\in Y^\calU$ denotes its equivalence class.
Since $\calU$ is non-principal and $\|[(y_n)]\|_{Y^\calU}=\lim_\calU\|y_n\|_Y$ for $(y_n)\subseteq Y$ bounded,
it is not difficult to verify that $f$ is $1$-Lipschitz and $\diam f(S)\geq C^{-1}\diam S$
for $S\in\calJ$.
Then, from \cref{lemma:gluing_building_blocks} with the single map $f$ and $\eta=(1,1,\dots)$,
we deduce that no positive-measure measurable subset of $(X,d,\Haus^s)$ is a $Y^\calU$-LDS,
where $s$ denotes the Hausdorff dimension of $X$.
Since $(X,d,\Haus^s)$ is a PI space, by \cite[Theorem 1.5]{cheeger_kleiner_PI_RNP_LDS}
we see that $Y^\calU$ does not have RNP and therefore it is not reflexive \cite[Corollary 5.12(ii)]{benyamini_lindenstraus_geom_nonlinear_FA}.
Finally, by \cite[Proposition 6.4]{Heinrich_ultraproducts_in_Bsp_theory},
we conclude that $Y$ is non-superreflexive.
\end{proof}
\section{Lipschitz dimension and biLipschitz embeddability}\label{sec:LIP_dim}
The notion of Lipschitz dimension was introduced in \cite{Cheeger_Kleiner_LIP_dimension}, where Cheeger and Kleiner prove
a structure theorem for metric spaces of Lipschitz dimension at most $1$ (and, more generally, of Lipschitz light maps)
and their biLipschitz embeddability into $L^1$ spaces.
It was further studied in \cite{Guy_C_David_LIP_dim_CK, Guy_C_David_Non_inj_Assouad_thm}. \par
In this section we prove that any \squashed\ Laakso space $(X,d_\eta)$
has topological, Nagata, and Lipschitz dimension $1$ (\cref{prop:Lipschitz_Nagata_dim}).
As a corollary, all such spaces admit a biLipschitz embedding in $L^1([0,1])$,
and a short argument yields the non-embeddability claimed in \cref{thm:main_thm_Laakso}.
\par
We begin focusing on Lipschitz dimension.
\begin{defn}
Given $r\in (0,\infty)$, a finite sequence $(x_0,\dots,x_n)$ of points in $X$ is an \emph{$r$-chain} or \emph{$r$-path (from $x_0$ to $x_n$)}
if $d(x_{i-1},x_i)\leq r$ for $1\leq i\leq n$.
Two points $x,y\in X$ are \emph{$r$-connected} if there is an $r$-path from $x$ to $y$.
This defines an equivalence relation on $X$, whose equivalence classes are called \emph{$r$-components}.
\end{defn}
\begin{defn}\label{defn:LIP_light}
Let $X,Y$ be metric spaces, $f\colon X\to Y$, and $C\in (0,\infty)$.
We say that $f$ is \emph{$C$-Lipschitz light} if
\begin{itemize}
\item $f$ is $C$-Lipschitz, and
\item for every $E\subseteq Y$ and $r\geq \diam E$, the $r$-components of $f^{-1}(E)$
	have diameter at most $Cr$.
\end{itemize}
We say that $f$ is Lipschitz light if it is $C$-Lipschitz light for some $C\in (0,\infty)$.
\end{defn}
\begin{rmk}
\cref{defn:LIP_light} differs from \cite[Definition 1.14]{Cheeger_Kleiner_LIP_dimension},
but is equivalent for maps into $\R^n$, $n\in\N$; see \cite[Section 1.3]{Guy_C_David_LIP_dim_CK}.
\end{rmk}
\begin{rmk}\label{rmk:LIP_light_closed_convex}
Let $f\colon X\to Y$ be Lipschitz.
To see if $f$ is Lipschitz light, it suffices to check if the second condition of \cref{defn:LIP_light}
is satisfied for closed sets $E\subseteq Y$
or, if $Y$ is a normed space, for $E$ closed and convex.
In particular, if $Y=\R$, it is enough to consider compact intervals.
\end{rmk}
\begin{defn}[{\cite[Definition 1.15]{Cheeger_Kleiner_LIP_dimension}}]\label{defn:LIP_dim}
The \emph{Lipschitz dimension} of a metric space $X$ is defined as
\begin{equation*}
\dim_LX:=\inf\{n\in\N_0\colon \text{there exists a Lipschitz light map }f\colon X\to\R^n\}.
\end{equation*}
\end{defn}
The following lemma will soon be useful.
\begin{lemma}\label{lemma:union_r_components}
For $n\in\N$ and $C_1,\dots,C_n>0$ there is a constant $C=C(C_1,\dots,C_n)>0$ with the following property.
Let $X$ be a metric space, $A_1,\dots, A_n\subseteq X$ non-empty, and $r_0>0$.
Suppose that for $1\leq i\leq n$ and $r\geq r_0$, every $r$-component of $A_i$ has diameter at most $C_i r$.
Then, for $r\geq r_0$, every $r$-component of $\bigcup_{i=1}^nA_i$ has diameter at most $Cr$.
\end{lemma}
\begin{proof}
By induction on $n\in\N$, it is enough to consider the case $n=2$.
Let $r\geq r_0$ and $(x_t)_{t=0}^m=(x(t))_{t=0}^m$ be an $r$-path in $A_1\cup A_2$;
we need to show that $d(x_0,x_m)\leq C(C_1,C_2)r$.
We may suppose w.l.o.g.\ $x_0\in A_1$ and $\{x_t\colon 0\leq t\leq m\}\cap A_2\neq\varnothing$.
Then, there is $k\in\N$ and integers $0=t_0<\cdots <t_{k-1}<m=:t_k-1$ such that, for $1\leq i\leq k$,
$\{x_t\colon t_{i-1}\leq t<t_i\}\subseteq A_1$ if $i$ is odd and $\{x_t\colon t_{i-1}\leq t<t_i\}\subseteq A_2$ otherwise.
That is, $t_i$ is the time at which $(x_t)_t$ changes set.
In particular, for $1\leq i\leq k$, $(x_t\colon t_{i-1}\leq t<t_i)$ is an $r$-path fully contained in either $A_1$ or $A_2$,
thus
\begin{equation}\label{eq:union_r_components_1}
d(x(t_{i-1}),x(t_i-1))\leq C_0r,
\end{equation}
where $C_0:=\max(C_1,C_2)$.
Therefore, for $0\leq i<k/2-1$,
\begin{align*}
d(x(t_{2i}),x(t_{2i+2}))
&\leq
d(x(t_{2i}),x(t_{2i+1}-1))+d(x(t_{2i+1}-1),x(t_{2i+1}))+d(x(t_{2i+1}),x(t_{2i+2}-1)) \\
&\qquad+d(x(t_{2i+2}-1),x(t_{2i+2}))
\leq
2(C_0+1)r;
\end{align*}
that is, $(x(t_{2i})\colon 0\leq i<k/2-1)$ is a $2(C_0+1)r$-path in $A_1$.
Hence, setting $T:=\max\{0\leq i\leq k-1\colon i\in 2\N_0\}$, we deduce
\begin{equation}\label{eq:union_r_components_2}
d(x(t_0),x(t_T))\leq 2C_0(C_0+1)r.
\end{equation}
If $k$ is odd, then $T=k-1$, and by \cref{eq:union_r_components_2,eq:union_r_components_1}, we have
\begin{equation*}
d(x_0,x_m)=d(x(t_0),x(t_k-1))\leq d(x(t_0),x(t_{k-1}))+d(x(t_{k-1}),x(t_k-1))
\leq C_0(2C_0+3)r.
\end{equation*}
If $k$ is even, then $T=k-2$, and we similarly obtain from \cref{eq:union_r_components_1,eq:union_r_components_2}
\begin{align*}
d(x_0,x_m)
&\leq
d(x(t_0),x(t_{k-2}))+d(x(t_{k-2}),x(t_{k-1}-1))+d(x(t_{k-1}-1),x(t_{k-1})) \\
&\qquad+d(x(t_{k-1}),x(t_k-1)) \\
&\leq 2C_0(C_0+1)r+C_0r+r+C_0r
=(2C_0^2+4C_0+1)r.
\end{align*}
\end{proof}
\begin{lemma}\label{lemma:Laakso_basic_separation}
Let $(X,d_\eta)$ be a \squashed\ Laakso space.
For $j\in\N$ and $x,y\in X$ with $x(j)\neq y(j)$,
we have
\begin{align*}
d_\eta(x,y)
&\geq \tfrac{1}{3}\min\big\{2d(h\{x,y\},W_j^h),2d(\{x,y\},\cup\calJ_j)+\eta_j\delta_j\big\}.
\end{align*}
\end{lemma}
\begin{proof}
If $d_\eta(x,y)=d(x,y)$, by \cref{prop:dist_in_Laakso} we have
\begin{equation}\label{eq:Laakso_basic_separation_1}
d_\eta(x,y)\geq 2d(h\{x,y\},W_j^h).
\end{equation}
Suppose now $d_\eta(x,y)<d(x,y)$, $\epsilon>0$, and let $i\in\N$, $p_-\neq p_+\in S\in\calJ_i$ be such that
$(3+\epsilon)d_\eta(x,y)\geq d(x,p_-)+\eta_i\delta_i+d(p_+,y)$;
they exist by \cref{prop:single_jump} and \cref{prop:Laakso_has_shortcuts}.
If $i=j$, then 
\begin{equation}\label{eq:Laakso_basic_separation_2}
(3+\epsilon)d_\eta(x,y)\geq 2d(\{x,y\},\cup\calJ_j)+\eta_j\delta_j.
\end{equation}
If $i=j-1$, then $h(p_\pm)\in\calJ_i^h\subseteq W_j^h$ and so
\begin{equation}\label{eq:Laakso_basic_separation_3}
(3+\epsilon)d_\eta(x,y)\geq 2d(h\{x,y\},W_j^h).
\end{equation}
If $i\neq j-1,j$, then $p_-(j)=p_+(j)$ and therefore either $x(j)\neq p_\pm(j)$ or $y(j)\neq p_\pm(j)$.
Hence, by \cref{prop:dist_in_Laakso}, we have
\begin{equation}\label{eq:Laakso_basic_separation_4}
(3+\epsilon)d_\eta(x,y)\geq \max\{d(x,p_-),d(p_+,y)\}\geq 2d(h\{x,y\},W_j^h)
\end{equation}
Finally, combining \cref{eq:Laakso_basic_separation_1,eq:Laakso_basic_separation_2,eq:Laakso_basic_separation_3,eq:Laakso_basic_separation_4},
we have the thesis.
\end{proof}
\begin{rmk}\label{rmk:LIP_light_equivalence_cubes}
Let $n\in\N$ and $I\in\calI_n$.
Then, for $1\leq j\leq n$, $W_j^h\cap I = W_j^h\cap\partial I$,
and moreover $W_n^h\cap I\neq\varnothing$.
\end{rmk}
\begin{lemma}\label{lemma:LIP_light_L_J_card_1}
Let $X$ be a Laakso space.
For $n\geq 2$ and $I\in\calI_n$, it holds
\begin{equation*}
	\#\{l\in [n-1]\colon \calJ_l^h\cap I\neq\varnothing\}\leq 1.
\end{equation*}
\end{lemma}
\begin{proof}
Let $\calL\subseteq [n-1]$ be the set of the thesis,
suppose $\calL\neq\varnothing$, and set $l:=\min\calL$.
If $l=n-1$, then $\#\calL=1$, so we may assume $n\geq 3$ and $1\leq l\leq n-2$.
Since $\calJ_l^h\subseteq W_{l+1}^h$ and $l+1<n$,
from \cref{rmk:LIP_light_equivalence_cubes} we deduce that
either $\calJ_{n-1}^h\cap\partial I =\varnothing$ or $\calL=\{l\}$.
Let $t\in I$ be such that
$\calJ_l^h\cap I = \calJ_l^h\cap \partial I = \{t\}$.
From the inclusion $\partial I \subseteq\{t-\tfrac{1}{N_1\cdots N_n}, t, t+\tfrac{1}{N_1\cdots N_n}\}$,
it is then enough to show $t_\pm:=t\pm \tfrac{1}{N_1\cdots N_n}\notin\calJ_{n-1}^h$.
Write $t=\sum_{i=1}^n\frac{t_i}{N_1\cdots N_i}$ with $t_i\in\N_0$, $0\leq t_i<N_i$,
and recall that
$t\in\calJ_l^h$ is equivalent to $1\leq t_l<N_l-1$, $t_{l+1}=N_{l+1}/2$, and $t_i=0$ for $l+2\leq i\leq n$.
Define $t_{\pm,i}$ in the analogous way.
Then $t_n=0$ implies $t_{+,n}=1<N_n/2$, proving $t_+\notin\calJ_{n-1}^h$.
For $t_-$, we have
\begin{align*}
t_-
&=\sum_{i=1}^l\frac{t_i}{N_1\cdots N_i}+\frac{N_{l+1}/2-1}{N_1\cdots N_{l+1}}+\frac{1}{N_1\cdots N_{l+1}}-\frac{1}{N_1\cdots N_n} \\
&=\sum_{i=1}^l\frac{t_i}{N_1\cdots N_i}+\frac{N_{l+1}/2-1}{N_1\cdots N_{l+1}}+\sum_{i=l+2}^n\frac{N_i-1}{N_1\cdots N_i}
\end{align*}
and so $t_{-,n}=N_n-1>N_n/2$, proving $t_-\notin\calJ_{n-1}^h$.
\end{proof}
\begin{lemma}\label{lemma:LIP_light_L_J_W_disjoint}
Let $X$ be a Laakso space.
Let $n\geq 2$, $I\in\calI_n$, $l\in [n-1]$ and suppose $\calJ_l^h\cap I\neq\varnothing$.
Then $W_l^h\cap I =\varnothing$.
\end{lemma}
\begin{proof}
From \cref{rmk:LIP_light_equivalence_cubes} we see that,
if $n\geq 3$ and $l+1<n$, then $\varnothing\neq \calJ_l^h\cap I \subseteq W_{l+1}^h\cap \partial I$
and so $\partial I \setminus\calJ_l^h\subseteq W_{n}^h$.
Suppose $n\geq 2$, $l=n-1$, and let $t\in I$ be such that $\calJ_l^h\cap \partial I =\{t\}$.
Arguing as in \cref{lemma:LIP_light_L_J_card_1}, it is not
difficult to see that $t\pm \tfrac{1}{N_1\cdots N_n}\notin W_{n-1}^h$.
Since $\partial I \subseteq \{t-\tfrac{1}{N_1\cdots N_n},t,t+\tfrac{1}{N_1\cdots N_n}\}$,
this implies $\varnothing=W_{n-1}^h\cap \partial I=W_{n-1}^h\cap I$ and concludes the proof.
\end{proof}
\begin{notation}\label{notation:LIP_light_equivalence_cubes}
Let $X=X(M;(N_n))$ be a Laakso space and $\eta=(\eta_i)_i\subseteq (0,1]$.
For $n\in\N$, $n\geq 2$, and $I\in\calI_n$,
set
\begin{equation*}
\begin{aligned}
\calL_W(I)&:=\{l\in [n-1]\colon W_l^h\cap I\neq \varnothing\}, \\
\calL_\calJ(I,\eta)&:=\{l\in [n-1]\colon \calJ_l^h\cap I\neq \varnothing \text{ and }\eta_l\delta_l\leq |I|\}, \\
\calL(I,\eta)&:=\calL_W(I)\cup\calL_\calJ(I,\eta),
\end{aligned}
\end{equation*}
and recall that $\calL_W(I),\calL_\calJ(I,\eta)$ are disjoint and contain at most one element each;
see \cref{lemma:LIP_light_L_J_card_1,lemma:LIP_light_L_J_W_disjoint}.
For $a\in [M]^{n-1}$,
let $A(I,\eta;a)$ denote the set of all $b\in [M]^{n-1}$ for which, setting 
\begin{equation*}
\Delta:=\{1\leq j\leq n-1\colon a(j)\neq b(j)\},
\end{equation*}
we have $\Delta\subseteq \calL(I,\eta)$ and,
if $\Delta\cap\calL_\calJ(I,\eta)\neq\varnothing$, then
\begin{equation*}
1=a(j)(=b(j))
\end{equation*}
for $\min(\Delta\cap\calL_\calJ(I,\eta))<j\leq n-1$ with $j\notin\calL_W(I)$.
The set $A(I,\eta;a)$ is well-defined even if $\calL(I,\eta)=\varnothing$,
and it always contains $a$.
Moreover, $\{(a,b)\in[M]^{n-1}\times[M]^{n-1}\colon a\in A(I,\eta;b)\}$
is an equivalence relation on $[M]^{n-1}$,
therefore
\begin{equation*}
\calA_\eta(I):=\{A(I,\eta;a)\colon a\in [M]^{n-1}\}
\end{equation*}
is a partition of $[M]^{n-1}$.
For $A\in\calA_\eta(I)$, we define
\begin{equation*}
	E_A:=\bigcup_{a\in A} Q_{I,a}.
\end{equation*}
\end{notation}
\begin{lemma}\label{lemma:LIP_light_sep}
Let $(X,d_\eta)$ be a \squashed\ Laakso space.
For $n\geq 2$, $I\in\calI_n$, and $A_1\neq A_2\in \calA_\eta(I)$,
we have
\begin{equation*}
d_\eta(E_{A_1},E_{A_2})\geq \tfrac{1}{3}|I|.
\end{equation*}
\end{lemma}
\begin{proof}
Let $a_i\in A_i$, and define $\Delta$ as in \cref{notation:LIP_light_equivalence_cubes}, but with $a_1,a_2$ in place of $a,b$.
Since $A_1\neq A_2$, we have $\Delta\neq\varnothing$ and $a_1\notin A(I,\eta;a_2)$.
Let $x_i\in Q_{I,a_i}$, $1\leq i\leq 2$, and observe that by definition of $Q_{I,a_i}$,
we have $x_i(j)=a_i(j)$ for $1\leq j\leq n-1$.
Suppose first $\Delta\nsubseteq \calL(I,\eta)$, so that there is $1\leq j\leq n-1$, $j\notin \calL(I,\eta)$, such that $x_1(j)\neq x_2(j)$.
Since $j\notin \calL(I,\eta)$, $W_j^h\cap I=\varnothing$ and either $\eta_j\delta_j>|I|$ or $\calJ_j^h\cap I=\varnothing$.
Then 
\begin{align*}
d(h\{x,y\},W_j^h)
&\geq d(I, W_{\leq n}^h\setminus I)\geq |I|, \\
2d(\{x,y\},\cup\calJ_j)+\eta_j\delta_j
&\geq
2d(I, \calJ_j^h)+\eta_j\delta_j\geq |I|,
\end{align*}
and, by \cref{lemma:Laakso_basic_separation}, we have
\begin{equation}\label{eq:LIP_light_sep_lemma_eq_1}
d_\eta(x_1,x_2)\geq \tfrac{1}{3}|I|.
\end{equation}
Suppose now $\Delta\subseteq \calL(I,\eta)$,
i.e.\ $x_1(j)=x_2(j)$ for $1\leq j\leq n-1$ with $j\notin\calL(I,\eta)$.
Then $\Delta\cap \calL_\calJ(I,\eta)\neq\varnothing$,
$\Delta\cap\calL_\calJ(I,\eta)=\{l\}$ for some $l\in[n-1]$ (by \cref{lemma:LIP_light_L_J_card_1}),
and there is $l<j\leq n-1$ with $j\notin \calL_W(I)$
and $x_1(j)=x_2(j)\neq 1$.
Since $l\in\calL_\calJ(I,\eta)$, by \cref{lemma:LIP_light_L_J_W_disjoint} we have $W_l^h\cap I=\varnothing$ and therefore $d(W_l^h,I)\geq |I|$.
Since $l\in \Delta$ and $l\leq n-1$, we have $x_1(l)=a_1(l)\neq a_2(l)=x_2(l)$ and \cref{lemma:Laakso_basic_separation} gives
\begin{align*}
d_\eta(x_1,x_2)
&\geq\tfrac{1}{3}\min\{2d(h\{x_1,x_2\},W_l^h),2d(\{x_1,x_2\},\cup\calJ_l)+\eta_l\delta_l\} \\
&\geq\tfrac{2}{3}\min\{|I|,d(\{x_1,x_2\},\cup\calJ_l)\}
\end{align*}
Let $x\in \{x_1,x_2\}$ and $y\in \cup\calJ_l$.
The inequality $l<j$ implies $x(j)\neq 1 = y(j)$, and \cref{prop:dist_in_Laakso} gives
\begin{equation*}
d(x,y)\geq d(h(x),W_j^h)\geq |I|,
\end{equation*}
because $I\cap W_j^h=\varnothing$.
Hence
\begin{equation}\label{eq:LIP_light_sep_lemma_eq_2}
d_\eta(x_1,x_2)\geq \tfrac{2}{3}|I|.
\end{equation}
Since $a_i\in A_i$ and $x_i\in Q_{I,a_i}$ were arbitrary, \cref{eq:LIP_light_sep_lemma_eq_1} and \cref{eq:LIP_light_sep_lemma_eq_2} conclude the proof.
\end{proof}
\begin{lemma}\label{lemma:LIP_light_diam}
Let $(X,d_\eta)$ be a \squashed\ Laakso space.
For $n\geq 2$, $I\in\calI_n$, and $A\in\calA_\eta(I)$, we have
\begin{equation*}
\diam_\eta(E_A)\leq 5|I|.
\end{equation*}
\end{lemma}
\begin{proof}
Let $a_1,a_2\in A$ and $x_i\in Q_{I,a_i}$ for $1\leq i\leq 2$ and define $\Delta$ as in \cref{notation:LIP_light_equivalence_cubes}, with $a_1,a_2$ in place of $a,b$.
Recall that $W_l^h\cap I\neq\varnothing$ for $l\geq n$.
Thus, if $\Delta\cap\calL_\calJ(I,\eta)=\varnothing$, then $\Delta\subseteq \calL_W(I)$ and so $I$ contains all wormhole heights necessary to travel from $x_1$ to $x_2$,
which, by \cref{prop:dist_in_Laakso}, gives
\begin{equation*}
d_\eta(x_1,x_2)\leq d(x_1,x_2)\leq 2|I|.
\end{equation*}
Suppose now $\Delta\cap \calL_\calJ(I,\eta)\neq\varnothing$; in particular $a_1\neq a_2$.
By \cref{lemma:LIP_light_L_J_card_1},
there is $1\leq l\leq n-1$ such that $\Delta\cap\calL_\calJ(I,\eta)=\{l\}$.
Let $t\in\calJ_l^h\cap I$ and let $p_1,p_2\in X$ be given by $h(p_1)=h(p_2)=t$,
\begin{equation*}
p_i(j)
=
\left\{
\begin{array}{cc}
	x_i(j), & 1\leq j\leq l \text{ and }j\notin\calL_W(I) \\
	1, & l+2\leq j<\infty \text{ or }j\in\calL_W(I)
\end{array}
\right.
,
\end{equation*}
where we have not prescribed the $(l+1)$-th digit because $h(p_i)\in\calJ_l^h\subseteq W_{l+1}^h$.
Since $\Delta\subseteq \calL(I,\eta)$ by assumption, it follows that $p_1,p_2$ agree at all digits except
possibly the $l$-th one.
Hence, there is $S\in\calJ_l$ for which $p_1,p_2\in S$, and so
$d_\eta(p_1,p_2)\leq  \eta_l\delta_l\leq |I|$ by definition of $\calL_\calJ(I,\eta)$.
By definition of $\calA_\eta(I)$, it follows that $x_i,p_i$ may differ only
at digits $j\geq n$ or $j\in\calL_W(I)$.
Thus, all wormhole heights necessary to travel from $x_i$ to $p_i$ are contained in $I$,
and \cref{prop:dist_in_Laakso} gives
\begin{equation*}
d_\eta(x_i,p_i)\leq d(x_i,p_i)\leq 2|I|.
\end{equation*}
The thesis then follows by triangle inequality.
\end{proof}
\begin{lemma}\label{lemma:LIP_light_nice_intervals}
Let $(X,d_\eta)$ be a \squashed\ Laakso space.
Then there is $C\geq 1$
such that for $n\in\N$, $I\in\calI_n$, and $r\geq |I|$,
the $r$-components of $h^{-1}(I)$ in $(X,d_\eta)$ have diameter at most $Cr$ w.r.t.\ $d_\eta$. \par
The constant $C$ depends only on $\max_n N_n$, where $(N_n)$ is such that $X=X(M;(N_n))$ for some $M$.
\end{lemma}
\begin{proof}
By \cref{lemma:LIP_light_sep} and \cref{lemma:LIP_light_diam}, there is a universal constant $C_1\geq 1$ such that
for $n\in\N$, $n\geq 2$, and $I\in\calI_n$, the sets $\{E_A\colon A\in\calA_\eta(I)\}$ are $C_1|I|$-bounded and strictly $C_1^{-1}|I|$-separated in $(X,d_\eta)$.
Since
\begin{equation*}
	h^{-1}(I)=\bigcup_{A\in\calA_\eta(I)}E_A,
\end{equation*}
the strict separation shows that any $C_1^{-1}|I|$-path must be contained in a single set $E_A$, for some $A\in\calA_\eta(I)$.
Hence, the $C_1^{-1}|I|$-components of $h^{-1}(I)$ in $(X,d_\eta)$ have diameter at most $C_1|I|$ w.r.t.\ $d_\eta$. \\
Let $0<r\leq C_1^{-1}\delta_2$, where $\delta_i=\frac{1}{N_1\cdots N_i}$ is as in \cref{defn:shortcuts_laakso}, and let $k\in\N$ be the greatest integer
for which $r\leq C_1^{-1}\delta_{k}$.
By maximality, we have
\begin{equation}\label{eq:LIP_light_nice_intervals_eq_1}
	C_1^{-1}\delta_{k+1}<r\leq C_1^{-1}\delta_k
\end{equation}
and $k\geq 2$.
Let $I\in\calI_n$, with $r\geq |I|$.
Then $n\geq 2$ and there is $J\in\calI_k$ with $I\subseteq J$.
Since $h^{-1}(I)\subseteq h^{-1}(J)$, each $r$-component of $h^{-1}(I)$ is contained in an $r$-component of $h^{-1}(J)$,
which in turn is contained in a $C_1^{-1}|J|$-component of $h^{-1}(J)$, because $r\leq C_1^{-1}|J|$.
This shows that every $r$-component of $h^{-1}(I)$ has diameter w.r.t.\ $d_\eta$ at most
\begin{equation*}
C_1|J|\leq C_1^2 N_{k+1}r\leq C_1^2 C_2r,
\end{equation*}
where we have used \cref{eq:LIP_light_nice_intervals_eq_1} and
set
$C_2:=\max_n N_n$, which (exists and) is finite by \cref{eq:N_n_theta_n_Laakso}.
If $r>C_1^{-1}\delta_2$ and $I\in\calI_n$,
the $r$-components of $h^{-1}(I)$ have diameter at most
$\diam_\eta X\leq 2\leq 2 C_1\delta_2^{-1}r\leq 2 C_1C_2^2r$.
\end{proof}
\begin{lemma}\label{lemma:LIP_light}
Let $(X,d_\eta)$ be a \squashed\ Laakso space.
Then there is $C\geq 1$ such that the map $h\colon (X,d_\eta)\to \R$ is $C$-Lipschitz light. \par
The constant $C$ depends only on $\max_n N_n$, where $(N_n)$ is such that $X=X(M;(N_n))$ for some $M$.
\end{lemma}
\begin{proof}
We already know that $\LIP_\eta(h)=\LIP(h)=1$.
Let $C_1\geq 1$ be the constant of \cref{lemma:LIP_light_nice_intervals}.
It is not difficult to see that there is an integer $N\geq 2$, depending only on $\max_nN_n$,
with the following property.
For every interval $J\subseteq [0,1]$
there are $n\in\N$ and $J_1,\dots,J_N\in\calI_n$ satisfying
\begin{align*}
	J&\subseteq \bigcup_{i=1}^NJ_i, & |J_i|\leq |J|
\end{align*}
for $1\leq i\leq N$.
(We can take $N=\max_nN_n+2$.)
It is now easy to conclude.
Let $J\subseteq [0,1]$ be an interval, $r\geq |J|$, and $J_1,\dots, J_N$ as above.
Since $J_i\in\calI_n$ and $r\geq |J|\geq |J_i|$, \cref{lemma:LIP_light_nice_intervals} shows
that the $r$-components of $h^{-1}(J_i)$ have diameter at most $C_1r$.
By \cref{lemma:union_r_components}, there is a constant $C\geq 1$, depending only on $C_1$ and $N$,
such that the $r$-components of $h^{-1}(J)$ have diameter at most $Cr$.
Hence, from \cref{rmk:LIP_light_closed_convex}, $h$ is $C$-Lipschitz light.
Since $C_1$ and $N$ depend only on $\max_nN_n$, so does $C$.
\end{proof}
The notion of Nagata dimension was introduced in \cite{Assouad_sur_la_distance_de_Nagata}, see also \cite{Lang_Schlichenmaier_Nagata}.
For our purposes, it is enough to know that
\begin{equation}\label{eq:top_vs_Nagata_vs_LIP_dim}
\dim_T Z\leq \dim_N Z\leq \dim_L Z,
\end{equation}
where $Z$ is a compact metric space and
$\dim_T Z$, $\dim_N Z$ denote respectively its topological and Nagata dimension.
See \cite[Theorem 2.2]{Lang_Schlichenmaier_Nagata} for the first inequality and \cite[Corollary 3.5]{Guy_C_David_LIP_dim_CK} for the second.
\begin{prop}\label{prop:Lipschitz_Nagata_dim}
Let $(X,d_\eta)$ be a \squashed\ Laakso space.
Then
\begin{equation*}
\dim_T (X,d_\eta)= \dim_{N} (X,d_\eta) = \dim_{L} (X,d_\eta)=1.
\end{equation*}
\end{prop}
\begin{proof}
By \cref{lemma:LIP_light}, we have $\dim_L (X,d_\eta)\leq 1$.
Since $(X,d_\eta)$ contains a homeomorphic copy of $[0,1]$, we also have $\dim_T (X,d_\eta)\geq 1$.
Finally, \cref{eq:top_vs_Nagata_vs_LIP_dim} concludes the proof.
\end{proof}
As a corollary, we have the following.
\begin{prop}\label{prop:Laakso_emb_L1}
Let $(X,d_\eta)$ be a \squashed\ Laakso space. There are $C\geq 1$
and a map $f\colon X\to L^1([0,1])$ satisfying
\begin{equation*}
d_\eta(x,y)\leq \|f(x)-f(y)\|_{L^1([0,1])}\leq C d_\eta(x,y),
\end{equation*}
for $x,y\in X$. \par
The constant $C$ depends only on $\max_n N_n$, where $(N_n)$ is such that $X=X(M;(N_n))$ for some $M$.
\end{prop}
\begin{proof}
Let $C_1\geq 1$ be the constant of \cref{lemma:LIP_light}.
By \cite[Theorem 1.11]{Cheeger_Kleiner_LIP_dimension} (with e.g.\ constant $m=2$) and \cref{lemma:LIP_light}, $(X,d_\eta)$ admits a biLipschitz
embedding into an admissible inverse system \cite[Definition 1.8]{Cheeger_Kleiner_LIP_dimension},
with biLipschitz constant depending only on $C_1$.
But then, by \cite[Theorem 1.16]{Cheeger_Kleiner_LIP_dimension}, there are a measure space $(\Omega,\calF,\nu)$
and a biLipschitz embedding $\tilde{f}\colon (X,d_\eta)\to L^1(\nu)$,
with biLipschitz constant depending only on $C_1$.
Since the closed span $Y$ of $\tilde{f}(X)$ is a separable subspace of $L^1(\nu)$,
by \cite[Fact 1.20]{Ostrovskii_metric_embeddings},
there is a linear isometry $\iota\colon Y\to L^1([0,1])$.
Finally, setting $f:=\iota\circ\tilde{f}$ concludes the proof.
\end{proof}
We now turn to non-embeddability in Banach spaces with RNP.
Note that, if $\eta$ is chosen such that $(X,d_\eta,\Haus^s)$ is purely PI unrectifiable, we cannot appeal to the differentiability theory
of Cheeger and Kleiner \cite{cheeger_kleiner_PI_RNP_LDS}, at least not directly.
Nonetheless, the controlled behaviour of the contraction $\eta$ (\cref{lemma:doubling_implies_DS_reg}) paired with
a weak Sard theorem for Lipschitz functions on LDS (\cite[Theorem 1.4]{Guy_C_David_tangents_ADR_LDS})
is sufficient to rule out DS-regular embeddability.
Indeed, this is implied by the following more general result.
\begin{prop}
Let $Y$ be a Banach space,
$(X,d,\mu)$ an $s$-ADR $Y$-LDS with \shortcuts, and $\eta=(\eta_i)_i\subseteq (0,1]$.
Then
\begin{equation*}
\Haus^s_Y(f(E))=0
\end{equation*}
for every $\mu$-measurable set $E\subseteq X$ and Lipschitz $f\colon (E,d_\eta)\to Y$. \par
In particular,
no positive-measure $\mu$-measurable subset of $(X,d_\eta)$ admits a David-Semmes regular embedding in $Y$.
\end{prop}
\begin{proof}
First, let us verify that if $(U,\varphi\colon X\to\R^n)$ is a chart in $(X,d,\mu)$ with $\mu(U)>0$,
then $n<s$.
To see this, suppose not.
Then $n\in\N$, $n=s$ (by \cite[Theorem 5.99]{schioppa2016_derivations_and_alberti_representations} or \cite[Theorem 5.3]{batekangasniemiorponen2019})
and $(U,d)$ is $s$-rectifiable by \cite{Bate_Li_char_rectifiable_mmsp}, but by \cref{prop:biLipschitz_pieces}
it has a Lipschitz image with positive measure and no biLipschitz piece,
which is not possible by Kirchheim's differentiation theorem \cite{Kirchheim_rectifiable_metric_spaces} and $\mu(U)>0$. \par
We now turn to the proof of the statement.
It is enough to consider the case where $\mu(E)>0$ and there is $n\in\N_0$ and $\varphi\colon (X,d)\to\R^n$
so that $(E,\varphi)$ is a chart of $(X,d,\mu)$.
If $f\colon (E,d_\eta)\to Y$ is Lipschitz, then so is $f\colon (E,d)\to Y$,
and \cite[Theorem 1.4]{Guy_C_David_tangents_ADR_LDS} gives $\Haus^s(f(E))=0$,
as claimed.
Note that the proof of \cite[Theorem 1.4]{Guy_C_David_tangents_ADR_LDS}
generalises verbatim to our setting. \par
Finally, suppose $E\subseteq X$ is $\mu$-measurable and $f\colon (E,d_\eta)\to Y$ is David-Semmes regular.
Then, by \cref{lemma:doubling_implies_DS_reg}, $f\colon (E,d)\to Y$ is also DS-regular
and so (by \cref{eq:DS_reg_maps_and_Haus}) we have
\begin{equation*}
\mu(E)\sim \Haus^s_Y(f(E))=0,
\end{equation*}
concluding the proof.
\end{proof}
\appendix
\section{Metric spaces with shortcuts uniformly and Carnot groups}
\subsection{Comparison with \cite{ledonne_li_rajala_shortcuts_heisenberg}}\label{sec:comparison_LDLR}
Our definition of metric spaces with \shortcuts\ (\cref{defn:shortcuts})
is inspired by \cite{ledonne_li_rajala_shortcuts_heisenberg},
where the following condition is introduced (without giving it a name).
\begin{defn}\label{defn:shortcuts_LDLR}
We say that a metric space $X$ has \emph{\shortcutsLDLR[$\delta$]}, $\delta\in (0,1)$, if
for $x_0\in X$ and $0<r<\diam X$ there is a $\delta r$-separated set $S\subseteq U(x_0,r)$ with $\#S\geq 2$
satisfying
\begin{equation*}
d(x,y)\leq d(x,z)+d(w,y)
\end{equation*}
for $x,y\notin U(x_0,r)$ and $z,w\in S$.
\end{defn}
\Cref{defn:shortcuts_LDLR} differs slightly from
\cite[Equation (1.1)]{ledonne_li_rajala_shortcuts_heisenberg}
in that we allow for sets $S$ with $\#S>2$.
As pointed out in \cref{rmk:subshortcuts}, this is only a minor change.
Metric spaces as in \cref{defn:shortcuts_LDLR}
appeared also in 
\cite{joseph_rajala_snowflaked_lines_shortcuts},
where they are called spaces with $\delta$-invisible pieces.

Under mild assumptions, \cref{defn:shortcuts_LDLR} implies the existence of \shortcuts\ as in \cref{defn:shortcuts}
(see \cref{lemma:LDLR_implies_shortcuts}), but
\cref{defn:shortcuts} allows for more flexibility,
since the scales $(\delta_i)$ are allowed to decay arbitrarily quickly. \par
Recall that a metric space $X$ is \emph{$C$-uniformly perfect}, $C>1$,
if $U(x,r)\setminus U(x,r/C)\neq\varnothing$ for $x\in X$ and $0<r<\diam X$,
and uniformly perfect if it there is $C>1$ as above.
Examples of uniformly perfect metric spaces include connected and ADR metric spaces.
\begin{lemma}\label{lemma:LDLR_implies_shortcuts}
Let $X$ be a $C$-uniformly perfect metric space with \shortcutsLDLR[$\delta_0$].
Then $X$ has \shortcuts\ (\cref{defn:shortcuts}). \par
More precisely, for
$0<\delta\leq\delta_0/16C$, $0<A<\diam X/4\delta$,
and $M\geq 4C+5$, 
$X$ has \shortcuts\  with
parameters $\delta_i=A\delta^i$,
$a=b=2$, $a_0=\delta_0$, and $M$ as above.
\end{lemma}
\cref{lemma:LDLR_implies_shortcuts} follows arguing as in \cite[Section 3.1]{ledonne_li_rajala_shortcuts_heisenberg}.
We provide a complete proof since our setting differs slightly and we need to keep track of the several constants ($a_0,a,b$, and $M$).
\begin{lemma}\label{lemma:nets_in_uniformly_perfect}
Let $X$ be a $C$-uniformly perfect metric space, $0<R<\diam X$, and $\calN_0\subseteq X$ an $R$-separated set.
Then, for $0<r\leq R/4C$,
any $r$-net $\calN$ of $X\setminus U(\calN_0,r)$
is a $(C+1)r$-net of $X$, i.e.\ $B(\calN,(C+1)r)=X$.
\end{lemma}
\begin{proof}
Since $X$ is $C$-uniformly perfect, we have
$U(x,\rho)\setminus U(y,t)\neq\varnothing$ for $x,y\in X$, $0<\rho<\diam X$, and $0<t<\rho/C$.
Let $Cr<\rho\leq R/2-r$ and pick $x\in X$.
Since $\calN_0$ is $R$-separated, the choice of $\rho$ implies that there is at most one $y\in\calN_0$
such that $U(x,\rho)\cap U(y,r)\neq\varnothing$.
Hence, there is $y_0\in\calN_0$ such that $U(x,\rho)\setminus U(\calN_0,r)= U(x,\rho)\setminus U(y_0,r)$,
which is then non-empty because $r<\rho/C$.
Since $X\setminus U(\calN_0,r)\subseteq B(\calN,r)$, we have
$d(x,\calN)<\rho+r$ for every $x\in X$ and $\rho$ as above,
proving $B(\calN,(C+1)r)=X$.
\end{proof}
\begin{rmk}
The conclusion of \cref{lemma:nets_in_uniformly_perfect} characterises uniformly perfect metric spaces.
Indeed, taking $x\in X$, $\calN_0=\{x\}$, $0<R<\diam X$, and $r=R/4C$,
we find $y\in X$ with $r\leq d(x,y)<4Cr$, i.e.\
$y\in U(x,R)\setminus U(x,R/4C)$.
\end{rmk}
\begin{proof}[Proof of \cref{lemma:LDLR_implies_shortcuts}]
We first construct $\calJ_1$.
Let $\calN_1\subseteq X$ be a maximal $4A\delta$-separated set
and,
for $x\in\calN_1$, let $S_x\subseteq U(x,A\delta)$ be as in \cref{defn:shortcuts_LDLR}
with $2\leq \# S_x\leq M$.
Set $\calJ_1:=\{S_x\colon x\in\calN_1\}$.
It is then not difficult to see that $\calJ_1$ satisfies the conditions of \cref{defn:shortcuts}
with the current values of $a_0,a,b,M$, and $\delta_1=A\delta$. \par
Let $n\in\N$, $n\geq 2$, and suppose we have found $\calJ_1,\dots,\calJ_{n-1}$ satisfying
the conditions of \cref{defn:shortcuts} with $a_0,a,b,M$ as above and $\delta_i=A\delta^i$ for $1\leq i\leq n-1$.
Set $\widetilde{\calN}:=\bigcup_{i=1}^{n-1}\bigcup\calJ_i$,
let $\calN_n\subseteq X\setminus U(\widetilde{\calN},4A\delta^n)$ be a maximal $4A\delta^n$-separated set,
and, for $x\in\calN_n$, let $S_x\subseteq U(x,A\delta^n)$ be as in \cref{defn:shortcuts_LDLR} with $2\leq\# S_x\leq M$.
Set $\calJ_n:=\{S_x\colon x\in\calN_n\}$.
Observe that $\widetilde{\calN}$ is $A\delta_0\delta^{n-1}$-separated, $A\delta_0\delta^{n-1}<\diam X$,
$4A\delta^n\leq(A\delta_0\delta^{n-1})/4C$,
and hence, by \cref{lemma:nets_in_uniformly_perfect},
$\calN_n\neq\varnothing$ and
$B(\calN_n,(C+1)4A\delta^n)=X$.
Since $\calN_n\subseteq U(\cup\calJ_n,A\delta^n)$,
$\calJ_n$ satisfies also the last condition of
\cref{defn:shortcuts} (with $M$ as above).
Lastly, it is not difficult to verify that $\{(\calJ_i,A\delta^i)\}_{i=1}^n$ satisfies the first four conditions
of \cref{defn:shortcuts} with $\delta_n=A\delta^n$.
\end{proof}
\begin{rmk}
Let $X$ be a metric space as \cref{lemma:LDLR_implies_shortcuts} and suppose there is $N\geq 2$
such that the sets $S$ as in \cref{defn:shortcuts_LDLR} may be taken with $\#S=N$.
It follows from the proof of \cref{lemma:LDLR_implies_shortcuts} that we can then construct $\{(\calJ_i,A\delta^i)\}_{i\geq 1}$
as in the thesis (with $M\geq \max(N,4C+5)$), which additionally satisfies $\#S=N$ for $S\in\calJ=\bigcup_i\calJ_i$.
\end{rmk}
\subsection{Carnot groups have \shortcuts}\label{subsec:Carnot_groups}
Carnot groups are a distinguished class of Lie groups, whose underlying smooth manifold may be identified with a finite dimensional Euclidean space.
We refer the reader to \cite{ledonne_primer_Carnot_groups, serra_cassano_gmt_in_Carnot, BLU_Carnot_groups, ledonne_metric_lie_groups_arxiv}
for more information on the topic.
From our point of view, Carnot groups are a class of ADR $1$-PI spaces which can be equipped
with a distance with \shortcuts\ compatible with
the Cheeger differentiable structure.
We will recall the few basic facts we need. \par
We identify Carnot groups with their Lie algebra via exponential coordinates (see \cite[Proposition 2.2.22]{BLU_Carnot_groups})
and write $\G=\R^{m_1}\times\dots\times\R^{m_k}$,
where $\R^{m_i}$ corresponds to the $i$-th layer of the stratification.
We denote points of $\G=\R^{m_1}\times\cdots\times\R^{m_k}$ as $x=(x^{(1)},\dots,x^{(k)})$
and define $\Delta_{\lambda}\colon\G\to\G$,
$\Delta_\lambda x:=(\lambda x^{(1)},\dots,\lambda^k x^{(k)})$,
for $\lambda>0$ and $x\in\G$.
The dilations $\Delta_\lambda$ are group automorphisms; in particular,
$\Delta_\lambda(x\cdot y)=\Delta_\lambda x\cdot\Delta_\lambda y$
and $\Delta_\lambda(x^{-1})=(\Delta_\lambda x)^{-1}$,
where $\cdot$ denotes the group law.
Also, the identity element $e_\G$ of $\G$ satisfies $\Delta_\lambda e_\G=e_\G$ for $\lambda>0$,
proving $e_\G=0=(0,\dots,0)$.
A distance $d\colon\G\times\G\to[0,\infty)$ is \emph{invariant} if it satisfies
\begin{equation*}
d(z\cdot x,z\cdot y)=d(x,y),\qquad d(\Delta_\lambda x,\Delta_\lambda y)=\lambda d(x,y),
\end{equation*}
for $x,y,z\in\G$ and $\lambda>0$.
Note that we do not require $d$ to induce the (Euclidean) topology of $\G$,
since this follows from the definition, see \cite[Theorem 1.1]{LeDonne_NicolussiGolo_metric_lie_groups_dilations}.
Invariant distances are closely related to
\emph{homogeneous norms}, namely functions $N\colon\G\to [0,\infty)$ such that
\begin{equation}\label{eq:homogeneous_norm}
\begin{aligned}
N(x)&=0 \text{ implies } x=0, \\
N(x^{-1})&=N(x), \\
N(\Delta_\lambda x)&=\lambda N(x), \\
N(x\cdot y)&\leq N(x)+N(y),
\end{aligned}
\end{equation}
for $x,y\in\G$ and $\lambda>0$.
Indeed, if $d$ is an invariant distance, then $d(\cdot,0)$ is a homogeneous norm
and, conversely, if $N$ is a homogeneous norm, then $d(x,y):=N(y^{-1}\cdot x)$ defines an invariant distance.
In particular, homogeneous norms are continuous and one may verify that they are all equivalent, i.e.\ for any two homogeneous norms $N_1,N_2$
there is a constant $C\geq 1$ such that $C^{-1}N_1\leq N_2\leq CN_1$;
see \cite[Proposition 5.1.4]{BLU_Carnot_groups}\footnote{Note that,
following \cite[Definition 2.6]{serra_cassano_gmt_in_Carnot},
our definition of homogeneous norm is more restrictive than \cite[Definition 5.1.1]{BLU_Carnot_groups}.
Also, we do not explicitly require continuity, since
it follows from the other conditions; see \cite[Theorem 1.1]{LeDonne_NicolussiGolo_metric_lie_groups_dilations}.}.
It follows that all invariant distances are biLipschitz equivalent.
\begin{prop}\label{prop:Carnot_have_shortcuts_prelim}
Let $\G =\R^{m_1}\times \cdots \times \R^{m_k}$ be a Carnot group of step $k\geq 2$ and $\delta\in(0,1/2)$.
Then there is an invariant distance $d$ on $\G$ such that
\begin{equation}\label{eq:Carnot_have_shortcuts_prelim}
	d(x,y)\leq d(x,z)+d(w,y),
\end{equation}
for $x,y\notin U_d(0,1)$ and $z,w\in B_d(0,\delta)$ with $z^{(1)}=w^{(1)}$.
\end{prop}
By left-translation and dilation, \cref{eq:Carnot_have_shortcuts_prelim} implies
\begin{equation*}
	d(x,y)\leq d(x,z)+d(w,y),
\end{equation*}
whenever there are $x_0\in \G$, $r>0$ such that
$x,y\notin U_d(x_0,r)$ and $z,w\in B_d(x_0,\delta r)$ with $z^{(1)}=w^{(1)}$. \par
The proof of \cref{prop:Carnot_have_shortcuts_prelim} relies on elementary estimates
which already appeared in \cite[Theorem 5.1]{Franchi_Serapioni_Serra_Cassano_str_finite_per_sets_step_2_Carnot},
which we follow closely.
They are based on a few properties of the group laws of Carnot groups, which we now recall; see \cite[Remark 1.4.4]{BLU_Carnot_groups}.
For any Carnot group $\G=\R^{m_1}\times\dots\times\R^{m_k}$,
there are polynomial functions $\calQ^{(i)}\colon\G\times\G\to\R^{m_i}$
such that the group law may be written as 
\begin{equation}\label{eq:law_Carnot_groups}
(x\cdot y)^{(i)}=x^{(i)}+y^{(i)}+\calQ^{(i)}(x,y)
\end{equation}
and, moreover, $\calQ^{(i)}(x,y)$ depends only on $x^{(1)},\dots,x^{(i-1)},y^{(1)},\dots,y^{(i-1)}$; in particular, $\calQ^{(1)}=0$.
From this, one may verify by induction that $(x^{-1})^{(i)}=-x^{(i)}$ for $1\leq i\leq k$,
i.e.\ $x^{-1}=-x=(-x^{(1)},\dots,-x^{(k)})$.
\begin{lemma}\label{lemma:Carnot_have_shortcuts_prelim}
Let $\G=\R^{m_1}\times\cdots\times\R^{m_k}$ be a Carnot group of step $k\geq 2$ and, for $1\leq i\leq k$,
let $\|\cdot\|_i$ be a norm on $\R^{m_i}$.
Let $\calQ^{(2)},\dots,\calQ^{(k)}$ be as in \cref{eq:law_Carnot_groups} and
set $\|x\|_{\G, i}:=\max_{1\leq j\leq i}\|x^{(j)}\|_j^{\frac1j}$ for $x\in\G$ and $1\leq i\leq k$.
Then, there is a constant $C>0$, such that for $2\leq i\leq k$ and $x,y\in\G$, we have
\begin{equation*}
	\|\calQ^{(i)}(x,y)\|_i\leq C\max_{1\leq j\leq i-1}\|x\|_{\G, i-1}^j\|y\|_{\G,i-1}^{i-j}.
\end{equation*}
\end{lemma}
\begin{proof}
This lemma is essentially extracted from the proof of \cite[Theorem 5.1]{Franchi_Serapioni_Serra_Cassano_str_finite_per_sets_step_2_Carnot}.
Set $h_0:=0$, $h_i:=\sum_{j=1}^im_j$ for $1\leq i\leq k$, and define 
$|\sigma|_{\G}:=\sum_{i=1}^k\sum_{l=h_{i-1}+1}^{h_i}i\,\sigma_l$,
for $\sigma=(\sigma_1,\dots,\sigma_{h_k})\in\N_0^{h_k}$.
Fix $2\leq i\leq k$, $h_{i-1}+1\leq l\leq h_i$, and let $\calQ_l$ denote the $(l-h_{i-1})$-th entry of
$\calQ^{(i)}$.
Since $\Delta_\lambda\colon\G\to\G$ is a group homomorphism, \cref{eq:law_Carnot_groups}
implies $\calQ_l(\Delta_\lambda x,\Delta_\lambda y)=\lambda^i\calQ_l(x,y)$,
and therefore (see \cite[Proposition 1.3.4]{BLU_Carnot_groups})
\begin{equation*}
	\calQ_l(x,y)=\sum_{\substack{\sigma,\theta\in\N_0^{h_k},\\|\sigma|_\G+|\theta|_\G=i}}q_l(\sigma,\theta)x^\sigma y^\theta,
\end{equation*}
for some $q_l(\sigma,\theta)\in\R$.
By definition of $|\cdot|_\G$ and $\calQ_l(x,0)=\calQ_l(0,x)=0$ (which follows from $0$ being the identity element of $\G$),
we deduce $q_l(\sigma,\theta)=0$ if $\sigma=0$ or $\theta=0$ or $\sigma_t+\theta_t>0$ for some $t>h_{i-1}$.
Thus, $q_l(\sigma,\theta)\neq 0$ implies
\begin{align*}
|x^\sigma|
&=|x_1^{\sigma_1}\cdots x_{h_{i-1}}^{\sigma_{h_{i-1}}}|\leq\prod_{j=1}^{i-1}\prod_{t=h_{j-1}+1}^{h_j}|x_t|^{\frac{j\sigma_t}{j}}
\lesssim\prod_{j=1}^{i-1}\prod_{t=h_{j-1}+1}^{h_j}\|x^{(j)}\|_j^{\frac{j\sigma_t}{j}} \\
&\leq \|x\|_{\G,i-1}^{\sum_{j=1}^{i-1}\sum_{t=h_{j-1}+1}^{h_j}j\,\sigma_t}=\|x\|_{\G,i-1}^{|\sigma|_\G}
\end{align*}
and similarly $|y^{\theta}|\lesssim \|y\|_{\G,i-1}^{|\theta|_{\G}}$.
The thesis then follows from
\begin{align*}
|\calQ_l(x,y)|
\lesssim\sum_{\substack{\sigma,\theta\in\N_0^{h_k}\setminus\{0\},\\|\sigma|_\G+|\theta|_\G=i}}
\|x\|_{\G,i-1}^{|\sigma|_\G}\|y\|_{\G,i-1}^{|\theta|_\G}
\lesssim
\sum_{j=1}^{i-1}
\|x\|_{\G,i-1}^{j}\|y\|_{\G,i-1}^{i-j}
\lesssim \max_{1\leq j\leq i-1}
\|x\|_{\G,i-1}^{j}\|y\|_{\G,i-1}^{i-j}.
\end{align*}
\end{proof}
\begin{proof}[Proof of \cref{prop:Carnot_have_shortcuts_prelim}]
The thesis is equivalent to the existence of a homogeneous norm $N$ on $\G$
satisfying
\begin{equation*}
N(x\cdot y)\leq N(x\cdot z)+N(w\cdot y)
\end{equation*}
for $x,y,z,w\in\G$ with $z^{(1)}=-w^{(1)}$ and $N(z),N(w)\leq\delta < 1\leq N(x),N(y)$.
Fix norms $\|\cdot\|_1,\dots,\|\cdot\|_k$ on $\R^{m_1},\dots,\R^{m_k}$ respectively.
For $\omega=(\omega_1,\dots,\omega_k)\in(0,\infty)^k$ and $1\leq i\leq k$,
set
\begin{equation}\label{eq:Carnot_have_shortcuts_prelim_1}
	N_{\omega,i}(x):=\max_{1\leq j\leq i}\omega_j\|x^{(j)}\|_j^{\frac1j},
	\qquad N_\omega(x):=N_{\omega,k}(x),
\end{equation}
for $x=(x^{(1)},\dots,x^{(k)})\in\G$.
It is clear that $N_{\omega}$ satisfies all conditions in \cref{eq:homogeneous_norm}, except possibly for
triangle inequality.
The proof of \cite[Theorem 5.1]{Franchi_Serapioni_Serra_Cassano_str_finite_per_sets_step_2_Carnot} shows
that for any $1<i\leq k$ if $N_{\omega,i-1}$ satisfies triangle inequality,
then so does $N_{\omega,i}$, provided $\omega_i$ is chosen small enough w.r.t.\ $\omega_1,\dots,\omega_{i-1}>0$.
We claim that for $1\leq i\leq k$ there are $\omega_1,\dots,\omega_i>0$ such that $N_{\omega,i}$
satisfies triangle inequality and, if so does $N_\omega$, then
\begin{equation}\label{eq:Carnot_have_shortcuts_prelim_2}
N_{\omega,i}(x\cdot y)\leq N_\omega(x\cdot z)+N_\omega(w\cdot y),
\end{equation}
for $x,y,z,w$ as in \cref{eq:Carnot_have_shortcuts_prelim_1} with $N=N_\omega$.
The thesis will then follow taking $i=k$ and setting $N:=N_\omega$. \par
We proceed by induction on $1\leq i\leq k$.
If $i=1$, then triangle inequality and \cref{eq:Carnot_have_shortcuts_prelim_2} hold for any $\omega_1>0$.
Suppose now $i>1$ and let $\omega_1,\dots,\omega_{i-1}>0$ be such that the claim holds with $i\equiv i-1$.
Let $\omega_i>0$ be small (to be determined depending on $\omega_1,\dots,\omega_{i-1}$),
suppose $N_\omega$ satisfies triangle inequality, and let
$x,y,z,w\in\G$ be as in \cref{eq:Carnot_have_shortcuts_prelim_1} with $N=N_\omega$.
Since $\|\cdot\|_i$ is a norm on $\R^{m_i}$, we have
\begin{equation}\label{eq:Carnot_have_shortcuts_prelim_3}
\begin{aligned}
\|x^{(i)}+y^{(i)}+\calQ^{(i)}(x,y)\|_i
&\leq
\|x^{(i)}+z^{(i)}+\calQ^{(i)}(x,z)\|_i+
\|w^{(i)}+y^{(i)}+\calQ^{(i)}(w,y)\|_i \\
&\qquad
+\|\calQ^{(i)}(x,y)\|_i+\|\calQ^{(i)}(x,z)\|_i+\|\calQ^{(i)}(w,y)\|_i \\
&\qquad+\|z^{(i)}\|_i+\|w^{(i)}\|_i.
\end{aligned}
\end{equation}
The assumptions on $x,y,z,w$ and $N_\omega$ imply
\begin{equation*}
1-\delta\leq (1-\delta)N_\omega(x)\leq N_\omega(x\cdot z)\leq (1+\delta)N_\omega(x)
\end{equation*}
and similarly for $N_\omega(w\cdot y)$;
in particular,
\begin{equation}\label{eq:Carnot_have_shortcuts_prelim_5}
\max(N_\omega(z),N_\omega(w))\leq\delta\leq\frac{\delta}{1-\delta}\min(N_\omega(x\cdot z),N_\omega(w\cdot y)).
\end{equation}
We then deduce from \cref{lemma:Carnot_have_shortcuts_prelim}
\begin{equation}\label{eq:Carnot_have_shortcuts_prelim_6}
\begin{aligned}
\|\calQ^{(i)}(x,y)\|_i
&\lesssim \max_{1\leq j\leq i-1}\|x\|_{\G,i-1}^j\|y\|_{\G,i-1}^{i-j} \\
&\leq\min_{1\leq j\leq i-1}\omega_j^{-i}
\max_{1\leq j\leq i-1}N_{\omega,i-1}(x)^jN_{\omega,i-1}(y)^{i-j} \\
&\lesssim 
\min_{1\leq j\leq i-1}\omega_j^{-i}
\max_{1\leq j\leq i-1}N_{\omega}(x\cdot z)^jN_{\omega}(w\cdot y)^{i-j}, \\
\|\calQ^{(i)}(x,z)\|_i
&\lesssim
\min_{1\leq j\leq i-1}\omega_j^{-i}
\max_{1\leq j\leq i-1}N_\omega(x\cdot z)^jN_\omega(w\cdot y)^{i-j}, \\
\|\calQ^{(i)}(w,y)\|_i
&\lesssim
\min_{1\leq j\leq i-1}\omega_j^{-i}
\max_{1\leq j\leq i-1}N_\omega(x\cdot z)^jN_\omega(w\cdot y)^{i-j}.
\end{aligned}
\end{equation}
\cref{eq:Carnot_have_shortcuts_prelim_5} also implies
\begin{equation*}
\begin{aligned}
\|z^{(i)}\|_i+\|w^{(i)}\|_i
&\leq \omega_i^{-i}(N_\omega(z)^i+N_\omega(w)^i)
\leq
\frac{2\delta^i}{\omega_i^i(1-\delta)^i}\min(N_\omega(x\cdot z),N_\omega(w\cdot y))^i \\
&\leq
\frac{2\delta^i}{\omega_i^i(1-\delta)^i}
\max_{1\leq j\leq i-1}N_\omega(x\cdot z)^jN_\omega(w\cdot y)^{i-j}
\end{aligned}
\end{equation*}
which, together with
\cref{eq:Carnot_have_shortcuts_prelim_3,eq:Carnot_have_shortcuts_prelim_6,eq:law_Carnot_groups}, gives
\begin{equation}\label{eq:Carnot_have_shortcuts_prelim_7}
\begin{aligned}
\omega_i^i\|(x\cdot y)^{(i)}\|_i
&\leq \omega_i^i\|(x\cdot z)^{(i)}\|_i+\omega_i^i\|(w\cdot y)^{(i)}\|_i \\
&\qquad
+
\left[\frac{C\omega_i^i}{\min_{1\leq j\leq i-1}\omega_j^i}+\frac{2\delta^i}{(1-\delta)^i}\right]
\max_{1\leq j\leq i-1}N_\omega(x\cdot z)^jN_\omega(w\cdot y)^{i-j},
\end{aligned}
\end{equation}
for some constant $C>0$ independent of $\omega$.
Since $\delta\in(0,1/2)$, it holds $\delta(1-\delta)^{-1}<1$
and, if $\omega_i>0$ was chosen small enough w.r.t.\ $\omega_1,\dots,\omega_{i-1}$, we have
\begin{equation*}
\frac{C\omega_i^i}{\min_{1\leq j\leq i-1}\omega_j^i}+\frac{2\delta^i}{(1-\delta)^i}\leq 2;
\end{equation*}
by (the proof of) \cite[Theorem 5.1]{Franchi_Serapioni_Serra_Cassano_str_finite_per_sets_step_2_Carnot}%
\footnote{
In fact, the proof of \cite[Theorem 5.1]{Franchi_Serapioni_Serra_Cassano_str_finite_per_sets_step_2_Carnot}
consists of the
the same argument exhibited here, but with $z=w=0$ and triangle inequality of $N_{\omega,i}$ as induction hypothesis.
},
we may
take $\omega_i>0$ so small that, in addition, $N_{\omega,i}$ satisfies triangle inequality.
Since $\min_{1\leq j\leq i-1}{i\choose j}={i\choose \floor{i/2}}\geq 2^{\floor{i/2}}\geq 2$,
\cref{eq:Carnot_have_shortcuts_prelim_7} yields
\begin{equation*}
\omega_i^i\|(x\cdot y)^{(i)}\|_i\leq\sum_{j=0}^i{i\choose j} N_\omega(x\cdot z)^jN_\omega(w\cdot y)^{i-j} = \big(N_\omega(x\cdot z)+N_\omega(w\cdot y)\big)^i.
\end{equation*}
Hence, the induction hypothesis and the definition of $N_{\omega,i}$ imply that \cref{eq:Carnot_have_shortcuts_prelim_2} holds,
concluding the proof of the claim.
\end{proof}
\begin{rmk}
We do not know what is the optimal value of $\delta$ in \cref{prop:Carnot_have_shortcuts_prelim};
nonetheless, one cannot do better than $\frac{1}{\sqrt{2}}$.
To see this, let $\G=\R^{m_1}\times\cdots\times\R^{m_k}$ be a Carnot group of step $k\geq 2$,
$d$ an invariant distance on $\G$, $\delta>0$, and suppose they satisfy the conclusion of \cref{prop:Carnot_have_shortcuts_prelim}.
Let $m_1+1\leq l\leq m_1+m_2$,
$u=(0,\dots,0,u_l,0,\dots,0)\in\G$ with $d(u,0)=1$, and set
$x:=y^{-1}:=u$, $z:=w^{-1}:=\Delta_{\delta}u$,
so that $d(x,0)=d(y,0)=1$ and $d(z,0)=d(w,0)=\delta$.
Observe that
$\calQ^{(i)}(u,u)=\calQ^{(i)}(-\Delta_\delta u,u)=\calQ^{(i)}(u,-\Delta_\delta u)=0$ for $1\leq i\leq k$;
see \cite[Proposition 2.2.22(4)]{BLU_Carnot_groups}.
It follows that $y^{-1}\cdot x=\Delta_{\sqrt{2}}u$, $z^{-1}\cdot x=y^{-1}\cdot w = \Delta_{\sqrt{1-\delta^2}}u$
and therefore
\begin{equation*}
	\sqrt{2}=d(x,y)\leq d(x,z)+d(w,y)=2\sqrt{1-\delta^2},
\end{equation*}
which implies $\delta\leq\frac{1}{\sqrt{2}}$.
\end{rmk}
Consider a Carnot group $\G=\R^{m_1}\times\cdots\times\R^{m_k}$ of step $k$,
let $s=\sum_{i=1}^ki\, m_i$
denote its \emph{homogeneous dimension},
and let $d$ be an invariant distance on $\G$.
Then there is a constant $c(d)>0$ such that
\begin{equation*}
\Haus^s_d(B_d(x,r))=c(d)r^s
\end{equation*}
$x\in\G$, and $0<r<\infty$; see e.g.\ \cite[Remark 1.4.5]{BLU_Carnot_groups} (and \cref{rmk:ADR_and_Haus}).
In particular, if $d$ is chosen as in \cref{prop:Carnot_have_shortcuts_prelim},
then $(\G,d,\Haus^s_d)$ is an $s$-ADR metric measure space with \shortcuts\ (by \cref{lemma:LDLR_implies_shortcuts}).
Since it is also a $1$-PI space (see e.g.\ \cite[Theorem 3.25]{serra_cassano_gmt_in_Carnot}),
by \cite{cheeger_kleiner_PI_RNP_LDS} we deduce that it is a $Y$-LDS for any Banach space $Y$ with RNP.
From Pansu's differentiation theorem with $\R$ as target (see \cite{Pansu_diff_thm} or \cite[Theorem 10.3.2]{ledonne_metric_lie_groups_arxiv}),
we see that $(\G,d,\Haus^s_d)$ is an LDS with single chart $(\G,\varphi\colon\G\to\R^{m_1})$,
$\varphi(x)=x^{(1)}$, for $x=(x^{(1)},\dots,x^{(k)})\in\G$.
In particular, the Cheeger differentiable structure is compatible with the \shortcuts.
\par
Hence, we have the following.
\begin{corol}\label{corol:summary_application_Carnot}
Let $\G=\R^{m_1}\times\cdots\times\R^{m_k}$ be a Carnot group of step $k\geq 2$,
homogeneous dimension $s$, and
let $d$ be an invariant distance as in \cref{prop:Carnot_have_shortcuts_prelim}. \par
Then $(\G,d,\Haus_d^s)$ is an $s$-ADR $1$-PI space with compatible \shortcuts.
Hence, for $\eta=(\eta_i)\subseteq (0,1]$, the following hold:
\begin{itemize}
\item $(\G,d_\eta,\Haus^s_d)$ is PI rectifiable if there is $I\subseteq\N$ such that
	$\sum_{i\in I}\eta_i^s<\infty$ and $\inf\{\eta_i\colon i\notin I\}>0$,
	and purely PI unrectifiable otherwise;
\item for any non-zero Banach space $Y$ with RNP, $(\G,d_\eta,\Haus^s_d)$ is a $Y$-LDS
	if and only if $\Haus^s_d(\Badf(f))=0$ for every Lipschitz map $f\colon (\G,d_\eta)\to Y$;
\item for any Banach space $Y$ with RNP and Lipschitz map $f\colon (\G,d_\eta)\to Y$,
	no positive-measure $\Haus^s_d$-measurable subset of $\Badf(f)$ is a $Y$-LDS.
\end{itemize}
\end{corol}
\begin{proof}
The first item follows from \cref{thm:PI_rectifiability} and \cref{thm:pure_PI_unrectifiability},
while the second and third from (\cref{lemma:doubling_meas_and_porosity} and) \cref{prop:char_LDS_squashed} and \cref{prop:LDS_unrect_subsets_Bad}, respectively.
\end{proof}
We do not know if $\Haus^s_d(\Bad(f))=0$ for 
Lipschitz $f\colon (\G,d_\eta)\to \R$
(let alone for $f\colon (\G,d_\eta)\to Y$ and $Y\neq 0$ Banach space with RNP).
This question is related to the `Vertical vs Horizontal' Poincar\'e inequalities in Carnot groups,
see \cite{Austin_Naor_Tessera_sharp_quantitative_non_emb_Heis,Seung_Yeon_Ryoo_quantitative_non_embeddability_groups},
but does not seem to be directly implied by such results. \par
Nonetheless, we have the following application.
\begin{corol}\label{corol:Semmes_conj}
Let $\G=\R^{m_1}\times\cdots\times\R^{m_k}$ be a Carnot group of step $k\geq 2$,
$d_\G$ an invariant distance on $\G$,
and let $s$ denote the
homogeneous dimension of $\G$. \par
Then there is a complete $s$-ADR purely PI unrectifiable metric measure space $(X,d_X,\Haus^s_X)$ and a David-Semmes regular
homeomorphism $f\colon (\G,d_\G)\to (X,d_X)$.
In particular, $f$ is biLipschitz on no positive-measure $\Haus^s_{d_\G}$-measurable set.
\end{corol}
\begin{proof}
Follows from \cref{corol:summary_application_Carnot}, \cref{lemma:some_metric_properties,lemma:same_topology,lemma:doubling_implies_DS_reg},
and the fact that $\G$ is a PI space (or by \cref{prop:biLipschitz_pieces}).
\end{proof}
\cref{corol:Semmes_conj} provides a complete negative answer to \cite[Conjecture 5.2]{semmes_novel_fractal}
and one of its variants \cite[Section 5.2]{semmes_novel_fractal}.
When $\G$ is the first Heisenberg group $\Heis^1$, this is due to \cite{ledonne_li_rajala_shortcuts_heisenberg}.
In the recent \cite{sean_li_schul_rectifiability_via_bilipschitz_pieces}, the authors provide a close-to-optimal
answer to the question of when Lipschitz (or DS-regular) maps have biLipschitz pieces.
In particular, \cite{sean_li_schul_rectifiability_via_bilipschitz_pieces} also resolves the above questions of Semmes.

\bibliographystyle{alpha}
\bibliography{references}
\end{document}